\documentclass{irmaart}

\usepackage{mathrsfs,eucal,enumerate}
\usepackage{epsfig,psfrag}

\renewcommand{\thepart}{\Alph{part}}

\makeatletter 

\def\@part[#1]#2{%
    \ifnum \c@secnumdepth >\m@ne
      \refstepcounter{part}%
      \addcontentsline{toc}{part}{\thepart\hspace{1em}#1}%
    \else
      \addcontentsline{toc}{part}{#1}%
    \fi
    {\parindent \z@ \raggedright
     \interlinepenalty \@M
     \normalfont
     \ifnum \c@secnumdepth >\m@ne%
       \LARGE\bfseries \partname\nobreakspace\thepart : \hspace{.35em}
     \fi
     \LARGE \bfseries #2%
     \markboth{}{}\par}%
    \nobreak
    \vskip 3ex
    \@afterheading}

\makeatother

\numberwithin{equation}{section}

\theoremstyle{plain}
\newtheorem{thm}{Theorem}%[section]
\newtheorem{cor}{Corollary}[section]
\newtheorem{lemma}{Lemma}[section]
\newtheorem{prop}{Proposition}[section]

\theoremstyle{definition}
\newtheorem{definition}{Definition}[section]
\newtheorem{remark}{Remark}[section]

\newcommand{\card}{\operatorname{card}}
\newcommand{\End}{\operatorname{End}}
\newcommand{\Bili}{\operatorname{Bil}_\bA}

\newcommand{\Ker}{\operatorname{Ker}}
\newcommand{\ad}{\operatorname{ad}}

\newcommand{\tRZ}{\widetilde{\boldsymbol{R}}_\Z}
\newcommand{\hRZ}{\widehat{\boldsymbol{R}}_\Z}

\newcommand{\RsimpZ}{\operatorname{RES^{\mathrm{simp}}_\Z}}
\newcommand{\tRsimpZ}{\raisebox{0ex}[1.9ex]{$\wt{\mathrm{RES}}$}^{\mathrm{simp}}_\Z}
\newcommand{\tRsimp}{\raisebox{0ex}[1.9ex]{$\wt{\mathrm{RES}}$}^{\mathrm{simp}}}
\newcommand{\hRsimpZ}{\raisebox{0ex}[1.9ex]{$\wh{\mathrm{RES}}$}^{\mathrm{simp}}_\Z}
\newcommand{\hRsimp}{\raisebox{0ex}[1.9ex]{$\wh{\mathrm{RES}}$}^{\mathrm{simp}}}

\newcommand{\bgRN}{{\raisebox{0ex}[2.6ex]{$\overset{\raisebox{-.8ex}{$\scriptscriptstyle\bul$}}{\gR}_{\rho,N}$}}}
\newcommand{\gRN}{{{\gR}_{\rho,N}}}

\newcommand{\wH}{\wh H}
\newcommand{\wHR}[1]{\wH\big(\gR(#1)\big)}

\newcounter{parag}[section]
\renewcommand{\theparag}{{\bf\thesection.\arabic{parag}}}
\newcommand{\parag}{\medskip \addtocounter{parag}{1} \noindent{\theparag\ } }

\newcommand{\dst}{\displaystyle}
\newcommand{\tst}{\textstyle}

\newcommand{\ens}{\enspace}

\newcommand{\al}{\alpha}
\newcommand{\alb}{\text{\boldmath{$\alpha$}}}
\newcommand{\be}{\beta}
\newcommand{\beb}{\text{\boldmath{$\beta$}}}
\newcommand{\ga}{\gamma}
\newcommand{\gab}{\text{\boldmath{$\gamma$}}}
\newcommand{\Ga}{{\Gamma}}
\newcommand{\de}{\delta}
\newcommand{\deb}{\text{\boldmath{$\delta$}}}

\newcommand{\De}{\Delta}
\newcommand{\bDe}{\text{\boldmath{$\De$}}}
\newcommand{\eps}{{\varepsilon}}
\newcommand{\etab}{\text{\boldmath{$\eta$}}}

\newcommand{\la}{\lambda}
\newcommand{\La}{\Lambda}
\newcommand{\om}{\omega}
\newcommand{\omb}{\text{\boldmath{$\om$}}}
\newcommand{\mb}{\text{\boldmath{$m$}}}
\newcommand{\nb}{\text{\boldmath{$n$}}}
\newcommand{\Om}{\Omega}
\newcommand{\vpi}{\varpi}
\newcommand{\vpib}{\text{\boldmath{$\vpi$}}}
\newcommand{\sig}{{\sigma}}
\newcommand{\Sig}{{\Sigma}}

\renewcommand{\th}{{\theta}}
\newcommand{\Th}{\Theta}
\newcommand{\ph}{\varphi}

\newcommand{\ze}{{\zeta}}

\newcommand{\demi}{\frac{1}{2}}

\newcommand{\lan}{\langle}
\newcommand{\ran}{\rangle}
\newcommand{\ao}{\{\,}          % accolade ouvrante
\newcommand{\af}{\,\}}          % accolade fermante

\newcommand{\Dist}{\operatorname{dist}}
\newcommand{\dist}{\operatorname{d}_{\val}}

\newcommand{\ID}{\operatorname{Id}}

\newcommand{\Sing}{\operatorname{sing}}
\newcommand{\cont}{\operatorname{cont}}

\newcommand{\conc}{\raisebox{.15ex}{$\centerdot$}}
\newcommand{\concsm}{\raisebox{.15ex}{$\centerdot\!\centerdot\!\centerdot$}\,}
\newcommand{\bul}{\bullet}
\newcommand{\bbul}{{\bul\bul}}
\newcommand{\cbul}{{\bul,\bul}}
\newcommand{\ccbul}{{\bul,\bul,\bul}}

\newcommand{\pa}{\partial}
\newcommand{\na}{\nabla}
\newcommand{\ii}{^{-1}}
\newcommand{\iim}{^{\times(-1)}}
\newcommand{\iic}{^{\circ(-1)}}
\newcommand{\ti}{\tilde}
\newcommand{\est}{\emptyset}

\newcommand{\RE}{\mathop{\Re e}\nolimits}

\newcommand{\tphan}{\wt\ph^{\text{ana}}}
\newcommand{\phan}{\ph^{\text{ana}}}
\newcommand{\psian}{\psi^{\text{ana}}}
\newcommand{\Yan}{\wt Y^{\text{ana}}}

\newcommand{\ceil}[1]{\lceil #1 \rceil}
\newcommand{\flo}[1]{\lfloor #1 \rfloor}
\newcommand{\sh}[3]{\operatorname{sh}\begin{pmatrix}#1,\, #2\\#3\end{pmatrix}}
\newcommand{\tsh}[3]{\operatorname{sh}{\big( \begin{smallmatrix}#1,\,
#2\\#3\end{smallmatrix} \big)}}
\newcommand{\tssh}{\operatorname{sh}{\big( \begin{smallmatrix}\alb,\,
\beb,\,\gab\\ \omb\end{smallmatrix} \big)}}
\newcommand{\tcsh}[3]{\operatorname{ctsh}{\big( \begin{smallmatrix}#1,\,
#2\\#3\end{smallmatrix} \big)}}
\newcommand{\val}{\operatorname{val}}
\newcommand{\vl}[1]{\val\left(#1\right)}
\newcommand{\valn}{\operatorname{val}_\nu}
\newcommand{\vln}[1]{\valn\left(#1\right)}
\newcommand{\ord}{\operatorname{ord}}
\newcommand{\od}[1]{\ord\left(#1\right)}

\newcommand{\taul}{\tau_{\ell}}
\newcommand{\taur}{\tau_{r}}

\newcommand{\ie}{{i.e.}}

\newcommand{\cf}{{cf.}\ }
\newcommand{\eg}{{e.g.}}

\newcommand{\wrt}{{with respect to}}
\newcommand{\lhs}{{left-hand side}}
\newcommand{\rhs}{{right-hand side}}

\newcommand{\bA}{\mathbf{A}}

\newcommand{\bB}{\mathbf{B}}
\newcommand{\C}{\mathbb{C}}
\newcommand{\bC}{\mathbf{C}}

\newcommand{\bJ}{\mathbf{J}}
\newcommand{\bM}{\mathbf{M}}
\newcommand{\N}{\mathbb{N}}

\newcommand{\bP}{\mathbf{P}}

\newcommand{\Q}{\mathbb{Q}}
\newcommand{\bQ}{\mathbf{Q}}
\newcommand{\R}{\mathbb{R}}

\newcommand{\bU}{\mathbf{U}}
\newcommand{\Z}{\mathbb{Z}}

\newcommand{\cA}{\mathcal{A}}
\newcommand{\cB}{\mathcal{B}}

\newcommand{\cD}{\mathcal{D}}

\newcommand{\cL}{\mathcal{L}}

\newcommand{\cN}{\mathcal{N}}
\newcommand{\cR}{\mathcal{R}}
\newcommand{\cS}{\mathcal{S}}

\newcommand{\cU}{\mathcal{U}}
\newcommand{\cV}{\mathcal{V}}
\newcommand{\cW}{\mathcal{W}}

\newcommand{\thl}{\wt\th^{{low}}}

\newcommand{\thu}{\wt\th^{{up}}}

\newcommand{\pL}[1]{L_{a,+}^{#1}}
\newcommand{\mL}[1]{L_{a,-}^{#1}}
\newcommand{\pmL}[1]{L_{a,\pm}^{#1}}

\newcommand{\pV}[1]{V_{a,+}^{#1}}
\newcommand{\mV}[1]{V_{a,-}^{#1}}
\newcommand{\pmV}[1]{V_{a,\pm}^{#1}}

\newcommand{\ombp}{\text{$\omb$\hspace{-.115em}'}}

\newcommand{\tS}{{S\hspace{-.65em}\raisebox{.35ex}{--}\hspace{0.07em}}}
\newcommand{\tcS}{{\cS\hspace{-.65em}\raisebox{.35ex}{--}\hspace{0.07em}}}

\newcommand{\cVt}{{\cV\hspace{-.45em}\raisebox{.35ex}{--}}}
\newcommand{\Vt}{{V\hspace{-.5em}\raisebox{.35ex}{-}_{\hspace{-.1em}a}\hspace{-0.22em}}}
\newcommand{\tVa}{{\wt V\hspace{-.45em}\raisebox{.35ex}{-}_{\hspace{-.1em}a}^{\hspace{.1em}\bul}(m)}}

\newcommand{\tcVtb}{{\wt\cV\hspace{-.45em}\raisebox{.35ex}{--}^\bul}}

\newcommand{\tcVto}{{\wt\cV\hspace{-.45em}\raisebox{.35ex}{--}^\omb}}

\newcommand{\tcVod}{{\wt\cV\hspace{-.45em}\raisebox{.35ex}{--}^{\omb^2}}}

\newcommand{\hcVto}{{\wh\cV\hspace{-.45em}\raisebox{.35ex}{--}^\omb}}

\newcommand{\tcVab}{{\wt\cV\hspace{-.45em}\raisebox{.35ex}{--}_a^\bul}}
\newcommand{\tcVae}{{\wt\cV\hspace{-.45em}\raisebox{.35ex}{--}_a^\est}}
\newcommand{\tcVao}{{\wt\cV\hspace{-.45em}\raisebox{.35ex}{--}_a^\omb}}
\newcommand{\tcVaop}{{\wt\cV\hspace{-.45em}\raisebox{.35ex}{--}_a^{\ombp}}}
\newcommand{\tcVaor}{{\wt\cV\hspace{-.45em}\raisebox{.35ex}{--}_a^{\om_1,\dotsc,\om_r}}}

\newcommand{\hcVab}{{\wh\cV\hspace{-.45em}\raisebox{.35ex}{--}_a^\bul}}
\newcommand{\hcVae}{{\wh\cV\hspace{-.45em}\raisebox{.35ex}{--}_a^\est}}
\newcommand{\hcVao}{{\wh\cV\hspace{-.45em}\raisebox{.35ex}{--}_a^\omb}}
\newcommand{\hcVaop}{{\wh\cV\hspace{-.45em}\raisebox{.35ex}{--}_a^{\ombp}}}
\newcommand{\hcVaun}{{\wh\cV\hspace{-.45em}\raisebox{.35ex}{--}_a^{\om_1}}}
\newcommand{\hcVai}{{\wh\cV\hspace{-.45em}\raisebox{.35ex}{--}_a^{\om_{i+1},\dotsc,\om_r}}}

\newcommand{\zDu}{\cD^{up}(R,\eps)}
\newcommand{\zDl}{\cD^{low}(R,\eps)}

\newcommand{\zDm}{\cD_-(R,\eps)}
\newcommand{\zDp}{\cD_+(R,\eps)}
\newcommand{\zDpm}{\cD_\pm(R,\eps)}

\newcommand{\dDem}{
  \raisebox{.23ex}{
  {$\stackrel{\raisebox{-.23ex}{$\scriptscriptstyle\bullet$}}\Delta_{
  \raisebox{-.23ex}{$\scriptstyle m$}}$}}
}

\newcommand{\dDep}{
  \raisebox{.23ex}{
  {$\stackrel{\raisebox{-.23ex}{$\scriptscriptstyle\bullet$}}\Delta_{
  \raisebox{-.23ex}{$\scriptstyle m$}}^{
  \raisebox{-.85ex}{$\scriptstyle +$}}$}}
}

\newcommand{\I}{{\mathrm i}}
\newcommand{\dd}{{\mathrm d}}

\newcommand{\ee}{{\mathrm e}}

\newcommand{\wh}{\widehat}
\newcommand{\wt}{\widetilde}

\newcommand{\htb}[1]{\raisebox{-.08ex}{${\stackrel{
            \raisebox{-.18ex}{$\scriptscriptstyle\wedge$}
          }{#1}
     }$}}
\newcommand{\shtb}[1]{
        \raisebox{0ex}[1.6ex][0ex]{${\stackrel{
            \raisebox{-.23ex}{$\scriptscriptstyle\wedge$}
          }{\scriptstyle#1}
     }$}}

\newcommand{\wc}[1]{\raisebox{-.08ex}{${\stackrel{
            \raisebox{-.23ex}{$\scriptscriptstyle\vee$}
          }{#1}
     }$}}

\newcommand{\swc}[1]{
        \raisebox{0ex}[1.6ex][0ex]{${\stackrel{
            \raisebox{-.23ex}{$\scriptscriptstyle\vee$}
          }{\scriptstyle#1}
     }$}}

\newcommand{\norm}[1]{\Vert#1\Vert}

\newcommand{\gA}{\mathscr A}
\newcommand{\gB}{\mathscr B}
\newcommand{\gC}{\mathscr C}
\newcommand{\gD}{\mathscr D}
\newcommand{\gE}{\mathscr E}
\newcommand{\gF}{\mathscr F}
\newcommand{\gH}{\mathscr H}
\newcommand{\gK}{\mathscr K}
\newcommand{\gL}{\mathfrak L}
\newcommand{\gM}{\mathfrak M}
\newcommand{\hM}{\mathscr M}
\newcommand{\gO}{\mathscr O}
\newcommand{\gP}{\mathscr P}
\newcommand{\gR}{\mathscr R}

\newcommand{\gT}{\mathscr T}
\newcommand{\rmT}{\mathrm{T}}
\newcommand{\gX}{\mathscr X}
\newcommand{\gY}{\mathscr Y}

\newcommand{\hMst}{\hM^\bul(\Om^*,\bA)}
\newcommand{\hMstC}{\hM^\bul(\Om^*,\C)}
\newcommand{\hMstR}{\hM^\bul(\Om^*,\tRsimpZ)}
\newcommand{\hMlic}{\hM_{\textit{lic}}^\bul(\Om^*,\bA)}

\newcommand{\Xl}{X^{\text{lin}}}
\newcommand{\fl}{f^{\text{lin}}}
\newcommand{\Fl}{F^{\text{lin}}}
\newcommand{\ZE}{\operatorname{Ze}}

\begin{document}

%%%%%%%%%%%%%%%%%%%%%%%%%%%%%%%%%%%%%%%%%%%%%%%%%%%%%%%%%%%%%%%%%%%%%%%%%
%%%%%%%%%%%%%%%%%%%%%%%%%%%%%%%%%%%%%%%%%%%%%%%%%%%%%%%%%%%%%%%%%%%%%%%%%

\address{
Institut de m\'ecanique c\'eleste, CNRS\\ 
77 av.\ Denfert-Rochereau, 75014 Paris, France\\
email:\,\tt{sauzin@imcce.fr}
}

%%%%%%%%%%%%%%%%%%%%%%%%%%%%%%%%%%%%%%%%%%%%%%%%%%%%%%%%%%%%%%%%%%%%%%%%%
%%%%%%%%%%%%%%%%%%%%%%%%%%%%%%%%%%%%%%%%%%%%%%%%%%%%%%%%%%%%%%%%%%%%%%%%%

\title{Mould expansions for the saddle-node and resurgence monomials}

\author{David Sauzin}

%%%%%%%%%%%%%%%%%%%%%%%%%%%%%%%%%%%%%%%%%%%%%%%%%%%%%%%%%%%%%%%%%%%%%%%%%
%%%%%%%%%%%%%%%%%%%%%%%%%%%%%%%%%%%%%%%%%%%%%%%%%%%%%%%%%%%%%%%%%%%%%%%%%

\maketitle

\vspace{-.75cm}

\begin{abstract}
This article is an introduction to some aspects of \'Ecalle's mould calculus, a
powerful combinatorial tool which yields surprisingly explicit formulas for the
normalising series attached to an analytic germ of singular vector field or of
map.
This is illustrated on the case of the saddle-node, a two-dimensional vector
field which is formally conjugate to Euler's vector field $x^2\frac{\pa}{\pa
x}+(x+y)\frac{\pa}{\pa y}$,
and for which the formal normalisation is shown to be resurgent in~$1/x$.
Resurgence monomials adapted to alien calculus are also described as another
application of mould calculus.
\end{abstract}

\vspace{-.75cm}

\tableofcontents

\vfill\eject

%%%%%%%%%%%%%%%%%%%%%%%%%%%%%%%%%%%%%%%%%%%%%%%%%%%%%%%%%%%%%%%%%%%%%%%%%
%%%%%%%%%%%%%%%%%%%%%%%%%%%%%%%%%%%%%%%%%%%%%%%%%%%%%%%%%%%%%%%%%%%%%%%%%

\section{Introduction}

%%%%%%%%%%%%%%%%%%%%%%%%%%%%%%%%%%%%%%%%%%%%%%%%%%%%%%%%%%%%%%%%%%%%%%%%%
%%%%%%%%%%%%%%%%%%%%%%%%%%%%%%%%%%%%%%%%%%%%%%%%%%%%%%%%%%%%%%%%%%%%%%%%%

Mould calculus was developed by J.~\'Ecalle in relation with his Resurgence
theory almost thirty years ago (\cite{Eca81}, \cite{dulac}, \cite{Eca93}). 
%
% the recent \cite{Eca05}
%
The primary goal of this text is to give an introduction to mould calculus,
together with an exposition of the way it can be applied to a specific geometric
problem pertaining to the theory of dynamical systems: the analytic
classification of saddle-node singularities.

The treatment of this example was indicated in \cite{Eca84} in concise manner
(see also \cite{CNP}), but I found it useful to provide a self-contained presentation of
mould calculus and detailed explanations for the saddle-node problem,
in the same spirit as Resurgence theory and alien calculus were presented in
\cite{kokyu} together with the example of the analytic classification of
tangent-to-identity transformations in complex dimension~$1$.

Basic facts from Resurgence theory are also recalled in the course of
the exposition, with the hope that this text will serve to a broad readership.
I also included a section on the relation between the resurgent approach to the
saddle-node problem and Martinet-Ramis's work \cite{MR}.

The text consists of three parts.
\begin{enumerate}[A.]
\item
Section~\ref{secSN} describes the problem of the normalisation of the
saddle-node and Section~\ref{secMCexpSN} outlines its treatment by the method of
mould-comould expansions.
\item
The second part has an ``algebraic'' flavour: it is devoted to a systematic
exposition of some features of mould algebras (Sections~\ref{secAlgMoulds}
and~\ref{secAltSym})
and mould-comould expansions (Sections~\ref{secGenMcM}
and~\ref{secContrAltSym}).
\item
The third part is mainly concerned by the applications
to Resurgence theory of the previous results
(Sections~\ref{secResurSN}--\ref{secMR} show the consequences for the problem
of the saddle-node and have an ``analytic'' flavour,
Section~\ref{secResurMonom} describes the construction of resurgence monomials
which allow one to check the freeness of the algebra of alien derivations);
other applications are also briefly alluded to in Section~\ref{secOtherAppli}
(with a few words about arborification and multizetas).
\end{enumerate}

All the ideas come from J.~\'Ecalle's articles and lectures. 
An effort has been made to provide full details, which occasionally may have
resulted in original definitions, but they must be considered as auxiliary \wrt\
the overall theory.
The details of the resurgence proofs which are given in
Sections~\ref{secResurSN} and~\ref{secBE} are original, at least I did not see
them in the literature previously.

\vfill

\pagebreak

%%%%%%%%%%%%%%%%%%%%%%%%%%%%%%%%%%%%%%%%%%%%%%%%%%%%%%%%%%%%%%%%%%%%%%%%%
%%%%%%%%%%%%%%%%%%%%%%%%%%%%%%%%%%%%%%%%%%%%%%%%%%%%%%%%%%%%%%%%%%%%%%%%%
%%%%%%%%%%%%%%%%%%%%%%%%%%%%%%%%%%%%%%%%%%%%%%%%%%%%%%%%%%%%%%%%%%%%%%%%%
%%%%%%%%%%%%%%%%%%%%%%%%%%%%%%%%%%%%%%%%%%%%%%%%%%%%%%%%%%%%%%%%%%%%%%%%%

\part{The saddle-node problem}

%%%%%%%%%%%%%%%%%%%%%%%%%%%%%%%%%%%%%%%%%%%%%%%%%%%%%%%%%%%%%%%%%%%%%%%%%
%%%%%%%%%%%%%%%%%%%%%%%%%%%%%%%%%%%%%%%%%%%%%%%%%%%%%%%%%%%%%%%%%%%%%%%%%

\section{The saddle-node and its formal normalisation}	\label{secSN}

%%%%%%%%%%%%%%%%%%%%%%%%%%%%%%%%%%%%%%%%%%%%%%%%%%%%%%%%%%%%%%%%%%%%%%%%%
%%%%%%%%%%%%%%%%%%%%%%%%%%%%%%%%%%%%%%%%%%%%%%%%%%%%%%%%%%%%%%%%%%%%%%%%%

\parag
Let us consider a germ of complex analytic $2$-dimensional vector field
\begin{equation}	\label{eqdefX}
X = x^2 \frac{\pa\,}{\pa x} + A(x,y) \frac{\pa\,}{\pa y},
\qquad A \in \C\{x,y\},
\end{equation}
for which we assume 
\begin{equation}	\label{eqassA}
A(0,y) = y, \qquad \frac{\pa^2 A}{\pa x\pa y}(0,0) = 0.
\end{equation}
Assumption~\eqref{eqassA} ensures that $X$ is formally conjugate to the normal
form
\begin{equation}
X_0 = x^2 \frac{\pa\,}{\pa x} + y \frac{\pa\,}{\pa y}.
\end{equation}
We shall be interested in the formal tranformations which conjugate~$X$ and~$X_0$.

%%%%%%%%%%%%%%%%%%%%%%%%%%%%%%%%%%%%%%%%%%%%%%%%%%%%%%%%%%%%%%%%%%%%%%%%%
%%%%%%%%%%%%%%%%%%%%%%%%%%%%%%%%%%%%%%%%%%%%%%%%%%%%%%%%%%%%%%%%%%%%%%%%%

\parag
This is the simplest case from the point of view of {\em formal} classification of
saddle-node singularities of analytic differential equations.\footnote{
A singular differential equation is essentially the same thing as a differential
$1$-form which vanish at the origin. It defines a singular foliation, the leaves
of which can also be obtained by integrating a singular vector field, but
classifying singular foliations (or singular differential equations) is
equivalent to classifying singular vector fields \emph{up to time-change}.
See \eg\ \cite{Moussu}.} 
Indeed, when a differential equation $B(x,y)\,\dd y - A(x,y)\,\dd x = 0$ is singular
at the origin ($A(0,0)=B(0,0)=0$) and its $1$-jet has eigenvalues $0$ and~$1$,
it is always formally conjugate to one of the normal forms
$x^{p+1}\,\dd y - (1+\la x)y\,\dd x = 0$ ($p\in\N^*$, $\la\in\C$)
or $y\,\dd x=0$.
%
%\marge{Mentionner L.~Teyssier pour la classification des champs?}
%
What we call saddle-node singularity corresponds to the first case, and the
normal form~$X_0$ corresponds to $p=1$ and $\la=0$.

Moreover, a saddle-node singularity can always be analytically reduced to the
form $x^{p+1}\,\dd y - A(x,y)\,\dd x=0$ with $A(0,y)=y$ 
(this result goes back to Dulac---see \cite{MR}, \cite{Moussu}), 
it is thus legitimate to
consider vector fields of the form~\eqref{eqdefX}, which generate the same
foliations
(we restrict ourselves to $(p,\la)=(1,0)$ for the sake of simplicity).

The problem of the {\em analytic} classification of saddle-node singularities was
solved in~\cite{MR}. The resurgent approach to this problem is indicated
in~\cite{Eca84} and~\cite[Vol.~3]{Eca81} (see also~\cite{CNP}).
The resurgent approach consists in analysing the divergence of the normalising
transformation through alien calculus.

%%%%%%%%%%%%%%%%%%%%%%%%%%%%%%%%%%%%%%%%%%%%%%%%%%%%%%%%%%%%%%%%%%%%%%%%%
%%%%%%%%%%%%%%%%%%%%%%%%%%%%%%%%%%%%%%%%%%%%%%%%%%%%%%%%%%%%%%%%%%%%%%%%%

\parag
Normalising transformation means a formal diffeomorphism $\th$ solution of the
conjugacy equation 
\begin{equation}	\label{eqconjugeq}
X = \th^* X_0.
\end{equation}
Due to the shape of~$X$, one can find a
unique formal solution of the form
\begin{equation}	\label{eqdefthph}
\th(x,y) = \big(x, \ph(x,y) \big), \qquad 
\ph(x,y) = y + \sum_{n\ge0} \ph_n(x) y^n, \qquad 
\ph_n(x) \in x\C[[x]].
\end{equation}
The first step in the resurgent approach consists in proving that the formal
series~$\ph_n$ are resurgent \wrt\ the variable $z=-1/x$. 
We shall prove this fact by using \'Ecalle's mould calculus (see
Theorem~\ref{thmResur} in Section~\ref{secResur} below).

The Euler series $\ph_0(x) = - \sum_{n\ge1} (n-1)! x^n$ appears in the case
$A(x,y)=x+y$, for which the solution of the conjugacy equation is simply 
$\th(x,y) = \big(x, y + \ph_0(x)\big)$.

%%%%%%%%%%%%%%%%%%%%%%%%%%%%%%%%%%%%%%%%%%%%%%%%%%%%%%%%%%%%%%%%%%%%%%%%%
%%%%%%%%%%%%%%%%%%%%%%%%%%%%%%%%%%%%%%%%%%%%%%%%%%%%%%%%%%%%%%%%%%%%%%%%%

\parag
Observe that $\th(x,y) = \left(x, y + \sum_{n\ge0} \ph_n(x) y^n\right)$ is solution of the conjugacy
equation if and only if 
\begin{equation}	\label{eqdefYzu}
\wt Y(z,u) = u \,  \ee^z + \sum_{n\ge0} u^n \ee^{nz} \wt\ph_n(z), \qquad 
\wt\ph_n(z) = \ph_n(-1/z) \in z\ii\C[[z\ii]],
\end{equation}
is solution of the differential equation 
\begin{equation}	\label{eqdiffeqX}
\pa_z \wt Y = A(-1/z,\wt Y)
\end{equation}
associated with the vector field~$X$.
(Indeed, the first component of the flow of~$X$ is trivial and the
second component is determined by solving~\eqref{eqdiffeqX}; on the other hand,
the flow of~$X_0$ is trivial and, by plugging it into~$\th$, one obtains the
flow of~$X$.)

The formal expansion~$\wt Y(z,u)$ is called {\em formal integral}
of the differential equation~\eqref{eqdiffeqX}.
One can obtain its components~$\wt\ph_n(z)$ 
(and, consequently, the formal series~$\ph_n(x)$ themselves)
as solutions of ordinary differential
equations, by expanding~\eqref{eqdiffeqX} in powers of~$u$:
\begin{gather}
\label{eqdiffeqzero}
\frac{\dd \wt\ph_0}{\dd z\,} = A\big(-1/z,\wt\ph_0(z)\big),\\
\label{eqdiffeqn}
\frac{\dd \wt\ph_n}{\dd z\,} + n\wt\ph_n(z) = 
\pa_y A\big(-1/z,\wt\ph_0(z)\big) \wt\ph_n(z) + \wt\chi_n(z),
\end{gather}
with $\wt\chi_n$ inductively determined by $\wt\ph_0,\dotsc,\wt\ph_{n-1}$.
Only the first equation is non-linear. One can prove the resurgence of the
$\wt\ph_n$'s by exploiting their property of being the unique solutions in
$z\ii\C[[z\ii]]$ of these equations and devising a perturbative scheme to solve
the first one,\footnote{
See Section~2.1 of~\cite{kokyu} for an illustration of this method on a
non-linear difference equation.
}
but mould calculus is quite a different approach.

%%%%%%%%%%%%%%%%%%%%%%%%%%%%%%%%%%%%%%%%%%%%%%%%%%%%%%%%%%%%%%%%%%%%%%%%%
%%%%%%%%%%%%%%%%%%%%%%%%%%%%%%%%%%%%%%%%%%%%%%%%%%%%%%%%%%%%%%%%%%%%%%%%%

\section{Mould-comould expansions for the saddle-node}

\label{secMCexpSN}

%%%%%%%%%%%%%%%%%%%%%%%%%%%%%%%%%%%%%%%%%%%%%%%%%%%%%%%%%%%%%%%%%%%%%%%%%
%%%%%%%%%%%%%%%%%%%%%%%%%%%%%%%%%%%%%%%%%%%%%%%%%%%%%%%%%%%%%%%%%%%%%%%%%

\parag
The analytic vector fields~$X$ and~$X_0$ can be viewed as derivations of the
algebra $\C\{x,y\}$,
but since we are interested in formal conjugacy, we now consider them as
derivations of $\C[[x,y]]$. We shall first rephrase our problem as a
problem about operators of this algebra.\footnote{
Our algebras will be, unless otherwise specified, associative unital algebras
over~$\C$ (possibly non-commutative).
In this article, operator means endomorphism of the underlying vector space;
thus an operator of $\gA=\C[[x,y]]$ is an element of $\End_\C(\gA)$.
The space $\End_\C(\gA)$ has natural structures of ring and of $\gA$-module,
which are compatible in the sense that $f\in\gA \mapsto f \ID\in \End_\C(\gA)$
is a ring homomorphism.
}

The commutative algebra $\gA=\C[[x,y]]$ is also a local ring; as such, it is
endowed with a metrizable topology, in which the powers of the maximal ideal
$\gM=\ao f\in\C[[x,y]]\mid f(0,0)=0\af$ form a system of neighbourhoods of~$0$,
which we call Krull topology or topology of the formal convergence and which is
complete (as a uniform structure).\footnote{
This is also called the $\gM$-adic topology, or the $(x,y)$-adic topology. 
Beware that $\C[[x,y]]$ is a topological algebra only if we put the discrete
topology on~$\C$.
}

\begin{lemma}
The set of all continuous algebra homomorphisms of $\C[[x,y]]$ coincides with the
set of all substitution operators, \ie\ operators of the form $f \mapsto
f\circ\th$ with $\th\in\gM\times\gM$.
\end{lemma}

\begin{proof}

Any substitution operator is clearly a continuous algebra homomorphism of
$\C[[x,y]]$.
Conversely, let $\Th$ be a continuous algebra homomorphism.
The idea is that $\Th$ will be determined by its action on the two
generators of the maximal ideal, and setting $\th = (\Th x,\Th y)$ we can
identify~$\Th$ with the substitution operator $f \mapsto f\circ\th$.
We just need to check that $\Th x$ and $\Th y$ both belong to the maximal ideal,
which is the case because, by continuity, $(\Th x)^n = \Th(x^n)$ and $(\Th y)^n
= \Th(y^n)$ must tend to~$0$ as $n\to\infty$;
one can then write any~$f$ as a convergent---for the Krull topology---series of
monomials $\sum f_{m,n} x^m y^n$ and its image as the formally convergent series
$\Th f = \sum  \Th(f_{m,n} x^m y^n) = \sum f_{m,n} (\Th x)^m (\Th y)^n$.
\end{proof}

A formal invertible transformation thus amounts to a continuous automorphism of
$\C[[x,y]]$.
Since the conjugacy equation~\eqref{eqconjugeq} can be written 
$$
X f = \big[ X_0(f\circ\th) \big] \circ \th\ii, 
\qquad f\in\C[[x,y]],
$$
if we work at the level of the substitution operator,
we are left with the problem of finding a continuous automorphism~$\Th$ of
$\C[[x,y]]$ such that $\Th (Xf) = X_0(\Th f)$ for all $f$, \ie
\begin{equation}	\label{eqconjugOp}
\Th X = X_0 \Th.
\end{equation}

%%%%%%%%%%%%%%%%%%%%%%%%%%%%%%%%%%%%%%%%%%%%%%%%%%%%%%%%%%%%%%%%%%%%%%%%%
%%%%%%%%%%%%%%%%%%%%%%%%%%%%%%%%%%%%%%%%%%%%%%%%%%%%%%%%%%%%%%%%%%%%%%%%%

\parag
The idea is to construct a solution to~\eqref{eqconjugOp} from the
``building blocks'' of~$X$.
Let us use the Taylor expansion
\begin{equation}	\label{eqdefiancN}
A(x,y) = y + \sum_{n\in\cN} a_n(x) y^{n+1}, \qquad
\cN = \ao n\in\Z \mid n \ge -1 \af
\end{equation}
to write
\begin{gather}	
	\label{eqdefBn}
X = X_0 + \sum_{n\in\cN} a_n(x) B_n, \qquad
B_n = y^{n+1} \frac{\pa\,}{\pa y}, \\
	\label{eqassan}
a_n(x) \in x\C\{x\}, \qquad a_0(x) \in x^2\C\{x\}
\end{gather}
(thus incorporating the information from~\eqref{eqassA}).
The series in~\eqref{eqdefBn} must be interpreted as a simply convergent
series of operators of~$\C[[x,y]]$ (the series $\sum a_n B_n f$ is formally
convergent for any $f\in\C[[x,y]]$).

Let us introduce the differential operators
\begin{equation}	\label{eqdefbBn}
\bB_\est = \ID, \qquad
\bB_{\om_1,\dotsc,\om_r} = B_{\om_r}  \dotsm B_{\om_1}
\end{equation}
for $\om_1,\dotsc,\om_r\in\cN$.
We shall look for an automorphism~$\Th$ solution of~\eqref{eqconjugOp} in the
form 
\begin{equation}	\label{eqexpmouldcomould}
\Th = \sum_{r\ge0} \, \sum_{\om_1,\dotsc,\om_r\in\cN} 
\cV^{\om_1,\dotsc,\om_r}(x) \bB_{\om_1,\dotsc,\om_r},
\end{equation}
with the convention that the only term with $r=0$ is
$\cV^\est\bB_\est$, with $\cV^\est = 1$,
and with coefficients
$\cV^{\om_1,\dotsc,\om_r}(x) \in x\C[[x]]$ 
to be determined from the data $\{a_n,\, n\in\cN\}$
in such a way that 

\begin{enumerate}[(i)]
\item the expression~\eqref{eqexpmouldcomould} is a formally convergent series of
operators of~$\C[[x,y]]$ and defines an operator~$\Th$ which is continuous for
the Krull topology,
\item the operator~$\Th$ is an algebra automorphism,
\item the operator~$\Th$ satisfies the conjugacy equation~\eqref{eqconjugOp}.
\end{enumerate}

In \'Ecalle's terminology, 
the collection of operators $\{ \bB_{\om_1,\dotsc,\om_r} \}$ is a typical
example of {\em comould};
any collection of coefficients $\{ \cV^{\om_1,\dotsc,\om_r}  \}$ is a {\em
mould} (here with values in~$\C[[x]]$, but other algebras may be used);
a formally convergent series of the form~\eqref{eqexpmouldcomould} is a {\em
mould-comould expansion}, often abbreviated as
$$
\Th = \sum \cV^\bul \bB_\bul
$$
(we shall clarify later what ``formally convergent'' means for such multiply-indexed series of
operators).

%%%%%%%%%%%%%%%%%%%%%%%%%%%%%%%%%%%%%%%%%%%%%%%%%%%%%%%%%%%%%%%%%%%%%%%%%
%%%%%%%%%%%%%%%%%%%%%%%%%%%%%%%%%%%%%%%%%%%%%%%%%%%%%%%%%%%%%%%%%%%%%%%%%

\parag
Let us indicate right now the formulas for the problem of the saddle-node~\eqref{eqdefX}:

\begin{lemma}	\label{lemdefcV}
The equations
\begin{align}
&\cV^\est = 1 \nonumber \\
\label{eqdefcV}
(x^2\frac{\dd\,}{\dd x} + \om_1 + \dotsb + \om_r) &\cV^{\om_1,\dotsc,\om_r}
= a_{\om_1} \cV^{\om_2,\dotsc,\om_r}, \qquad
\om_1,\dotsc,\om_r \in \cN
\end{align}
inductively determine a unique collection of formal series
$\cV^{\om_1,\dotsc,\om_r} \in x\C[[x]]$ for $r\ge1$.
Moreover,
\begin{equation}	\label{eqvalcV}
\cV^{\om_1,\dotsc,\om_r} \in x^{\ceil{r/2}}\C[[x]],
\end{equation}
where $\ceil{s}$ denotes, for any $s\in\R$, the least integer not smaller than~$s$.
\end{lemma}

\begin{proof}
Let $\nu$ denote the valuation in $\C[[x]]$: $\nu(\sum c_m x^m) = \min\ao m \mid
c_m\neq0 \af \in\N$ for a non-zero formal series and $\nu(0)=\infty$.

Since $\pa = x^2\frac{\dd\,}{\dd x}$ increases valuation by at least one unit,
$\pa + \mu$ is invertible for any $\mu\in\C^*$ and the inverse operator
\begin{equation}	\label{eqdefinvpamu}
(\pa+\mu)\ii = \sum_{r\ge0} \mu^{-r-1}(-\pa)^r
\end{equation} 
(formally convergent series of operators) leaves $x\C[[x]]$ invariant.
On the other hand, we {\em define} $\pa\ii \colon x^2\C[[x]] \to x\C[[x]]$ by the
formula
$\pa\ii \ph(x) = \int_0^x \big( t^{-2}\ph(t) \big) \,\dd t$,
so that $\psi=\pa\ii\ph$ is the unique solution in $x\C[[x]]$ of the equation
$\pa\psi=\ph$ whenever $\ph\in x^2\C[[x]]$.

For $r=1$, equation~\eqref{eqdefcV} has a unique solution~$\cV^{\om_1}$ in
$x\C[[x]]$, because the \rhs\ is~$a_{\om_1}$, element of $x\C[[x]]$, and even of
$x^2\C[[x]]$ when $\om_1=0$.
By induction, for $r\ge2$, we get a \rhs\ in $x^2\C[[x]]$ and a unique solution
$\cV^\omb$ in $x\C[[x]]$ for $\omb=(\om_1,\dotsc,\om_r)\in\cN^r$.
Moreover, with the notation $`\omb = (\om_2,\dotsc,\om_r)$, we have
$$
\nu(\cV^\omb) \ge \al^\omb + \nu(\cV^{`\omb}), \qquad 
\text{with}\;
\al^\omb = \left| \begin{aligned} 
0 \quad &\text{if $\om_1+\dotsb+\om_r=0$ and $\om_1\neq0$,} \\
1 \quad &\text{if $\om_1+\dotsb+\om_r\neq0$ or $\om_1=0$.} 
\end{aligned}  \right.
$$
Thus $\nu(\cV^\omb) \ge \card \cR^\omb$, with 
$\cR^\omb = \ao i\in[1,r] \mid \om_i+\dotsb+\om_r\neq0 \;\text{or}\; \om_i=0
\af$
for $r\ge1$.

Let us check that $\card\cR^\omb \ge \ceil{r/2}$.
This stems from the fact that if $i\not\in\cR^\omb$, $i\ge2$, then $i-1\in\cR^\omb$
(indeed, in that case $\om_{i-1}+\dotsb+\om_r = \om_{i-1}$),
and that $\cR^\omb$ has at least one element, namely~$r$.
The inequality is thus true for $r=1$ or~$2$; by induction, if $r\ge3$, then
$\cR^\omb \cap [3,r] = \cR^{``\omb}$ with $``\omb = (\om_3,\dotsc,\om_r)$ and
either $2\in \cR^\omb$, or $2\not\in\cR^\omb$ and $1\in\cR^\omb$, 
thus $\card\cR^\omb \ge 1 + \card\cR^{``\omb}$.
\end{proof}

%%%%%%%%%%%%%%%%%%%%%%%%%%%%%%%%%%%%%%%%%%%%%%%%%%%%%%%%%%%%%%%%%%%%%%%%%
%%%%%%%%%%%%%%%%%%%%%%%%%%%%%%%%%%%%%%%%%%%%%%%%%%%%%%%%%%%%%%%%%%%%%%%%%

\parag
To give a definition of formally summable families of operators adapted to our
needs, we shall consider our operators as elements of a topological ring of a
certain kind and make use of the Cauchy criterium for summable families.

\begin{definition}	\label{defpseudoval}
Given a ring~$\gE$ (possibly non-commutative), we call pseudovaluation any map
$\val \colon \gE \to
\Z\cup\{\infty\}$ satisfying, for any $\Th,\Th_1,\Th_2\in\gE$,
\begin{itemize}
\item $\vl{\Th}=\infty$ iff $\Th=0$,
\item $\vl{\Th_1-\Th_2}  \ge \min\big\{ \!\vl{\Th_1}, \vl{\Th_2} \big\}$,
\item $\vl{\Th_1\Th_2} \ge \vl{\Th_1} + \vl{\Th_2}$, 
\end{itemize}
The formula $\dist(\Th_1,\Th_2) = 2^{-\vl{\Th_2-\Th_1}}$ then defines a
distance, for which $\gE$ is a topological ring.
%
% ; the corresponding tolopoly is called ``topology of formal convergence''.
%
We call $(\gE,\val)$ a complete pseudovaluation ring if the distance~$\dist$ is complete.
\end{definition}

We use the word pseudovaluation rather than valuation because
$\gE$ is not assumed to be an integral domain, and we dot impose equality in the
third property.
The distance~$\dist$ is ultrametric, translation-invariant, and it satisfies 
$\dist(0,\Th_1\Th_2)  \le \dist(0,\Th_1) \dist(0,\Th_2)$.

Let us denote by~$1$ the unit of~$\gE$.
Giving a pseudovaluation on~$\gE$ such that $\vl{1}=0$ is equivalent to giving a filtration
$(\gE_\de)_{\de\in\Z}$ that is compatible with its ring structure (\ie\ a sequence if
additive subgroups such that $1\in\gE_0$, $\gE_{\de+1}\subset\gE_\de$ and
$\gE_\de\gE_{\de'}\subset\gE_{\de+\de'}$ for all $\de,\de'\in\Z$),
exhaustive ($\bigcup\gE_\de=\gE$) and separated ($\bigcap\gE_\de=\{0\}$).
Indeed, the order function~$\val$ associated with the filtration, defined by $\vl{\Th} =
\sup\ao\de\in\Z \mid \Th\in\gE_\de \af$, is then a
pseudovaluation; conversely, one can set $\gE_\de = \ao \Th\in\gE \mid
\vl{\Th}\ge\de \af$.

\begin{definition}	\label{defformsumfam}
Let $(\gE,\val)$ be a complete pseudovaluation ring.
Given a set $I$, a family $(\Th_i)_{i\in I}$ in~$\gE$ is said to be formally
summable if, for any $\de\in\Z$, % almost all $\Th_i$ have valuation~$>\de$, \ie\
the set $\ao i\in I \mid \vl{\Th_i} \le \de \af$ is finite
(the support of the family is thus countable, if not~$I$ itself).
\end{definition}

One can then check that, for any exhaustion $(I_k)_{k\in\N}$ by finite sets of
the support of the family, the sequence $\sum_{i\in I_k} \Th_i$ is a Cauchy
sequence for $\dist$, and that the limit does not depend on the chosen
exhaustion; the common limit is then denoted $\sum_{i\in I}\Th_i$.
Observe that there must exist $\de_*\in\Z$ such that $\vl{\Th_i}\ge\de_*$ for
all $i\in I$.

%%%%%%%%%%%%%%%%%%%%%%%%%%%%%%%%%%%%%%%%%%%%%%%%%%%%%%%%%%%%%%%%%%%%%%%%%
%%%%%%%%%%%%%%%%%%%%%%%%%%%%%%%%%%%%%%%%%%%%%%%%%%%%%%%%%%%%%%%%%%%%%%%%%

\parag
We apply this to operators of $\gA = \C[[x,y]]$ as follows.
The Krull topology of~$\gA$ can be defined with the help of the monomial valuation
$$
\nu_4(f) = \min\ao 4m+n \mid f_{m,n}\neq0  \af
\quad \text{for} \; f = \sum f_{m,n} x^m y^n \neq 0,
\qquad \nu_4(0)=\infty.
$$
Indeed, for any sequence $(f_k)_{k\in\N}$ of~$\gA$,
$$
f_k \xrightarrow[k\to\infty]{} 0
\quad\Longleftrightarrow\quad
\sum_{k\in\N} f_k \;\text{formally convergent}
\quad\Longleftrightarrow\quad
\nu_4(f_k)\xrightarrow[k\to\infty]{}\infty.
$$
In particular, $(\C[[x,y]],\nu_4)$ is a complete pseudovaluation ring.

Suppose more generally that $(\gA,\nu)$ is any complete pseudovaluation ring
such that~$\gA$ is also an algebra.
Corresponding to the filtration $\gA_p = \ao f\in\gA \mid \nu(f)\ge p\af$,
$p\in\Z$, there is a filtration of $\End_\C(\gA)$:
$$
\gE_\de = \ao \Th \in \End_\C(\gA) \mid \Th(\gA_p) \subset \gA_{p+\de}
\;\text{for each $p$} \af, \qquad \de\in\Z.
$$

\begin{definition}	\label{defopval}
Let $\de\in\Z$. An element~$\Th$ of~$\gE_\de$ is said to be 
an ``operator of valuation~$\ge\de$''.
We then define $\vln\Th \in \Z\cup\{\infty\}$, the ``valuation of~$\Th$'',
as the largest $\de_0$ such that $\Th$ has
valuation~$\ge\de_0$; this number is infinite only for $\Th=0$. %in the case of $\Th=0$.
\end{definition}

Denote by $\gE$ the union $\bigcup\gE_\de$ over all $\de\in\Z$: 
these are the operators of~$\gA$ ``having a valuation'' (\wrt~$\nu$), \ie
$$
\gE = \ao \Th\in\End_\C(\gA) \mid 
\vln{\Th} = \inf_{f\in\gA} \{ \nu(\Th f) - \nu(f) \} > -\infty \af.
$$
They clearly are continuous for the topology induced by~$\nu$ on~$\gA$;
they form a subalgebra of the algebra of all continuous operators\footnote{
Not all continuous operators of~$\gA$ belong to~$\gE$: think of the operator
of~$\C[[y]]$ which maps $y^m$ to~$y^{m/2}$ if $m$ is even and to~$0$ if $m$
is odd.
}
and $(\gE,\valn)$ is a complete pseudovaluation ring.
For any formally summable family $(\Th_i)_{i\in I}$ of sum~$\Th$ in~$\gE$ and
$f\in\gA$, the family $(\Th_i f)_{i\in I}$ is summable in the topological
ring $\gA$, with sum~$\Th f$.

%%%%%%%%%%%%%%%%%%%%%%%%%%%%%%%%%%%%%%%%%%%%%%%%%%

\begin{lemma}	\label{lemCVformal}
With the notation of formula~\eqref{eqdefbBn} and Lemma~\ref{lemdefcV}, the
family $(\cV^{\om_1,\dotsc,\om_r} \bB_{\om_1,\dotsc,\om_r})_{r\ge1,\,
\om_1,\dotsc,\om_r\in\cN}$ is formally summable in the algebra of
operators of $\C[[x,y]]$ having a valuation \wrt~$\nu_4$.
In particular the resulting operator~$\Th$ is continuous for the Krull
topology.
Similarly, the formula
\begin{equation}	\label{eqdefcVt}
\cVt^{\om_1,\dotsc,\om_r} = (-1)^r \cV^{\om_r,\dotsc,\om_1}
\end{equation}
gives rise to a formally summable family
$(\cVt^{\om_1,\dotsc,\om_r} \bB_{\om_1,\dotsc,\om_r})_{r\ge1,\,
\om_1,\dotsc,\om_r\in\cN}$. 
\end{lemma}

\begin{proof}

Clearly $\nu_4(B_n f)  \ge \nu_4(f) + n$ and, by induction,
$$
\nu_4(\bB_{\om_1,\dotsc,\om_r} f)  \ge \nu_4(f) + \om_1+\dotsb+\om_r.
$$
As a consequence of~\eqref{eqvalcV}, 
$$
\nu_4(\cV^{\om_1,\dotsc,\om_r} \bB_{\om_1,\dotsc,\om_r} f)  
\ge \nu_4(f) + \om_1+\dotsb+\om_r + 2r, \qquad
\om_1,\dotsc,\om_r\in\cN.
$$
Hence, with the above notations, each $\cV^{\om_1,\dotsc,\om_r}
\bB_{\om_1,\dotsc,\om_r}$ is an element $\gE$ with valuation $\ge
\om_1+\dotsb+\om_r + 2r$,
and the same thing holds for each $\cVt^{\om_1,\dotsc,\om_r}
\bB_{\om_1,\dotsc,\om_r}$.

The $\om_i$'s may be negative but they are always $\ge-1$, thus
$\om_1+\dotsb+\om_r + r\ge0$.
Therefore, for any $\de>0$, the condition $\om_1+\dotsb+\om_r + 2r \le \de$
implies $r\le \de$ and 
$\sum (\om_i+1) = \om_1+\dotsb+\om_r + r \le \de$.
Since this condition is fulfilled only a finite number of times, the conclusion
follows.
\end{proof}

%%%%%%%%%%%%%%%%%%%%%%%%%%%%%%%%%%%%%%%%%%%%%%%%%%%%%%%%%%%%%%%%%%%%%%%%%
%%%%%%%%%%%%%%%%%%%%%%%%%%%%%%%%%%%%%%%%%%%%%%%%%%%%%%%%%%%%%%%%%%%%%%%%%

\parag
Here is the key statement, the proof of which will be spread over 
Sections~\ref{secAlgMoulds}--\ref{secPfThm}:

\begin{thm}	\label{thmSNformal}
The continuous operator $\Th = \sum \cV^\bul \bB_\bul$ defined by
Lemmas~\ref{lemdefcV} and~\ref{lemCVformal} is an algebra automorphism of
$\C[[x,y]]$ which satisfies the conjugacy equation~\eqref{eqconjugOp}.
The inverse operator is $\sum \cVt^\bul \bB_\bul$.
\end{thm}

Observe that $\Th x = x$, thus $\Th$ is must be the substitution operator for a
formal transformation of the form $\th(x,y) = \big( x,\ph(x,y)\big)$,
with $\ph = \Th y$, in accordance with~\eqref{eqdefthph}.
An easy induction yields
\begin{equation}	\label{eqdefbeomb}
\bB_\omb y = \be_\omb y^{\om_1+\dotsb+\om_r+1}, \qquad \omb\in\cN^r,\, r\ge 1,
\end{equation}
with $\be_\omb = 1$ if $r=1$,
$\be_\omb = (\om_1+1)(\om_1+\om_2+1)\dotsm(\om_1+\dotsb+\om_{r-1}+1)$ if $r\ge2$.
We have $\be_\omb=0$ whenever $\om_1+\dotsb+\om_r \le -2$ (since \eqref{eqdefbeomb}
holds a priori in the fraction field $\C(\!(y)\!)$ but $\bB_\omb y$ belongs to~$\C[[y]]$), hence
\begin{multline}	\label{eqseriesgivphn}
\th(x,y) = \big( x,\ph(x,y)\big), \\
\ph(x,y) = y + \sum_{n\ge0} \ph_n(x) y^n, \qquad 
\ph_n = \sum_{\substack{r\ge1, \, \omb\in\cN^r  \\ \om_1+\dotsb+\om_r+1=n}} 
\be_\omb \cV^\omb
\end{multline}
(in the series giving $\ph_n$, there are only finitely many terms for each~$r$,
\eqref{eqvalcV} thus yields its formal convergence in $x\C[[x]]$).

Similarly, $\Th\ii = \sum \cVt^\bul \bB_\bul$ is the substitution operator of
a formal transformation $(x,y) \mapsto \big(x,\psi(x,y)\big)$, which is nothing
but~$\th\ii$, and 
\begin{equation}	\label{eqpsiThii}
\psi(x,y) = \Th\ii y = y + \sum_{n\ge0} \psi_n(x) y^n,
\end{equation} 
where each coefficient can be represented as a formally convergent series
$\dst
\psi_n = \sum_{\om_1+\dotsb+\om_r+1=n}
%
% \psi_n = \sum_{\substack{r\ge1, \, \omb\in\cN^r  \\ \om_1+\dotsb+\om_r+1=n}} 
%
\be_\omb \cVt^\omb$.

\smallskip

See Lemma~\ref{lemLagrpsinphn} on p.~\pageref{lemLagrpsinphn} for formulas
relating directly the~$\ph_n$'s and the~$\psi_n$'s.

\begin{remark}	\label{remheuri}
The $\cV^\omb$'s are generically divergent with at most Gevrey-$1$ growth of the
coefficients, as can be expected from formula~\eqref{eqdefinvpamu}; for
instance, for $\om_1\neq0$, we get 
$\cV^{\om_1}(x) = \sum \om_1^{-r-1} \left(-x^2\frac{\dd\,}{\dd
x}\right)^r a_{\om_1}$ which is generically divergent because the repeated
differentiations are not compensated by the division by any factorial-like
expression. 
This divergence is easily studied through formal Borel transform \wrt\ $z=-1/x$,
which is the starting point of the resurgent analysis of the saddle-node---see
Section~\ref{secResur}.
We shall see in Section~\ref{secBESN} why the $\cV^\omb$'s can be called
``resurgence monomials''.
\end{remark}

The proof of Theorem~\ref{thmSNformal} will follow easily from the general
notions introduced in the next sections.

%%%%%%%%%%%%%%%%%%%%%%%%%%%%%%%%%%%%%%%%%%%%%%%%%%%%%%%%%%%%%%%%%%%%%%%%%
%%%%%%%%%%%%%%%%%%%%%%%%%%%%%%%%%%%%%%%%%%%%%%%%%%%%%%%%%%%%%%%%%%%%%%%%%
%%%%%%%%%%%%%%%%%%%%%%%%%%%%%%%%%%%%%%%%%%%%%%%%%%%%%%%%%%%%%%%%%%%%%%%%%
%%%%%%%%%%%%%%%%%%%%%%%%%%%%%%%%%%%%%%%%%%%%%%%%%%%%%%%%%%%%%%%%%%%%%%%%%

\part{The formalism of moulds}

%%%%%%%%%%%%%%%%%%%%%%%%%%%%%%%%%%%%%%%%%%%%%%%%%%%%%%%%%%%%%%%%%%%%%%%%%
%%%%%%%%%%%%%%%%%%%%%%%%%%%%%%%%%%%%%%%%%%%%%%%%%%%%%%%%%%%%%%%%%%%%%%%%%

\section{The algebra of moulds}	\label{secAlgMoulds}

%%%%%%%%%%%%%%%%%%%%%%%%%%%%%%%%%%%%%%%%%%%%%%%%%%%%%%%%%%%%%%%%%%%%%%%%%
%%%%%%%%%%%%%%%%%%%%%%%%%%%%%%%%%%%%%%%%%%%%%%%%%%%%%%%%%%%%%%%%%%%%%%%%%

\parag
In this section and the next three ones, we assume that we are given a non-empty
set~$\Om$ and a commutative $\C$-algebra~$\bA$, the unit of which is
denoted~$1$.
In the previous section, the roles of~$\Om$ and~$\bA$ were played by~$\cN$ and
$\C[[x]]$.

It is sometimes convenient to have a commutative semigroup structure on~$\Om$;
then we would rather take $\Om=\Z$ in the previous section and consider that the
mould $\{ \cV^{\om_1,\dotsc,\om_r}  \}$ was defined on~$\Z$ but supported
on~$\cN$ (\ie\ we extend it by~$0$ whenever one of the $\om_i$'s is $\le-2$).

We consider~$\Om$ as an alphabet and denote by $\Om^\bul$ the free monoid of
{\em words}: a word is any finite sequence of letters, 
$\omb = (\om_1,\dotsc,\om_r)$ with $\om_1, \dotsc,\om_r\in\Om$;
its {\em length} $r = r(\omb)$ can be any non-negative integer.
The only word of zero length is the empty word, denoted~$\est$, which is
the unit of {\em concatenation}, the monoid law $(\omb,\etab)\mapsto\omb\conc\etab$
defined by
$$
(\om_1,\ldots,\om_r) \conc (\eta_1,\ldots,\eta_s) = 
(\om_1,\ldots,\om_r,\eta_1,\ldots,\eta_s)
$$
for non-empty words.

As previously alluded to, a {\em mould on~$\Om$ with values in~$\bA$} is nothing but
a map $\Om^\bul \to \bA$.
It is customary to denote the value of the mould on a word~$\omb$ by
affixing~$\omb$ as an upper index to the symbol representing the mould,
and to refer to the mould itself by using~$\bul$ as upper index.
Hence $\cV^\bul$ is the mould, the value of which at~$\omb$ is denoted $\cV^\omb$.

A mould with values in~$\C$ is called a {\em scalar mould}.

%%%%%%%%%%%%%%%%%%%%%%%%%%%%%%%%%%%%%%%%%%%%%%%%%%%%%%%%%%%%%%%%%%%%%%%%%
%%%%%%%%%%%%%%%%%%%%%%%%%%%%%%%%%%%%%%%%%%%%%%%%%%%%%%%%%%%%%%%%%%%%%%%%%

\parag
Being the set of all maps from a set to the ring~$\bA$, the set of moulds
$\hM^\bul(\Om,\bA)$ has a natural structure of $\bA$-module:
addition and ring multiplication are defined component-wise
(for instance, if $\mu\in\bA$ and $M^\bul\in\hM^\bul(\Om,\bA)$,
the mould $N^\bul = \mu  M^\bul$ is defined by $N^\omb = \mu M^\omb$
for all $\omb\in\Om^\bul$).

The ring structure of~$\bA$ together with the monoid structure
of~$\Om^\bul$ also give rise to a {\em multiplication of moulds}, thus defined:
\begin{equation}	\label{eqdefmultiplimould}
P^\bul = M^\bul \times N^\bul \colon \;
\omb \mapsto P^\omb = \sum_{\omb= \omb^1\!\conc\omb^2} M^{\omb^1} N^{\omb^2},
\end{equation}
with summation over the $r(\omb)+1$ decompositions of~$\omb$ into two words (including $\omb^1$
or $\omb^2 = \est$).
Mould multiplication is associative but not commutative (except if $\Om$ has
only one element).
We get a ring structure on $\hM^\bul(\Om,\bA)$, with unit
$$
1^\bul \colon\; \omb \mapsto 1^\omb = 
\left| \begin{aligned} 
1 \quad &\text{if $\omb = \est$}\\
0 \quad &\text{if $\omb \neq \est$.}
\end{aligned}  \right.
$$
One can check that a mould~$M^\bul$ is invertible if and only if $M^\est$
is invertible in~$\bA$ (see below).

One must in fact regard $\hM^\bul(\Om,\bA)$ as an $\bA$-algebra, \ie\ its module
structure and ring structure are compatible: 
$\mu\in\bA \mapsto \mu\, 1^\bul\in\hM^\bul(\Om,\bA)$ is indeed a ring
homomorphism, the image of which lies in the center of the ring of moulds.
The reader familiar with Bourbaki's {\em Elements of mathematics} will have
recognized in $\hM^\bul(\Om,\bA)$ the large algebra (over~$\bA$) of the
monoid~$\Om^\bul$ ({\em Alg.}, chap.~III, \S2, n$^{\text{o}}$10).
Other authors use the notation $\bA \lan\!\lan \Om \ran\!\ran$ or $\bA[[ \rmT^\Om ]]$
to denote this $\bA$-algebra, viewing it as the completion of the
free $\bA$-algebra over~$\Om$ for the pseuvoluation~$\ord$ defined below.
The originality of moulds lies in the way they are used:
\begin{enumerate}[--]
\item the shuffling operation available in the free monoid~$\Om^\bul$ will lead us in
Section~\ref{secAltSym} to single out specific classes of moulds, enjoying
certain symmetry or antisymmetry properties of fundamental importance (and this
is only a small amount of all the structures used by \'Ecalle in wide-ranging contexts);
\item we shall see in Sections~\ref{secGenMcM} and~\ref{secContrAltSym} how to contract moulds into
``comoulds'' (and this yields non-trivial results in the local study of analytic
dynamical systems);
\item the extra structure of commutative semigroup on~$\Om$ will allow us to
define another operation, the ``composition'' of moulds (see below).
\end{enumerate}

There is a pseudovaluation $\ord \colon \hM^\bul(\Om,\bA) \to
\N\cup\{\infty\}$, which we call ``order'': 
we say that a mould~$M^\bul$ has order $\ge s$ if $M^\omb=0$ whenever
$r(\omb)<s$,
and $\od{M^\bul}$ is the largest such~$s$.
This way, we get a complete pseudovaluation ring
$\big(\hM^\bul(\Om,\bA),\ord\big)$.
In fact, if $\bA$ is an integral domain (as is the case of~$\C[[x]]$), then
$\hM^\bul(\Om,\bA)$ is an integral domain and $\ord$ is a valuation.

%%%%%%%%%%%%%%%%%%%%%%%%%%%%%%%%%%%%%%%%%%%%%%%%%%%%%%%%%%%%%%%%%%%%%%%%%
%%%%%%%%%%%%%%%%%%%%%%%%%%%%%%%%%%%%%%%%%%%%%%%%%%%%%%%%%%%%%%%%%%%%%%%%%

\parag
It is easy to construct ``mould derivations'', \ie\ $\C$-linear operators~$D$ of
$\hM^\bul(\Om,\bA)$ such that $D(M^\bul\times N^\bul) = (D M^\bul)\times N^\bul
+ M^\bul \times D  N^\bul$.

For instance, for any function $\ph \colon \Om \to \bA$, the formula
$$
D_\ph M^\omb = \left| \begin{aligned}
0 \qquad\qquad\qquad \quad & \text{if $\omb = \est$}\\
\big( \ph(\om_1) + \dotsb + \ph(\om_r) \big) M^\omb
\quad & \text{if $\omb = (\om_1,\ldots,\om_r)$}
\end{aligned}  \right.
$$
defines a mould derivation $D_\ph$.
With $\ph\equiv1$, we get $D M^\omb = r(\omb)M^\omb$.

When $\Om$ is a commutative semigroup (the operation of which is denoted
additively), we define the {\em sum of a non-empty word} as
$$ 
\norm{\omb} = \om_1 + \dotsb + \om_r \in \Om,
\qquad \omb = (\om_1, \dotsc, \om_r) \in \Om^\bul.
$$
Then, for any mould $U^\bul$ such that $U^\est=0$, the formula
\begin{equation}	\label{eqdefnaUbul}
\na_{U^\bul} M^\omb = \sum_{\omb=\alb\conc\beb\conc\gab,\,\beb\neq\est}
U^\beb \, M^{\alb\conc \norm{\beb} \conc\gab} 
\end{equation}
defines a mould derivation $\na_{U^\bul}$.
The derivation~$D_\ph$ is nothing but $\na_{U^\bul}$ with $U^\omb =
\ph(\om_1)$ for $\omb=(\om_1)$ and $U^\omb = 0$ for $r(\omb)\neq1$.

When $\Om\subset\bA$, an important example is 
\begin{equation}	\label{eqdefna}
\na M^\omb = \norm{\omb} M^\omb,
\end{equation}
obtained with $\ph(\eta)\equiv\eta$.
On the other hand, every derivation $d\colon\bA\to\bA$ obviously induces a mould
derivation $D$, the action of which on any mould~$M^\bul$ is defined by

\begin{equation}	\label{eqdefDd}
D M^\omb = d(M^\omb), \qquad \omb\in\Om^\bul.
\end{equation}

\begin{remark}	\label{remmouldVsol}
With $\Om=\cN$ defined by~\eqref{eqdefiancN} and $\bA=\C[[x]]$, the
mould~$\cV^\bul$ determined in Lemma~\ref{lemdefcV} is the unique solution of
the mould equation
\begin{equation}	\label{eqmouldeq}
(D+\na)\cV^\bul = J_a^\bul \times \cV^\bul,
\end{equation}
such that $\cV^\est=1$ and $\cV^\omb\in x\C[[x]]$ for $\omb\neq\est$,
with $D$ induced by $d=x^2\frac{\dd\,}{\dd x}$ and
\begin{equation}	\label{eqdefJa}
J_a^\omb = \left| \begin{aligned}
a_{\om_1} \quad & \text{if $\omb = (\om_1)$}\\
0 \ens\; \quad & \text{if $r(\omb)\neq1$}.
\end{aligned}  \right.
\end{equation}
\end{remark}

%%%%%%%%%%%%%%%%%%%%%%%%%%%%%%%%%%%%%%%%%%%%%%%%%%%%%%%%%%%%%%%%%%%%%%%%%
%%%%%%%%%%%%%%%%%%%%%%%%%%%%%%%%%%%%%%%%%%%%%%%%%%%%%%%%%%%%%%%%%%%%%%%%%

\parag
When $\Om$ is a commutative semigroup, the {\em composition of moulds} is defined as follows:
\label{seccomposmoulds}
\begin{multline*}
C^\bul = M^\bul \circ U^\bul \colon \qquad
\est \ens\mapsto\ens C^\est = M^\est, \\
\omb\neq\est \mapsto C^\omb = \sum_{ \substack{s\ge1,\,
\omb^1,\dotsc,\omb^s\neq\est  \\
\omb = \omb^1 \concsm \omb^s}  }
M^{(\norm{\omb^1},\dotsc,\norm{\omb^s})} 
U^{\omb^1}  \dotsm U^{\omb^s}, \qquad \,
\end{multline*}
with summation over all possible decompositions of $\omb$ into non-empty words
(thus $1\le s\le r(\omb)$ and the sum is finite).
The map $M^\bul \mapsto M^\bul \circ U^\bul$ is clearly $\bA$-linear; it is in fact an
$\bA$-algebra homomorphism: 
$$ 
(M^\bul \circ U^\bul) \times (N^\bul \circ U^\bul) = 
(M^\bul\times N^\bul) \circ U^\bul
$$ 
(the verification of this distributivity property is left as an exercise).

Obviously,
$ 1^\bul \circ U^\bul  = 1^\bul $ for any mould~$U^\bul$.
The {\em identity mould}
$$
I^\bul \colon \omb \mapsto I^\omb = 
\left| \begin{aligned} 
1 \quad &\text{if $r(\omb) = 1$}\\
0 \quad &\text{if $r(\omb) \neq 1$}
\end{aligned}  \right.
$$
satisfies $ M^\bul \circ I^\bul = M^\bul $ for any mould~$M^\bul$.
But $ I^\bul \circ U^\bul = U^\bul $ only if $U^\est = 0$ (a requirement that we
could have imposed when defining mould composition, since the value of~$U^\est$
is ignored when computing $M^\bul\circ U^\bul$); in general, 
$ I^\bul \circ U^\bul = U^\bul - U^\est\,1^\bul$.

Mould composition is associative\footnote{%
Hint: The computation of $M^\bul\circ(U^\bul\circ V^\bul)$ at~$\omb$ involves all the
decompositions $\omb = \omb^1 \concsm \omb^s$ into non-empty words and then
all the decompositions of each factor $\omb^i$ as
$\omb^1 = \alb^1 \concsm \alb^{i_1},
\omb^2 = \alb^{i_1+1} \concsm \alb^{i_2}, \dotsc,
\omb^s = \alb^{i_{s-1}+1} \concsm \alb^{i_s}$
(where $1\le i_1 < i_2 < \dotsb < i_s=t$, with each $\alb^j$ non-empty);
it is equivalent to sum first over all the decompositions
$\omb = \alb^1 \concsm \alb^t$ and then to consider all manners of regrouping
adjacent factors $(\alb^1 \concsm \alb^{i_1}) \conc 
(\alb^{i_1+1} \concsm \alb^{i_2}) \concsm 
(\alb^{i_{s-1}+1} \concsm \alb^{i_s})$,
which yields the value of $(M^\bul\circ U^\bul)\circ V^\bul$ at~$\omb$.
}
and not commutative. One can check that a
mould~$U^\bul$ admits an inverse for composition (a mould~$V^\bul$ such that
$V^\bul \circ U^\bul = U^\bul \circ V^\bul = I^\bul$) if and only if
$U^\omb$ is invertible in~$\bA$ whenever $r(\omb)=1$
and $U^\est=0$.
These moulds thus form a group under composition.

In the following, we do not always assume~$\Om$ to be a commutative semigroup
and mould composition is thus not always defined.
However, observe that, in the absence of semigroup structure, the definition of
$M^\bul \circ U^\bul$ makes sense for any mould~$M^\bul$ such that $M^\omb$ only
depends on~$r(\omb)$ and that most of the above properties can be adapted to
this particular situation.

%%%%%%%%%%%%%%%%%%%%%%%%%%%%%%%%%%%%%%%%%%%%%%%%%%%%%%%%%%%%%%%%%%%%%%%%%
%%%%%%%%%%%%%%%%%%%%%%%%%%%%%%%%%%%%%%%%%%%%%%%%%%%%%%%%%%%%%%%%%%%%%%%%%

\parag
As an elementary illustration, one can express the multiplicative inverse of a
mould $M^\bul$ with  $\mu=M^\est$ invertible as
$$
(M^\bul)\iim = G^\bul \circ M^\bul, \qquad
\text{with}\ens
G^\omb = (-1)^{r(\omb)} \mu^{-r(\omb)-1}.
$$
Indeed, $G^\bul$ is nothing but the multiplicative inverse of
$\mu\,1^\bul+I^\bul$ and
$$
M^\bul = \mu\,1^\bul + I^\bul\circ M^\bul = (\mu\,1^\bul+I^\bul)\circ M^\bul,
$$
whence the result follows immediately.

The above computation does not require any semigroup structure on~$\Om$.
Besides, one can also write
$ (M^\bul)\iim = \sum_{s\ge0}  (-1)^s \mu^{-s-1} (M^\bul-\mu \, 1^\bul)^{\times s}$
(convergent series for the topology of $\hM^\bul(\Om,\bA)$ induced by~$\ord$).

%%%%%%%%%%%%%%%%%%%%%%%%%%%%%%%%%%%%%%%%%%%%%%%%%%%%%%%%%%%%%%%%%%%%%%%%%
%%%%%%%%%%%%%%%%%%%%%%%%%%%%%%%%%%%%%%%%%%%%%%%%%%%%%%%%%%%%%%%%%%%%%%%%%

\parag
%
% Here is another application of mould composition.
%
We define elementary scalar moulds $\exp_t^\bul$, $t\in\C$, and $\log^\bul$ by the formulas
$\exp_t^\omb = \frac{t^{r(\omb)}}{r(\omb)!}$ and
$$
\log^\omb = 0 \quad \text{if $\omb = \est$}, \qquad 
\log^\omb = \tfrac{ (-1)^{r(\omb)-1} }{ r(\omb) } \quad \text{if $\omb \neq \est$.}
$$
One can check that
\begin{gather*}
\exp_0^\bul = 1^\bul, \qquad
\exp_{t_1}^\bul \times \exp_{t_2}^\bul = \exp_{t_1+t_2}^\bul, 
\qquad t_1,t_2\in\C,\\
(\exp_t^\bul - 1^\bul) \circ \frac{1}{t}\log^\bul
= \frac{1}{t}\log^\bul \circ \, (\exp_t^\bul - 1^\bul)
= I^\bul, \qquad t\in\C^*
\end{gather*}
(use for instance $\exp_t^\bul = \sum_{s\ge0} \frac{t^s}{s!}(I^\bul)^{\times
s}$ and $\log^\bul = \sum_{s\ge1} \frac{(-1)^{s-1}}{s} (I^\bul)^{\times s}$;
mould composition is well-defined here even if~$\Om$ is not a semigroup).

Now, consider on the one hand the Lie algebra 
\begin{multline}	\label{eqdefgLbul}
\gL^\bul(\Om,\bA) = \ao U^\bul  \in \hM^\bul(\Om,\bA) \mid U^\est = 0 \af,\\
\quad \text{with bracketting 
$[U^\bul,V^\bul] = U^\bul\times V^\bul - V^\bul\times U^\bul$,}
\end{multline}
and on the other hand the subgroup
\begin{equation}	\label{eqdefGbul}
G^\bul(\Om,\bA) = \ao M^\bul  \in \hM^\bul(\Om,\bA) \mid M^\est = 1 \af
\end{equation}
of the multiplicative group of invertible moulds.

Then, for each $U^\bul \in \gL^\bul(\Om,\bA)$, $( \exp_t^\bul\circ \, U^\bul)_{t\in\C}$
is a one-parameter group inside $G^\bul(\Om,\bA)$.
Moreover, the map
$$
E_t \colon 
U^\bul  \in \gL^\bul(\Om,\bA) \mapsto
M^\bul = \exp_t^\bul \circ\, U^\bul \in G^\bul(\Om,\bA)
$$
is a bijection for each $t\in\C^*$ (with reciprocal $M^\bul\mapsto
\frac{1}{t}\log^\bul\circ\, U^\bul$),
which allows us to consider $\gL^\bul(\Om,\bA)$ as the Lie algebra of
$G^\bul(\Om,\bA)$ in the sense that
$$
[U^\bul,V^\bul] = \frac{\dd\,}{\dd t} \Big(
E_t(U^\bul) \times V^\bul \times E_t(U^\bul)\iim
\Big)_{\tst | t=0}.
$$
Observe that mould composition is not necessary to define the map~$E_t$ and its
reciprocal: one can use the series
\begin{equation}	\label{eqdefEt}
E_t(U^\bul) = \sum_{s\ge0} \frac{t^s}{s!} (U^\bul)^{\times s}, \quad
E_t\ii(M^\bul) = \frac{1}{t} \sum_{s\ge1}  \tfrac{(-1)^{s-1}}{s} (M^\bul-1^\bul)^{\times s}
\end{equation}
(they are formally convergent because $\od{U^\bul}$ and $\od{M^\bul-1^\bul}\ge1$).

%%%%%%%%%%%%%%%%%%%%%%%%%%%%%%%%%%%%%%%%%%%%%%%%%%%%%%%%%%%%%%%%%%%%%%%%%
%%%%%%%%%%%%%%%%%%%%%%%%%%%%%%%%%%%%%%%%%%%%%%%%%%%%%%%%%%%%%%%%%%%%%%%%%

\section{Alternality and symmetrality}	\label{secAltSym}

%%%%%%%%%%%%%%%%%%%%%%%%%%%%%%%%%%%%%%%%%%%%%%%%%%%%%%%%%%%%%%%%%%%%%%%%%
%%%%%%%%%%%%%%%%%%%%%%%%%%%%%%%%%%%%%%%%%%%%%%%%%%%%%%%%%%%%%%%%%%%%%%%%%

\parag
Even if $\Om$ is not a semigroup,
another operation available in~$\Om^\bul$ is {\em shuffling}: if two non-empty words 
$\omb^1 = (\om_1,\dotsc,\om_\ell)$
and $\omb^2 = (\om_{\ell+1},\dotsc,\om_r)$
are given, one says that a word $\omb$ belongs to their shuffling
if it can be written $(\om_{\sig(1)},\dotsc,\om_{\sig(r)})$ with a
permutation~$\sig$ such that 
$\sig(1)<\dotsm<\sig(\ell)$ and $\sig(\ell+1)<\dotsm<\sig(r)$
(in other words, $\omb$ can be obtained by interdigitating the letters
of~$\omb^1$ and those of~$\omb^2$ while preserving their internal order
in~$\omb^1$ or~$\omb^2$).
We denote by $\sh{\omb^1}{\omb^2}{\omb}$ the number of such permutations~$\sig$,
and we set $\sh{\omb^1}{\omb^2}{\omb}=0$ if $\omb$ does not belong to the
shuffling of~$\omb^1$ and~$\omb^2$.

\begin{definition}	\label{defialtsym}
A mould~$M^\bul$ is said to be alternal if $M^\est=0$ and, for any two non-empty words $\omb^1$,
$\omb^2$,
\begin{equation}	\label{eqdefaltal}
\sum_{\omb\in\Om^\bul}  \sh{\omb^1}{\omb^2}{\omb} M^\omb = 0.
\end{equation}
It is said to be symmetral if $M^\est=1$ and, for any two non-empty words $\omb^1$,
$\omb^2$,
\begin{equation}	\label{eqdefsymal}
\sum_{\omb\in\Om^\bul}  \sh{\omb^1}{\omb^2}{\omb} M^\omb = M^{\omb^1} M^{\omb^2}.
\end{equation}
\end{definition}

Of course the above sums always have finite support.
For instance, if $\omb^1=(\om_1)$ and $\omb^2=(\om_2,\om_3)$, the \lhs\ in both
previous formulas is
$M^{\om_1,\om_2,\om_3} + M^{\om_2,\om_1,\om_3} + M^{\om_2,\om_3,\om_1}$.

The motivation for this definition lies in formula~\eqref{eqmotivsh} below.
We shall see in Section~\ref{secmotivsh} the interpretation of alternality or
symmetrality in terms of the operators obtained by mould-comould expansions:
alternal moulds will be related to the Lie algebra of derivations, symmetral
moulds to the group of automorphisms.

Alternal (resp.\ symmetral) moulds have to do with primitive (resp.\ group-like)
elements of a certain graded cocommutative Hopf algebra, at least when $\bA$ is a
field---see the remark on Lemma~\ref{lemxitenstau} below.

An obvious example of alternal mould is $I^\bul$, or any moud~$J^\bul$ such that
$J^\omb=0$ for $r(\omb)\neq1$ (as is the case of~$J_a^\bul$ defined by~\eqref{eqdefJa}). 
An elementary example of symmetral mould is $\exp_t^\bul$ for any $t\in\C$;
a non-trivial example is the mould~$\cV^\bul$ determined by
Lemma~\ref{lemdefcV}, the symmetrality of which is the object of
Proposition~\ref{propcVsym} below.
The mould $\log^\bul$ is not alternal (nor symmetral), but ``alternel'';
alternelity and symmetrelity are two other types of symmetry introduced by
\'Ecalle, parallel to alternality and symmetrality, but we shall not be
concerned with them in this text (see however the end of Section~\ref{secAlludel}).

The next paragraphs contain the proof of the following properties:

\begin{prop}	\label{propstructaltsym}
Alternal moulds form a Lie subalgebra $\gL^\bul_{\text{alt}}(\Om,\bA)$ of the
Lie algebra $\gL^\bul(\Om,\bA)$ defined by~\eqref{eqdefgLbul}.
Symmetral moulds form a subgroup $G^\bul_{\text{sym}}(\Om,\bA)$ of the
multiplicative group $G^\bul(\Om,\bA)$ defined by~\eqref{eqdefGbul}.
The map~$E_t$ defined by~\eqref{eqdefEt} induces a bijection from
$\gL^\bul_{\text{alt}}(\Om,\bA)$ to $G^\bul_{\text{sym}}(\Om,\bA)$ for each $t\in\C^*$.
\end{prop}

\begin{prop}	\label{propinvsym}
Given a mould~$M^\bul$, we define a mould $S M^\bul = \wt M^\bul$ by the formulas 
\begin{equation}	\label{eqdefinvol}
\wt M^\est = M^\est, \qquad
\wt M^{\om_1,\dotsc,\om_r} = (-1)^r M^{\om_r,\dots,\om_1}, \qquad
r\ge1, \; \om_1, \dotsc, \om_r \in \Om.
\end{equation}
Then $S$ is an involution and an antihomomorphism of the $\bA$-algebra
$\hM^\bul(\Om,\bA)$, and
\begin{align*}
M^\bul \ens\text{alternal} &\quad\Rightarrow\quad
S M^\bul = - M^\bul, \\
M^\bul \ens\text{symmetral} &\quad\Rightarrow\quad
S M^\bul = (M^\bul)\iim \ens \text{(multiplicative inverse)}.
\end{align*}
\end{prop}

\begin{prop}	\label{propcomposalt}
If $\Om$ is a commutative semigroup and $U^\bul$ is alternal, then 
\begin{align}
M^\bul \ens\text{alternal} &\quad\Rightarrow\quad
M^\bul\circ U^\bul \ens\text{alternal,}\\
M^\bul \ens\text{symmetral} &\quad\Rightarrow\quad
M^\bul\circ U^\bul \ens\text{symmetral.}
\end{align}
If moreover $U^\bul$ admits an inverse for composition (\ie\ if $U^\omb$ has a
multiplicative inverse in~$\bA$ whenever $r(\omb)=1$), then this inverse is
alternal itself; thus alternal invertible moulds form a subgroup of the group
(for composition) of invertible moulds.
\end{prop}

\begin{prop}	\label{propDerivSym}
If $D$ is a mould derivation induced by a derivation of~$\bA$, or of the form
$D_\ph$ with $\ph\colon\Om\to\bA$, or of the form $\na_{J^\bul}$ with $J^\bul$
alternal (with the assumption that $\Om$ is a
commutative semigroup in this last case), and if $M^\bul$ is symmetral, then 
$(D M^\bul)\times (M^\bul)\iim$ and $(M^\bul)\iim \times(D M^\bul)$
are alternal.
\end{prop}

%%%%%%%%%%%%%%%%%%%%%%%%%%%%%%%%%%%%%%%%%%%%%%%%%%%%%%%%%%%%%%%%%%%%%%%%%
%%%%%%%%%%%%%%%%%%%%%%%%%%%%%%%%%%%%%%%%%%%%%%%%%%%%%%%%%%%%%%%%%%%%%%%%%

\parag 
The following definition will facilitate the proof of most of these properties
and enlighten the connection with derivations and algebra automorphisms to be
discussed in Section~\ref{secaltsym}.

\begin{definition}
We call dimould\footnote{
Not to be confused with the {\em bimoulds} introduced by \'Ecalle in
connection with Multizeta values, which correspond to the case where the
set~$\Om$ itself is the cartesian product of two sets---see the end of
Section~\ref{secOtherAppli}. 
}
any map~$\bM^\cbul$ from $\Om^\bul\times\Om^\bul$ to~$\bA$;
its value on $(\omb,\etab)$ is denoted $\bM^{\omb,\etab}$.
The set of dimoulds is denoted $\hM^\bbul(\Om,\bA)$; when viewed as the large
algebra of the monoid $\Om^\bul\times\Om^\bul$, it is a non-commutative $\bA$-algebra.
\end{definition}

Observe that, the monoid law on $\Om^\bul\times\Om^\bul$ being 
$$
\vpib^1=(\omb^1,\etab^1),\; \vpib^2=(\omb^2,\etab^2)
\ens\Rightarrow\ens
\vpib^1\conc\vpib^2 = (\omb^1\conc\etab^1,\omb^2\conc\etab^2),
$$
the finitess of the number of decompositions of any
$\vpib\in\Om^\bul\times\Om^\bul$ as $\vpib = \vpib^1\conc\vpib^2$ allows us to
consider this large algebra, in which the multiplication is defined by a formula
similar to~\eqref{eqdefmultiplimould}.
The unit of dimould multiplication is $1^\cbul \colon (\omb,\etab)\mapsto 1$ if
$\omb=\etab=\est$ and $0$ otherwise.

\begin{lemma}	\label{lemhomomtau}
The map $\tau \colon M^\bul\in\hM^\bul(\Om,\bA) \mapsto \bM^\cbul\in\hM^\bbul(\Om,\bA)$ defined by
$$
\bM^{\alb,\beb} = \sum_{\omb\in\Om^\bul} \sh{\alb}{\beb}{\omb} M^\omb,
\qquad \alb,\beb \in \Om^\bul
$$
is an $\bA$-algebra homomorphism.
\end{lemma}

\begin{proof}
The map~$\tau$ is clearly $\bA$-linear and $\tau(1^\bul) = 1^\cbul$.
Let $P^\bul=M^\bul \times N^\bul$ and $\bP^\cbul=\tau(P^\bul)$; since
$$
\bP^{\alb,\beb} = \sum_{\gab^1,\gab^2\in\Om^\bul} 
\sh{\alb}{\beb}{\gab^1\conc\gab^2} M^{\gab^1} N^{\gab^2},
\qquad \alb,\beb\in\Om^\bul,
$$
the property $\bP^\cbul = \tau(M^\bul) \times \tau(N^\bul)$
follows from the identity
\begin{equation}	\label{eqidentiteshdeux}
\sh{\alb}{\beb}{\gab^1\conc\gab^2} = 
\sum_{\alb=\alb^1\conc\alb^2,\, \beb=\beb^1\conc\beb^2}
\sh{\alb^1}{\beb^1}{\gab^1} \sh{\alb^2}{\beb^2}{\gab^2}
\end{equation}
(the verification of which is left to the reader).
\end{proof}

As in the case of moulds, we can define the ``order'' of a dimould and get a
pseudovaluation $\ord \colon \hM^\bbul(\Om,\bA) \to \N\cup\{\infty\}$: 
by definition $\od{\bM^\cbul}  \ge s$ if $\bM^{\omb,\etab}=0$ whenever
$r(\omb)+r(\etab)<s$.
We then get a complete pseudovaluation ring $\big(\hM^\bbul(\Om,\bA),\ord\big)$
and the homomorphism $\tau$ is continuous since
$\od{\tau(M^\bul)} \ge \od{M^\bul}$.

\begin{definition}	\label{defidec}
We call decomposable a dimould $\bP^\cbul$ of the form
$\bP^{\omb,\etab} = M^\omb N^\etab$ (for all $\omb,\etab\in\Om^\bul$), where
$M^\bul$ and $N^\bul$ are two moulds.
We then use the notation
$\bP^\cbul = M^\bul \otimes N^\bul$.
\end{definition}

One can check that the relation
\begin{equation}	\label{eqrhohomomAalg}
(M_1^\bul \otimes N_1^\bul) \times (M_2^\bul \otimes N_2^\bul) =
(M_1^\bul \times M_2^\bul) \otimes (M_2^\bul \times N_2^\bul) 
\end{equation}
holds in $\hM^\bbul(\Om,\bA)$, for any four moulds $M_1^\bul,N_1^\bul,M_2^\bul,N_2^\bul$.

With this notation for decomposable dimoulds, we can now rephrase
Definition~\ref{defialtsym} with the help of the homomorphism~$\tau$ of
Lemma~\ref{lemhomomtau}:
\begin{lemma}	\label{lemtaualtsym}
A mould~$M^\bul$ is alternal iff $\tau(M^\bul)=M^\bul\otimes 1^\bul + 1^\bul\otimes M^\bul$.
A mould~$M^\bul$ is symmetral iff $M^\est=1$ and $\tau(M^\bul)=M^\bul\otimes M^\bul$.
\end{lemma}

Notice that the image of~$\tau$ is contained in the set of {\em symmetric dimoulds},
\ie\ those $\bM^\cbul$ such that $\bM^{\alb,\beb} = \bM^{\beb,\alb}$, because of
the obvious relation
\begin{equation}	\label{eqcommsh}
\sh{\alb}{\beb}{\omb} = \sh{\beb}{\alb}{\omb},
\qquad \alb,\beb,\omb\in\Om^\bul.
\end{equation}

%%%%%%%%%%%%%%%%%%%%%%%%%%%%%%%%%%%%%%%%%%%%%%%%%%%%%%%%%%%%%%%%%%%%%%%%%
%%%%%%%%%%%%%%%%%%%%%%%%%%%%%%%%%%%%%%%%%%%%%%%%%%%%%%%%%%%%%%%%%%%%%%%%%

\parag 
\emph{Remark on Definition~\ref{defidec}.}
The tensor product is used here as a mere notation, which is related to the
tensor product of $\bA$-algebras as follows:
there is a unique $\bA$-linear map
$\rho \colon \hM^\bul(\Om,\bA) \otimes_\bA \hM^\bul(\Om,\bA) \to \hM^\bbul(\Om,\bA)$
such that $\rho(M^\bul \otimes N^\bul)$ is the above dimould~$\bP^\cbul$.
The map~$\rho$ is an $\bA$-algebra homomorphism, according to~\eqref{eqrhohomomAalg}, however its
injectivity is not obvious when~$\bA$ is not a field, and denoting $\rho(M^\bul
\otimes N^\bul)$ simply as $M^\bul \otimes N^\bul$, as in
Definition~\ref{defidec}, is thus an abuse of notation.  

In fact, if $\bA$ is an integral domain, then the $\bA$-module
$\hM^\bul(\Om,\bA)$ is torsion-free ($\mu M^\bul=0$ implies $\mu=0$ or
$M^\bul=0$) and $\Ker\rho$ coincides with the set $\gT$ of all torsion elements of
$\hM^\bul(\Om,\bA) \otimes_\bA \hM^\bul(\Om,\bA)$.
Indeed, for any $\xi\in\gT$, there is a non-zero $\mu\in\bA$ such that
$\mu\xi=0$, thus $\mu\rho(\xi)=0$ in $\hM^\bbul(\Om,\bA)$, whence $\rho(\xi)=0$.
Conversely, suppose $\xi = \sum_{i=1}^n M_i^\bul  \otimes N_i^\bul \in \Ker\rho$,
where the moulds $M_i^\bul$ are not all zero; without loss of generality we can
suppose $M_n^\bul\neq0$ and choose $\omb^1\in\Om^\bul$ such that $\mu_n =
M_n^{\omb^1} \neq0$. Setting $\mu_i = M_i^{\omb^1}$ for the other $i$'s, we get
$\sum_{i=0}^n \mu_i N_i^\bul = 0$, whence
$\mu_n \xi = \sum_{i=0}^{n-1} (\mu_n M_i^\bul - \mu_i M_n^\bul)\otimes N_i^\bul$,
still with $\mu_n\xi\in\Ker\rho$. By induction on~$n$, one gets a non-zero
$\mu\in\bA$ such that $\mu\xi=0$.

Therefore, $\rho$ is injective when~$\bA$ is a principal integral domain,
as is the case of $\C[[x]]$,
because any torsion-free $\bA$-module is then flat (Bourbaki, {\em Alg.\ comm.},
chap.~I, \S2, n$^{\text{o}}$4, Prop.~3), hence its tensor product with itself is also
torsion-free (by flatness, the injectivity of $\phi \colon M^\bul \mapsto \mu
M^\bul$, for $\mu\neq0$, implies the injectivity of $\phi\otimes\ID \colon \xi
\mapsto \mu\xi$).

This is a fortiori the case when $\bA$ is a field; this is used in the
remark on Lemma~\ref{lemxitenstau} below. 

%%%%%%%%%%%%%%%%%%%%%%%%%%%%%%%%%%%%%%%%%%%%%%%%%%%%%%%%%%%%%%%%%%%%%%%%%
%%%%%%%%%%%%%%%%%%%%%%%%%%%%%%%%%%%%%%%%%%%%%%%%%%%%%%%%%%%%%%%%%%%%%%%%%

\parag 
\emph{Proof of Proposition~\ref{propstructaltsym}.}
The set $\gL^\bul_{\text{alt}}(\Om,\bA)$ of alternal moulds is clearly an
$\bA$-submodule of $\gL^\bul(\Om,\bA)$.
Given $U^\bul$ and $V^\bul$ in this set, the alternality of $[U^\bul,V^\bul]$ is
easily checked with the help of Lemma~\ref{lemhomomtau},
formula~\eqref{eqrhohomomAalg} and Lemma~\ref{lemtaualtsym}.

Let $M^\bul$ and $N^\bul$ be symmetral. The symmetrality of $M^\bul\times
N^\bul$ follows from Lemma~\ref{lemhomomtau}, formula~\eqref{eqrhohomomAalg} and
Lemma~\ref{lemtaualtsym}.
Similarly, the multiplicative inverse $\wt M^\bul$ of~$M^\bul$ satisfies
$\tau(\wt M^\bul)\times\tau(M^\bul) = \tau(M^\bul)\times\tau(\wt
M^\bul) = \tau(1^\bul) = 1^\cbul$, by uniqueness of the multiplicative inverse in
$\hM^\bbul(\Om,\bA)$ it follows that $\tau(\wt M^\bul) = \wt
M^\bul\otimes \wt M^\bul$ and $\wt M^\bul$ is symmetral.

Now let $t\in\C^*$. Suppose first $U^\bul\in\gL^\bul_{\text{alt}}(\Om,\bA)$.
We check that $M^\bul = E_t(U^\bul)$ is symmetral by using the continuity of~$\tau$
and formula~\eqref{eqdefEt}:
$\tau(M^\bul) = \exp(a+b)$ with $a=t U^\bul \otimes 1^\bul$ and $b=1^\bul \otimes
t U^\bul$, where the exponential series is well-defined in $\hM^\bbul(\Om,\bA)$
because $\od{a},\od{b}>0$; since $a\times b=b\times a$, the standard properties
of the exponential series yield 
$$
\tau(M^\bul) = \exp(a) \times \exp(b) =  
\big( \exp(t U^\bul)\otimes 1^\bul \big) \times \big( 1^\bul \otimes \exp(t U^\bul) \big)
= M^\bul \otimes M^\bul.
$$
Conversely, supposing $M^\bul = 1^\bul+N^\bul \in G_{\text{sym}}(\Om,\bA)$, we check that
$U^\bul = E_t\ii(M^\bul)$ is alternal:
by continuity, we can apply $\tau$ termwise to the logarithm series
in~\eqref{eqdefEt} and write $\tau(M^\bul-1^\bul) = N^\bul\otimes 1^\bul +
1^\bul\otimes N^\bul + N^\bul\otimes N^\bul = a + b + a\times b$, with
$a = N^\bul\otimes 1^\bul$ and $b = 1^\bul\otimes N^\bul$ commuting in
$\hM^\bbul(\Om,\bA)$, the conclusion then follows from the identity 
$$
\sum_{s\ge1}\tfrac{(-1)^{s-1}}{s} (a + b + a\times b)^{\times s} =
\sum_{s\ge1}\tfrac{(-1)^{s-1}}{s} a^{\times s} +
\sum_{s\ge1}\tfrac{(-1)^{s-1}}{s} b^{\times s}
$$
(which follows from the observation that, given $c\in\hM^\bbul(\Om,\bA)$ with
$\od{c}>0$, $\sum\tfrac{(-1)^{s-1}}{s} c^{\times s}$ is the only dimould~$\ell$
of positive order such that $\exp(\ell)=1^\cbul+c$).

%%%%%%%%%%%%%%%%%%%%%%%%%%%%%%%%%%%%%%%%%%%%%%%%%%%%%%%%%%%%%%%%%%%%%%%%%
%%%%%%%%%%%%%%%%%%%%%%%%%%%%%%%%%%%%%%%%%%%%%%%%%%%%%%%%%%%%%%%%%%%%%%%%%

\parag 
\emph{Proof of Proposition~\ref{propinvsym}.}
It is obvious that $S$ is an involution and the identity
$$
S(M^\bul\times N^\bul) = S N^\bul \times S M^\bul,
\qquad M^\bul, N^\bul \in \hM^\bul(\Om,\bA)
$$
clearly follows from the Definition~\eqref{eqdefmultiplimould} of mould multiplication.
Let us define an $\bA$-linear map 
$$
\xi \,\colon\; \bM^\cbul \in \hM^\bbul(\Om,\bA) \mapsto 
P^\bul = \xi(\bM^\cbul) \in \hM^\bul(\Om,\bA)
$$ 
by the formula
\begin{equation}	\label{eqdefxiPbM}
P^\omb = \sum_{\omb=\alb\conc\beb} (-1)^{r(\alb)} \bM^{\wt\alb,\beb},
\qquad \omb \in \Om^\bul,
\end{equation}
where $\wt\alb = (\om_i,\dotsc,\om_1)$ for $\alb = (\om_1,\dotsc,\om_i)$ with $i\ge1$
and $\wt\est = \est$.
Thus 
\begin{multline*}
P^\est = \bM^{\est,\est}, \quad
P^{(\om_1)} = \bM^{\est,(\om_1)} - \bM^{(\om_1),\est}, \\
P^{(\om_1,\om_2)} = \bM^{\est,(\om_1,\om_2)} - \bM^{(\om_1),(\om_2)} + \bM^{(\om_2,\om_1),\est},
\end{multline*}
and so on. The rest of Proposition~\ref{propinvsym} follows from
\begin{lemma}	\label{lemxitenstau}
For any two moulds $M^\bul,N^\bul$, one has
\begin{gather}	
\label{eqxitens}
\xi(M^\bul\otimes N^\bul) = (S M^\bul) \times N^\bul, \\
%
% with the notation of Proposition~\ref{propinvsym} for~$\wt M^\bul$, and
%
\label{eqxitau}
\xi\circ\tau(M^\bul) = M^\est \, 1^\bul,
\end{gather}
with the homomorphism~$\tau$ of Lemma~\ref{lemhomomtau}.
\end{lemma}

Indeed, if $M^\bul$ is alternal, then 
$$
S M^\bul + M^\bul = (S M^\bul)\times 1^\bul + 1^\bul\times M^\bul 
= \xi(M^\bul \otimes 1^\bul + 1^\bul \otimes M^\bul) = \xi\circ\tau(M^\bul) = 0,
$$
and if $M^\bul$ is symmetral, then 
$$
(S M^\bul) \times M^\bul = \xi(M^\bul \otimes M^\bul) =
\xi\circ\tau(M^\bul) = 1^\bul
$$
and similarly $M^\bul \times S M^\bul = 1^\bul$ because $S M^\bul$ is
clearly symmetral too.

\medskip

\noindent
\emph{Proof of Lemma~\ref{lemxitenstau}.}
Formula~\eqref{eqxitens} is obvious.
Let $\bM^\cbul = \tau(M^\bul)$ and $P^\bul = \xi(\bM^\cbul)$.
Clearly $P^\est = \bM^{\est,\est} = M^\est$. 
Let $\omb = (\om_1,\ldots,\om_r)$ with $r\ge 1$: we must show that $P^\omb=0$.

Using the notations $\alb^i = (\om_1,\dotsc,\om_i)$
and $\beb^i = (\om_{i+1},\dotsc,\om_r)$ for $0 \le i \le r$
(with $\alb^0 = \beb^r = \est$),
we can write
$P^\omb = \sum_{i=0}^r (-1)^i \bM^{\wt\alb^i,\beb^i}$;
we then split the sum
$$
\bM^{\wt\alb^i,\beb^i} = \sum_\gab \tsh{\wt\alb^i}{\beb^i}{\gab} M^\gab
$$
according to the first letter of the mute variable:
$\bM^{\wt\alb^i,\beb^i} = Q_i + R_i$ with
\begin{alignat*}{3}
Q_i &= \sum_{\gab}
\tsh{\wt\alb^{i-1}}{\beb^i}{\gab} M^{(\om_i)\conc\gab}
&\ens &\text{if $1\le i\le r$},&\qquad Q_0&=0,\\
R_i &= \sum_{\gab}
\tsh{\wt\alb^i}{\beb^{i+1}}{\gab} M^{(\om_{i+1})\conc\gab}
&\ens &\text{if $1\le i\le r-1$},&\qquad R_r&=0.
\end{alignat*}
But, if $0 \le i \le r-1$, $Q_{i+1} = R_i$,
whence $\dst P^\omb = \sum_{i=1}^r (-1)^i Q_i + \sum_{i=0}^{r-1} (-1)^i Q_{i+1} = 0$.

%%%%%%%%%%%%%%%%%%%%%%%%%%%%%%%%%%%%%%%%%%%%%%%%%%%%%%%%%%%%%%%%%%%%%%%%%
%%%%%%%%%%%%%%%%%%%%%%%%%%%%%%%%%%%%%%%%%%%%%%%%%%%%%%%%%%%%%%%%%%%%%%%%%

\parag 
\emph{Remark on Lemma~\ref{lemxitenstau}.}
Although this will not be used in the rest of the article, it is worth noting
here that the structure we have on~$\hM^\bul(\Om,\bA)$ is very reminiscent of
that of a {\em cocommutative Hopf algebra}:
the algebra structure is given by mould
multiplication~\eqref{eqdefmultiplimould}, with its unit~$1^\bul$;
as for the cocommutative cogebra structure, we may think of the map $\eps \colon
M^\bul\mapsto M^\est$ as of a {\em counit} and of the homomorphism~$\tau$ as of a kind
of {\em coproduct} (although its range is not exactly
$\hM^\bul(\Om,\bA)\otimes_\bA\hM^\bul(\Om,\bA)$);
we now may consider that the involution $S \colon M^\bul \mapsto \wt M^\bul$ behaves as
an {\em antipode}.

Indeed, the identity\footnote{%
derived from the obvious relation 
$\tsh{\est}{\alb}{\omb} =\tsh{\alb}{\est}{\omb} = 1_{\{\omb=\alb\}}$.
}
$\tau(M^\bul)^{\est,\alb} = \tau(M^\bul)^{\alb,\est} = M^\alb$
can be interpreted as a counit-like property for~$\eps$
and the fact that any dimould in the image of~$\rho$ is symmetric (consequence
of~\eqref{eqcommsh}) as a cocommutativity-like property, in the sense that
$\tau(M^\bul) = \sum P_i^\bul\otimes Q_i^\bul$ implies
$\sum \eps(P_i^\bul) Q_i^\bul = \sum \eps(Q_i^\bul) P_i^\bul = M^\bul$
and $\sum P_i^\bul\otimes Q_i^\bul = \sum Q_i^\bul\otimes P_i^\bul$.
The analogue of coassociativity for~$\tau$ is obtained by considering the
maps~$\taul$ and~$\taur$ which associate with any dimould~$\bM^\cbul$ the ``trimoulds''
$\bP^\ccbul = \taul(\bM^\cbul)$ and $\bQ^\ccbul = \taur(\bM^\cbul)$
defined by
$$
\bP^{\alb,\beb,\gab} = \sum_{\etab\in\Om^\bul} 
\tsh{\alb}{\beb}{\etab} \bM^{\etab,\gab},
\quad
\bQ^{\alb,\beb,\gab} = \sum_{\etab\in\Om^\bul} 
\tsh{\beb}{\gab}{\etab} \bM^{\alb,\etab}
$$
and by observing\footnote{%
Proof: for a mould~$M^\bul$, we have
$\taul\circ\tau(M^\bul)^{\alb,\beb,\gab} = \sum_\omb \tssh M^\omb$
with $\tssh = \sum_{\etab} \tsh{\alb}{\beb}{\etab} \tsh{\etab}{\gab}{\omb}$
coinciding with $\sum_{\etab} \tsh{\alb}{\etab}{\omb} \tsh{\beb}{\gab}{\etab}$,
hence $\taur\circ\tau(M^\bul)^{\alb,\beb,\gab} = \sum_\omb \tssh M^\omb$ as well.}
that $\taul\circ\tau = \taur\circ\tau$: when
$\tau(M^\bul) = \sum_i P_i^\bul\otimes Q_i^\bul$ 
with $\tau(P_i^\bul) = \sum_j A_{i,j}^\bul\otimes B_{i,j}^\bul$ 
and $\tau(Q_i^\bul) = \sum_k C_{i,k}^\bul\otimes D_{i,k}^\bul$,
this yields
$$
\sum_{i,j} A_{i,j}^\bul\otimes B_{i,j}^\bul\otimes Q_i^\bul =
\sum_{i,k} P_i^\bul \otimes C_{i,k}^\bul\otimes D_{i,k}^\bul.
$$
Finally, the compatibility of $\eps$, $\tau$ and~$S$ is expressed through formulas
\eqref{eqxitens}--\eqref{eqxitau}
(complemented by relations $\xi'(M^\bul\otimes N^\bul) = M^\bul  \times 
S N^\bul$ and $\xi'\circ \tau(M^\bul) = M^\est \, 1^\bul$
involving a map~$\xi'$ defined by replacing $(-1)^{r(\alb)} \bM^{\wt\alb,\beb}$
with $(-1)^{r(\beb)} \bM^{\alb,\wt\beb}$ in~\eqref{eqdefxiPbM});
therefore
$$
\tau(M^\bul) = \sum_i P_i^\bul\otimes Q_i^\bul
\ens\Rightarrow\ens
\sum S P_i^\bul \times Q_i^\bul = M^\est\,1^\bul
= \sum P_i^\bul \times S Q_i^\bul.
$$

When $\bA$ is a field, we get a true cocommutative Hopf algebra (graded by~$\ord$) by
considering
$\gH^\bul(\Om,\bA) = \tau\ii(\gB)$ with $\gB = \hM^\bul(\Om,\bA)\otimes_\bA\hM^\bul(\Om,\bA)$
(we can view $\gB$ as a subalgebra of $\hM^\bbul(\Om,\bA)$ according to the
remark on Definition~\ref{defidec}).
Indeed, in view of the above, it suffices essentially to check that
$M^\bul\in\gH=\gH^\bul(\Om,\bA)$ implies $\tau(M^\bul) \in \gH\otimes_\bA\gH$ (and not only
$\tau(M^\bul) \in \gB$), so that the restriction of the homomorphism~$\tau$
to~$\gH$ is a {\em bona fide} coproduct
$$
\De \colon \gH \to \gH \otimes_\bA \gH.
$$
This can be done by choosing a minimal~$N$ such that $\tau(M^\bul)$ can be written
as a sum of $N$ decomposable dimoulds:
$\tau(M^\bul) = \sum_{i=1}^N P_i^\bul \otimes  Q_i^\bul$ then implies that the
$Q_i^\bul$'s are linearly independent over~$\bA$ and the coassociativity property
allows one to show that each~$P_i^\bul$ lies in~$\gH$
(choose a basis of~$\hM^\bul(\Om,\bA)$, the first $N$ vectors of which are
$Q_1^\bul,\dotsc,Q_N^\bul$, and call $\xi_1,\dotsc,\xi_N$ the first $N$ covectors
of the dual basis: the coassociativity identity can be written
$\sum_i \tau(P_i^\bul)^{\alb,\beb}  Q_i^\gab = 
\sum_j P_i^\alb \tau(Q_j^\bul)^{\beb,\gab}$,
thus $\tau(P_i^\bul) = \sum_j P_j^\bul  \otimes N_{i,j}^\bul$ with
$N_{i,j}^\beb = \xi_i\big(\tau(Q_j^\bul)^{\beb,\bul}\big)$,
hence $P_i^\bul\in\gH$); similarly each~$Q_i^\bul$ lies in~$\gH$.

By definition, all the alternal and symmetral moulds belong to this Hopf
algebra~$\gH$, in which they appear respectively as {\em primitive} and {\em group-like}
elements. 

Finally, when $\bA$ is only supposed to be an integral domain,
$\hM^\bul(\Om,\bA)$ can be viewed as a subalgebra of~$\hM^\bul(\Om,K)$, where $K$
denotes the fraction field of~$\bA$; the $\bA$-valued alternal and symmetral moulds
belong to the corresponding Hopf algebra $\gH^\bul(\Om,K)$.

%%%%%%%%%%%%%%%%%%%%%%%%%%%%%%%%%%%%%%%%%%%%%%%%%%%%%%%%%%%%%%%%%%%%%%%%%
%%%%%%%%%%%%%%%%%%%%%%%%%%%%%%%%%%%%%%%%%%%%%%%%%%%%%%%%%%%%%%%%%%%%%%%%%

\parag 
\emph{Proof of Proposition~\ref{propcomposalt}.}
The structure of commutative semigroup on~$\Om$ allows us to define a
composition involving a dimould and a mould as follows:
$\bC^\cbul = \bM^\cbul \circ U^\bul$ if, for all $\alb,\beb\in\Om^\bul$,
$$
\bC^{\alb,\beb} = \sum
\bM^{(\norm{\alb^1},\dotsc,\norm{\alb^s}),(\norm{\beb^1},\dotsc,\norm{\beb^t})} 
U^{\alb^1}  \dotsm U^{\alb^s} U^{\beb^1}  \dotsm U^{\beb^t},
$$
with summation over all possible decompositions of $\alb$ and~$\beb$ into
non-empty words; when $\alb$ is the empty word, the convention is to
replace $(\norm{\alb^1},\dotsc,\norm{\alb^s})$ by~$\est$ and 
$U^{\alb^1}  \dotsm U^{\alb^s}$ by~$1$, and similarly when $\beb$ is
the empty word.

One can check that $\bM^\cbul \circ I^\bul = \bM^\cbul$ and
$\bM^\cbul \circ (U^\bul  \circ V^\bul) =
(\bM^\cbul  \circ U^\bul) \circ V^\bul$
for any dimould~$\bM^\cbul$ and any two moulds $U^\bul,V^\bul$ (by the same
argument as for the associativity of mould composition).

%%%%%%%%%%%%%%%%%%%%%%%%%%%%%%%%%%%%%%%%%%%%%%%%%%%%%%%%%%%%%%%%%%%%%%%%%

Proposition~\ref{propcomposalt} will follow fom
\begin{lemma}
For any three moulds $M^\bul,N^\bul,U^\bul$,
\begin{equation}	\label{eqcomposotimes}
(M^\bul\otimes N^\bul)\circ U^\bul = (M^\bul\circ U^\bul) \otimes (N^\bul\circ U^\bul).
\end{equation}
For any two moulds $M^\bul,U^\bul$,
\begin{equation}	\label{eqcompostaualt}
U^\bul \ens\text{alternal} \quad\Rightarrow\quad
\tau(M^\bul\circ U^\bul) = \tau(M^\bul) \circ U^\bul.
\end{equation}
\end{lemma}

\begin{proof}
The identity~\eqref{eqcomposotimes} is an easy consequence of the definition of
mould composition in Section~\ref{seccomposmoulds}.
As for~\eqref{eqcompostaualt}, let us suppose $U^\bul$ alternal and let 
$\bM^\cbul = \tau(M^\bul)$, $\bU^\cbul=\tau(U^\bul)$, $\bC^\cbul = \tau(M^\bul\circ U^\bul)$.
We have $\bC^{\est,\est} = M^\est = \bM^{\est,\est}$, as desired. Suppose now $\alb$ or
$\beb \neq \est$, then
$$
\bC^{\alb,\beb} = \sum_{s\ge1, \, \gab^1,\dotsc,\gab^s\neq\est}
\sh{\alb}{\beb}{\gab^1\concsm\gab^s} 
M^{(\norm{\gab^1},\dotsc,\norm{\gab^s})}
U^{\gab^1}  \dotsm U^{\gab^s}.
$$
Using the identity
(which is an easy generalisation of~\eqref{eqidentiteshdeux})
\begin{equation}	\label{eqidentiteshs}
\sh{\alb}{\beb}{\gab^1\concsm\gab^s} = 
\sum_{\alb=\alb^1\concsm\alb^s,\, \beb=\beb^1\concsm\beb^s}
\sh{\alb^1}{\beb^1}{\gab^1} \dotsm \sh{\alb^s}{\beb^s}{\gab^s},
\end{equation}
with possibly empty factors $\alb^i, \beb^i$,
we get
\begin{equation}	\label{eqsumbC}
\bC^{\alb,\beb} = \sum_{s\ge1, \, \alb=\alb^1\concsm\alb^s,\, \beb=\beb^1\concsm\beb^s}
M^{(\norm{\alb^1}+\norm{\beb^1},\dotsc,\norm{\alb^s}+\norm{\beb^s})}
\bU^{\alb^1,\beb^1}  \dotsm \bU^{\alb^s,\beb^s},
\end{equation}
with the convention $\norm{\est}=0$. Observe that this last summation involves
only finitely many nonzero terms because $\alb^i=\beb^i=\est$ implies $\bU^{\alb^i,\beb^i}=0$.

If $\alb$ or~$\beb$ is the empty word, since
$\bU^{\omb,\est} = \bU^{\est,\omb} = U^\omb$ we obtain that the values of $\bC^\cbul$
and $\bM^\cbul\circ U^\bul$ at $(\alb,\beb)$ coincide.
If neither $\alb$ nor~$\beb$ is empty, then we have moreover
$\bU^{\alb^i,\beb^i}\neq0 \;\Rightarrow\; \alb^i \ \text{or}\ \beb^i=\est$, thus
\eqref{eqsumbC} can be rewritten (retaining only non-empty factors)
$$
\bC^{\alb,\beb} = \sum \sum_{\omb\in\Om^\bul}
\tsh{(\norm{\alb^1},\dotsc,\norm{\alb^s})}{(\norm{\beb^1},\dotsc,\norm{\beb^t})}{\omb}
M^{\omb}
U^{\alb^1}  \dotsm U^{\alb^s} U^{\beb^1}  \dotsm U^{\beb^t},
$$
with the first summation over all possible decompositions of $\alb$ and~$\beb$ into
non-empty words. 
We thus get the desired result.
\end{proof}

%%%%%%%%%%%%%%%%%%%%%%%%%%%%%%%%%%%%%%%%%%%%%%%%%%%%%%%%%%%%%%%%%%%%%%%%%
%
\noindent
\emph{End of the proof of Proposition~\ref{propcomposalt}.}
We now suppose that $U^\bul$ is an alternal mould.
If $M^\bul$ is an alternal mould, then 
$$
\tau(M^\bul\circ U^\bul) = (M^\bul\otimes 1^\bul + 1^\bul\otimes M^\bul)\circ
U^\bul
= (M^\bul\circ U^\bul)\otimes 1^\bul + 1^\bul \otimes (M^\bul\circ U^\bul)
$$
by \eqref{eqcomposotimes}--\eqref{eqcompostaualt}, while, for $M^\bul$ symmetral,
$$
\tau(M^\bul\circ U^\bul) = (M^\bul\otimes M^\bul)\circ U^\bul
= (M^\bul\circ U^\bul)\otimes(M^\bul\circ U^\bul).
$$
Finally, if moreover $U^\bul$ is invertible for composition and $V^\bul =
(U^\bul)\iic$, then
$\tau(V^\bul) = \tau(V^\bul) \circ (U^\bul\circ V^\bul)
= \big( \tau(V^\bul) \circ U^\bul \big)\circ V^\bul
= \tau(V^\bul\circ U^\bul) \circ V^\bul
= (I^\bul\otimes 1^\bul + 1^\bul\otimes I^\bul) \circ V^\bul
= V^\bul \otimes 1^\bul + 1^\bul \otimes V^\bul$.

%%%%%%%%%%%%%%%%%%%%%%%%%%%%%%%%%%%%%%%%%%%%%%%%%%%%%%%%%%%%%%%%%%%%%%%%%
%%%%%%%%%%%%%%%%%%%%%%%%%%%%%%%%%%%%%%%%%%%%%%%%%%%%%%%%%%%%%%%%%%%%%%%%%

\parag 
\emph{Proof of Proposition~\ref{propDerivSym}.}
Let $M^\bul$ be a symmetral mould.
If $D$ is induced by a derivation $d\colon \bA\to\bA$, then we can apply~$d$ to
both sides of equation~\eqref{eqdefsymal} and we get
\begin{equation}	\label{eqtauDM}
\tau(D M^\bul) = D M^\bul \otimes M^\bul + M^\bul \otimes D M^\bul.
\end{equation}
Let us show that the same relation holds when $D=\na_{J^\bul}$ with $J^\bul$
alternal (this includes the case $D=D_\ph$).
We set $C^\bul = \na_{J^\bul} M^\bul$ and denote respectively by $\bJ^\cbul, \bM^\cbul,
\bC^\cbul$ the images of $J^\bul, M^\bul, C^\bul$ by the homomorphism~$\tau$.
We first observe that  $C^\est = 0$ and $\bC^{\est,\est} = 0$.
Let $\omb^1,\omb^2\in\Om^\bul$ with at least one of them non-empty.
From the definition of~$C^\bul$, we have
\begin{multline*}
\bC^{\omb^1,\omb^2} = \sum_{\alb,\beb,\gab,\, \beb\neq\est} 
\tsh{\omb^1}{\omb^2}{\alb\conc\beb\conc\gab}
M^{\alb\conc\norm{\beb}\conc\gab} J^\beb \\[1ex]
= \sum_{ \substack{ \omb^1 = \alb^1 \conc \beb^1 \conc \gab^1 \\
\omb^2 = \alb^2 \conc \beb^2 \conc \gab^2 } } \ 
\sum_{ \substack{ \alb,\gab \\ \beb\neq\est } } \ 
\tsh{\alb^1}{\alb^2}{\alb} \tsh{\gab^1}{\gab^2}{\gab} 
M^{\alb \conc ( \norm{\beb^1}+\norm{\beb^2} ) \conc\gab}
\tsh{\beb^1}{\beb^2}{\beb} J^\beb
\end{multline*}
by virtue of~\eqref{eqidentiteshs} with $s=3$.
The summation over~$\beb$ leads to the appearance of the factor
$\bJ^{\beb^1,\beb^2}$. By alternality of~$J^\bul$, this factor vanishes if
both~$\beb^1$ and~$\beb^2$ are non-empty, thus
\begin{multline*}
\bC^{\omb^1,\omb^2} = \sum_{
\omb^1 = \alb^1 \conc \beb^1 \conc \gab^1,\, \beb^1\neq\est}
\Phi_1(\alb^1,\norm{\beb^1},\gab^1;\omb^2) J^{\beb^1} \\
+ \sum_{
\omb^2 = \alb^2 \conc \beb^2 \conc \gab^2,\, \beb^2\neq\est}
\Phi_2(\omb^1;\alb^2,\norm{\beb^2},\gab^2) J^{\beb^2},
\end{multline*}
with $\dst \Phi_1(\alb^1,b,\gab^1;\omb^2) = 
\sum_{\alb,\gab,\ \omb^2=\alb^2\conc\gab^2} 
\tsh{\alb^1}{\alb^2}{\alb} \tsh{\gab^1}{\gab^2}{\gab} 
M^{\alb \conc b \conc\gab}$
and a symmetric definition for~$\Phi_2$.
A moment of thought shows that 
$$
\Phi_1(\alb^1,b,\gab^1;\omb^2) = 
\sum_\omb \sh{\alb^1\conc b\conc \gab^1}{\omb^2}{\omb} M^\omb
= \bM^{\alb^1\conc b\conc \gab^1,\omb^2},
$$
with a symmetric formula for~$\Phi_2$, so that
$$
\bC^{\omb^1,\omb^2}  = \sum_{\omb^1 = \alb\conc\beb\conc\gab}
\bM^{\alb\conc\norm{\beb}\conc\gab,\omb^2} U^\beb
+ \sum_{\omb^2 = \alb\conc\beb\conc\gab}
M^{\omb^1,\alb\conc\norm{\beb}\conc\gab} U^\beb,
$$
whence formula~\eqref{eqtauDM} follows.

Since the multiplicative inverse $\wt M^\bul$ of~$M^\bul$ is known
to be symmetral by Proposition~\ref{propstructaltsym},
we can multiply both sides of~\eqref{eqtauDM} by $\tau(\wt M^\bul)$
and use Lemma~\ref{lemhomomtau} and formula~\eqref{eqrhohomomAalg}; 
this yields the symmetrality of $D M^\bul \times \wt M^\bul$ and $\wt M^\bul \times D M^\bul$.

%%%%%%%%%%%%%%%%%%%%%%%%%%%%%%%%%%%%%%%%%%%%%%%%%%%%%%%%%%%%%%%%%%%%%%%%%
%%%%%%%%%%%%%%%%%%%%%%%%%%%%%%%%%%%%%%%%%%%%%%%%%%%%%%%%%%%%%%%%%%%%%%%%%

\parag 
There is a kind of converse to Proposition~\ref{propDerivSym}, which is
essential in the application to the saddle-node; we state it in this context only:

\begin{prop}	\label{propcVsym}
Let $\Om = \cN$ as in~\eqref{eqdefiancN} and $\bA=\C[[x]]$.
Then the mould~$\cV^\bul$ defined by Lemma~\ref{lemdefcV} is symmetral.
\end{prop}

\begin{proof}
We must show
\begin{equation}	\label{eqVsymal}
\cV^\alb \cV^\beb = \sum_{\gab\in\Om^\bul} \sh{\alb}{\beb}{\gab} \cV^\gab,
\qquad \alb,\beb\in\Om^\bul.
\end{equation}
Since $\cV^\est = 1$, this is obviously true for $\alb$ or $\beb = \est$.
We now argue by induction on $r = r(\alb) + r(\beb)$.
We thus suppose $r\ge1$ and, without loss of generality, both of $\alb$ and
$\beb$ non-empty.
With the notations $d = x^2\frac{\dd\,}{\dd x}$,
$\norm{\alb}=\al_1+\dotsb+\al_{r(\alb)}$
and $\norm{\beb}=\be_1+\dotsb+\be_{r(\beb)}$,
we compute
\begin{multline*}
A := (d+\norm{\alb}+\norm{\beb}) \sum_\gab \tsh{\alb}{\beb}{\gab} \cV^\gab \\
= \sum_{\gab\neq\est} \tsh{\alb}{\beb}{\gab} (d+\norm{\gab}) \cV^\gab 
= \sum_{\gab\neq\est} \tsh{\alb}{\beb}{\gab} a_{\ga_1} \cV^{`\gab},
\end{multline*}
using the notation $`\omb = (\om_2,\dotsc,\om_s)$ for any non-empty
$\omb = (\om_1,\dotsc,\om_s)$ and the defining equation of~$\cV^\bul$.
Splitting the last summation according to the value of~$\ga_1$, we get
$$
A = \sum_\deb \tsh{`\alb}{\beb}{\deb} a_{\al_1}  \cV^\deb
+ \sum_\deb \tsh{\alb}{`\beb}{\deb} a_{\be_1}  \cV^\deb
= a_{\al_1} \cV^{`\alb}  \cdot \cV^\beb 
+ \cV^{\alb}  \cdot a_{\be_1} \cV^{`\beb}
$$
(using the induction hypothesis), hence
$$
A = (d+\norm{\alb})\cV^\alb \cdot \cV^\beb 
+ \cV^\alb \cdot (d+\norm{\beb})\cV^\beb
= (d+\norm{\alb}+\norm{\beb}) (\cV^\alb \cV^\beb).
$$
We conclude that both sides of~\eqref{eqVsymal} must coincide, because
$d+\norm{\alb}+\norm{\beb}$ is invertible if $\norm{\alb}+\norm{\beb}\neq0$
and both of them belong to $x\C[[x]]$, thus even if $\norm{\alb}+\norm{\beb}=0$
the desired conclusion holds.
\end{proof}

%%%%%%%%%%%%%%%%%%%%%%%%%%%%%%%%%%%%%%%%%%%%%%%%%%%%%%%%%%%%%%%%%%%%%%%%%
%%%%%%%%%%%%%%%%%%%%%%%%%%%%%%%%%%%%%%%%%%%%%%%%%%%%%%%%%%%%%%%%%%%%%%%%%

\section{General mould-comould expansions}	\label{secGenMcM}

%%%%%%%%%%%%%%%%%%%%%%%%%%%%%%%%%%%%%%%%%%%%%%%%%%%%%%%%%%%%%%%%%%%%%%%%%
%%%%%%%%%%%%%%%%%%%%%%%%%%%%%%%%%%%%%%%%%%%%%%%%%%%%%%%%%%%%%%%%%%%%%%%%%

\parag 
We still assume that we are given a set~$\Om$ and a commutative $\C$-algebra~$\bA$.
When $\Om$ is the trivial one-element semigroup $\{0\}$, the algebra of
$\bA$-valued moulds on~$\Om$ is nothing but the algebra of formal series
$\bA[[\rmT]]$, with its usual multiplication and composition laws:
the monoid of words is then isomorphic to~$\N$ via the map~$r$, and one can identify a
mould~$M^\bul$ with the generating series $\sum_{\om\in\Om^\bul} M^\omb
\rmT^{r(\omb)}$; it is then easy to check that the above definitions of
multiplication and composition boil down to the usual ones.

In the case of a general set~$\Om$, the analogue of this is to identify a
mould~$M^\bul$ with the element 
$\sum M^{\om_1,\dotsc,\om_r} \rmT_{\om_1}\dotsm\rmT_{\om_r}$
of the completion
of the free associative (non-commutative) algebra generated by the symbols
$\rmT_\eta$, $\eta\in\Om$.
When replacing the $\rmT_\eta$'s by elements~$B_\eta$ of an $\bA$-algebra, one gets
what is called a mould-comould expansion; we now define these objects in a
context inspired by Section~\ref{secMCexpSN}.

%%%%%%%%%%%%%%%%%%%%%%%%%%%%%%%%%%%%%%%%%%%%%%%%%%%%%%%%%%%%%%%%%%%%%%%%%
%%%%%%%%%%%%%%%%%%%%%%%%%%%%%%%%%%%%%%%%%%%%%%%%%%%%%%%%%%%%%%%%%%%%%%%%%

\parag 
Suppose that $(\gF,\val)$ is a complete pseudovaluation ring, possibly
non-commutative, with unit denoted by $\ID$, such that~$\gF$ is also an
$\bA$-algebra.
We thus have a ring homomorphism $\mu\in\bA \mapsto \mu\ID\in\gF$, the image of
which lies in the center of~$\gF$.

\begin{definition}	\label{defcomouldmult}
A comould on~$\Om$ with values in~$\gF$ is any map 
$\bB_\bul \colon \omb\in\Om^\bul \mapsto \bB_\omb \in \gF$ such that 
$\bB_\est = \ID$ and
\begin{equation}	\label{eqcaraccomould}
\bB_{\omb^1\conc\omb^2} = \bB_{\omb^2}  \bB_{\omb^1},
\qquad \omb^1,\omb^2 \in \Om^\bul.
\end{equation}
\end{definition}

Such an object could even be called {\em multiplicative comould} to emphasize
that the map $\bB_\bul \colon \Om^\bul \to \gF$ is required to be a monoid
homomorphism from~$\Om^\bul$ to the multiplicative monoid underlying the
opposite ring of~$\gF$.

Observe that there is a one-to-one correspondence between comoulds and families 
$(B_\eta)_{\eta\in\Om}$ of~$\gF$ indexed by one-letter words:
the formulas $\bB_\est=\ID$ and $\bB_\omb = B_{\om_r}\dotsm B_{\om_1}$ for
$\omb = (\om_1,\dotsc,\om_r) \in\Om^\bul$ with $r\ge1$ define a comould,
which we call the {\em comould generated by $(B_\eta)_{\eta\in\Om}$}, and all
comoulds are obtained this way.

Suppose a comould $\bB_\bul$ is given.
For any $\bA$-valued mould~$M^\bul$ on~$\Om$ such that the family
$(M^\omb \bB_\omb)_{\omb\in\Om^\bul}$ is formally summable in~$\gF$ (in
particular this family has countable support---\cf
Definition~\ref{defformsumfam}), we can consider the mould-comould expansion,
also called {\em contraction of~$M^\bul$ into~$\bB_\bul$},
$$
\sum M^\bul \bB_\bul = \sum_{\omb\in\Om^\bul} M^\omb \bB_\omb \,\in\, \gF.
$$

%%%%%%%%%%%%%%%%%%%%%%%%%%%%%%%%%%%%%%%%%%%%%%%%%%%%%%%%%%%%%%%%%%%%%%%%%
%%%%%%%%%%%%%%%%%%%%%%%%%%%%%%%%%%%%%%%%%%%%%%%%%%%%%%%%%%%%%%%%%%%%%%%%%

\parag 
The example to keep in mind is related to Definition~\ref{defopval}. 
Suppose that $(\gA,\nu)$ is any complete pseudovaluation ring such that~$\gA$ is
a commutative $\bA$-algebra, the unit of which is denoted by~$1$; thus $\bA$ is
identified to a subalgebra of~$\gA$
(for instance $(\gA,\nu)=(\C[[x,y]],\nu_4)$ and $\bA=\C[[x]]$).
Denote by $\gE$ the subalgebra of $\End_\C(\gA)$ consisting of operators having
a valuation \wrt~$\nu$, so that $(\gE,\valn)$ is a complete pseudovaluation
ring.
Let
\begin{equation}	\label{eqdefgFexemp}
\gF_{\gA,\bA} = 
\ao \Th\in\gE \mid \text{$\Th$ and $\mu\ID$ commute for all $\mu\in\bA$\af}
= \gE \cap \End_\bA(\gA).
\end{equation}
We get an $\bA$-algebra, which is a closed subset of~$\gE$ for the topology
induced by~$\valn$, thus $(\gF_{\gA,\bA},\valn)$ is also a complete pseudovaluation ring;
these are the $\bA$-linear operators of~$\gA$ having a valuation \wrt~$\nu$.

In practice, the $B_\eta$'s which generate a comould are related
to the homogeneous components of an operator of~$\gA$ that one wishes to analyse.
In Section~\ref{secMCexpSN} for instance, the derivation $X-X_0$ of $\gA=\C[[x,y]]$
was decomposed into a sum of multiples of $B_n$ according to~\eqref{eqdefBn},
where each term $a_n(x)B_n$ is homogeneous of degree~$n$ in the sense that it sends
$y^{n_0}\C[[x]]$ in $y^{n_0+n}\C[[x]]$ for every~$n_0$.
Observe that the commutation of the $B_\eta$'s with the image of $\bA=\C[[x]]$ in~$\gE$
reflects the fact that the vector field $X-X_0$ is ``fibred'' over the
variable~$x$; 
similarly, one can look for a solution~$\Th$ of equation~\eqref{eqconjugOp}
in $\gF_{\gA,\bA}$ because the corresponding formal transformation $(x,y)\mapsto\th(x,y)$
is expected to be fibred likewise---\cf \eqref{eqdefthph}.

%%%%%%%%%%%%%%%%%%%%%%%%%%%%%%%%%%%%%%%%%%%%%%%%%%%%%%%%%%%%%%%%%%%%%%%%%
%%%%%%%%%%%%%%%%%%%%%%%%%%%%%%%%%%%%%%%%%%%%%%%%%%%%%%%%%%%%%%%%%%%%%%%%%

\parag 
Returning to the general situation, we now show how, via mould-comould
expansions, mould multiplication corresponds to multiplication in~$\gF$:
\begin{prop}	\label{propmultiplimould}
Suppose that $\bB_\bul$ is an $\gF$-valued comould on~$\Om$ and that $M^\bul$
and~$N^\bul$ are $\bA$-valued moulds on~$\Om$ such that the families
$(M^\omb \bB_\omb)_{\omb\in\Om^\bul}$ and $(N^\omb \bB_\omb)_{\omb\in\Om^\bul}$
are formally summable.
Then the mould $P^\bul = M^\bul  \times N^\bul$ gives rise to a
formally summable family $(P^\omb \bB_\omb)_{\omb\in\Om^\bul}$ and
$$
\sum \left(M^\bul  \times N^\bul \right) \bB_\bul = 
\left( \sum N^\bul \bB_\bul \right)\left( \sum M^\bul \bB_\bul \right).
$$
\end{prop}

\begin{proof}
Let $\de_*\in\Z$ such that 
$v_1(\omb) = \vl{M^\omb \bB_\omb} \ge \de_*$
and $v_2(\omb) = \vl{N^\omb \bB_\omb} \ge \de_*$
for all $\omb\in\Om^\bul$.
Then
$$
P^\omb \bB_\omb = \sum_{\omb = \omb^1\!\conc\omb^2} 
N^{\omb^2} \bB_{\omb^2}
M^{\omb^1} \bB_{\omb^1}
$$
(since $\bA$ is a commutative algebra and its image in~$\gF$ commutes with
the~$\bB_{\omb^2}$'s), thus
$\vl{P^\omb \bB_\omb} \ge  \min\ao v_1(\omb^1) + v_2(\omb^2) \mid 
\omb = \omb^1\!\conc\omb^2 \af  \ge 2\de_*$ 
and, for any $\de\in\Z$, the condition $\vl{P^\omb \bB_\omb} \le \de$ implies
that $\omb$ can be written as $\omb^1\!\conc\omb^2$ with 
$v_1(\omb^1)  \le \de-\de_*$ and $v_2(\omb^2)  \le \de-\de_*$, 
hence they are only finitely many such $\omb$'s.

To compute $\sum  P^\bul \bB_\bul$, we can suppose $\Om$ countable (replacing it,
if necessary, by the set of all letters appearing in the union of the supports
of $(M^\omb \bB_\omb)$ and $(N^\omb \bB_\omb)$, which is countable), 
choose an exhaustion of~$\Om$ by finite sets $\Om_K$, $K\ge0$, and use
$ \Om^{K,R} = \ao \omb  \in  \Om^\bul \mid r = r(\omb) \le R,\;
\om_1,\dotsc,\om_r \in \Om_K \af$, $K,R\ge0$, as an exhaustion of $\Om^\bul$.
The conclusion follows from the identity
\begin{multline*}
\bigg( \sum_{\omb\in\Om^{K,R}}  N^\omb \bB_\omb \bigg)
\bigg( \sum_{\omb\in\Om^{K,R}}  M^\omb \bB_\omb \bigg)
- \sum_{\omb\in\Om^{K,R}} P^\omb \bB_\omb \\ =
\sum_{ \substack{ \omb^1,\omb^2 \in\Om^{K,R} \\
r(\omb^1) + r(\omb^2) > R }} 
N^{\omb^2} \bB_{\omb^2}
M^{\omb^1} \bB_{\omb^1},
\end{multline*}
where the \rhs\ tends to~$0$ as $K,R\to\infty$, since its valuation is at least
$\min\ao v_1(\omb^1) + v_2(\omb^2) \mid 
\omb^1, \omb^2 \in \Om^{K,R},\, r(\omb^1)+r(\omb^2)>R \af
\ge \nu_*(K,R) + \de_*$,
with 
$$
\nu_*(K,R) = \min\ao \min(v_1(\omb), v_2(\omb)) \mid
\omb \in \Om^{K,R},\, r(\omb)>R/2 \af
\xrightarrow[R\to\infty]{}  \infty
$$
for any $K$ (because, for any finite subset~$F$ of~$\Om^\bul$, $\omb\notin F$ as
soon as $r(\omb)$ is large enough).
\end{proof}

%%%%%%%%%%%%%%%%%%%%%%%%%%%%%%%%%%%%%%%%%%%%%%%%%%%%%%%%%%%%%%%%%%%%%%%%%
%%%%%%%%%%%%%%%%%%%%%%%%%%%%%%%%%%%%%%%%%%%%%%%%%%%%%%%%%%%%%%%%%%%%%%%%%

\parag 
Suppose $\Om$ is a commutative semigroup.
A motivation for the definition of mould composition in Section~\ref{secAlgMoulds} is
\begin{prop}	\label{propcomposexp}
Suppose that $U^\bul$ and $M^\bul$ are moulds such that the families
$(U^\omb \bB_\omb)_{\omb\in\Om^\bul}$ and % is formally summable, 
$$
\Th_{\omb^1,\dotsc,\omb^s} = 
M^{\norm{\omb^1},\dotsc,\norm{\omb^s}} 
U^{\omb^1}  \dotsm U^{\omb^s}
\bB_{\omb^1\concsm\omb^s},
\qquad s\ge1,\; \omb^1,\dotsc,\omb^s \in \Om^\bul
$$ 
are formally summable.\footnote{
Notice that the formal summability of the second family follows from the formal
sumability of the first one when the valuation $\val$ on~$\gF$ only takes non-negative values.
}
Suppose moreover $U^\est=0$, let
\begin{equation}	\label{eqdefBprime}
B'_\eta = \sum_{\omb\in\Om^\bul \,\text{s.t.}\, \norm{\omb}=\eta}
U^\omb \bB_\omb, \qquad \eta\in\Om,
\end{equation}
and the consider the comould $\bB'_\bul$ generated by $\{B'_{\eta}, \, \eta\in\Om\}$. 
Then the mould $C^\bul = M^\bul \circ U^\bul$ gives rise to a
formally summable family $(C^\omb \bB_\omb)_{\omb\in\Om^\bul}$ and
$$
\sum \left(M^\bul  \circ U^\bul \right) \bB_\bul = 
\sum M^\bul \bB'_\bul.
$$
\end{prop}

\begin{proof}
We have $C^\est \bB_\est = M^\est \bB'_\est$, since $C^\est = M^\est$.
If $\omb$ and $\etab$ are non-empty words in~$\Om^\bul$, with
$\etab = (\eta_1,\dotsc,\eta_\sig)$,
$$
C^\omb \bB_\omb = 
\sum_{\substack{s\ge1,\,\omb^1,\dotsc,\omb^s \neq\est  \\
\omb = \omb^1 \concsm \omb^s} } 
\Th_{\omb^1,\dotsc,\omb^s}, \qquad
M^\etab \bB'_\etab = 
\sum_{ \substack{\omb^1,\dotsc,\omb^\sig \neq\est  \\
\norm{\omb^1}=\eta_1, \dotsc, \norm{\omb^\sig}=\eta_\sig} }
\Th_{\omb^1,\dotsc,\omb^\sig}.
$$
The conclusion follows easily.
\end{proof}

The idea is that, when indexation by $\eta\in\Om$ corresponds to a decomposition
of an element of~$\gF$ into homogeneous components, we use the mould~$U^\bul$ to
go from $X = \sum_{\eta\in\Om} B_\eta$ to 
$Y = \sum_{\eta\in\Om} B'_\eta$
by contracting it into the comould~$\bB_\bul$ associated with~$X$; then we
use~$M^\bul$ to go from~$Y$ to the contraction~$Z$ of~$M^\bul$ into the
comould~$\bB'_\bul$ associated with~$Y$.
Mould composition thus reflects the composition of these operations on elements
of~$\gF$, $X\mapsto Y$ and $Y\mapsto Z$.

%%%%%%%%%%%%%%%%%%%%%%%%%%%%%%%%%%%%%%%%%%%%%%%%%%%%%%%%%%%%%%%%%%%%%%%%%
%%%%%%%%%%%%%%%%%%%%%%%%%%%%%%%%%%%%%%%%%%%%%%%%%%%%%%%%%%%%%%%%%%%%%%%%%

\parag 
For example, suppose that $(\bB_\omb)_{\omb\in\Om^\bul}$ is formally
summable. Then, in particular, $(B_\eta)_{\eta\in\Om}$ is formally summable, and
$X=\sum B_\eta$ is ``exponentiable'': for any $t\in\C$, the series $\exp(tX) =
\sum_{s\ge0}\frac{t^s}{s!}X^s$ is convergent; moreover, $\exp(tX) =
\sum \exp_t^\bul \bB_\bul$. 
On the other hand, $\ID+X$ has an ``infinitesimal generator'': the series $Y =
\sum_{s\ge1} \frac{(-1)^{s-1}}{s}X^s$ is convergent and $\exp(Y)=\ID+X$; one has
$Y = \sum \log^\bul\bB_\bul$.

Now, if $U^\bul \in \gL^\bul(\Om,\bA)$ is such that $(U^{\omb^1}\dotsm
U^{\omb^s} \bB_{\omb^1\concsm\omb^s})_{s\ge1,\,\omb^1,\dotsc,\omb^s\in\Om^\bul}$
is formally summable, then in particular $(U^\omb \bB_\omb)_{\omb\in\Om^\bul}$ is formally
summable and $X' = \sum U^\bul\bB_\bul$ is exponentiable, with
$\exp(tX') = \sum (\exp_t^\bul \circ\, U^\bul)\bB_\bul$
for any $t\in\C$.

Similarly, if $M^\bul= 1^\bul+V^\bul \in G^\bul(\Om,\bA)$ with $(V^{\omb^1}\dotsm
V^{\omb^s} \bB_{\omb^1\concsm\omb^s})$ formally
summable, then $\sum M^\bul\bB_\bul$ has infinitesimal generator 
$\sum (\log^\bul \circ\, V^\bul)\bB_\bul$.

%%%%%%%%%%%%%%%%%%%%%%%%%%%%%%%%%%%%%%%%%%%%%%%%%%%%%%%%%%%%%%%%%%%%%%%%%
%%%%%%%%%%%%%%%%%%%%%%%%%%%%%%%%%%%%%%%%%%%%%%%%%%%%%%%%%%%%%%%%%%%%%%%%%

\parag
For the interpretation of the mould derivations $\na_{U^\bul}$
defined by~\eqref{eqdefnaUbul},
consider a situation similar to that of Proposition~\ref{propcomposexp}, with a
comould $\bB_\bul \colon \Om^\bul \to \gF$, a mould $U^\bul \in \gL^\bul(\Om,\bA)$ such that $(U^\omb
\bB_\omb)_{\omb\in\Om^\bul}$ is formally summable and, for each $\eta\in\Om$,
$B'_\eta \in \gF$ still defined by~\eqref{eqdefBprime}.
But instead of considering the comould~$\bB'_\bul$ generated by $(B'_\eta)$, 
\ie\ 
$\bB'_\omb = B'_{\om_r} \dotsm B'_{\om_1}$ for $\omb = (\om_1,\dotsc,\om_r)$,
set
$$
\bB'_\omb = \sum_{\omb = \alb\beb\gab,\, r(\beb)=1}
\bB_\gab B'_\beb \bB_\alb,
\qquad \omb \in \Om^\bul
$$
\ie\ $\bB'_\est = 0$ and
$\bB'_\omb = \sum_{i=1}^r B_{\om_r}\dotsm B_{\om_{i+1}} B'_{\om_i} 
B_{\om_{i-1}} \dotsm B_{\om_1}$ for $r\ge1$
(beware that $\bB'_\bul \colon \Om^\bul \to \gF$ is not a comould, since
multiplicativity fails).

Then one can check the formal summability 
of $\big( (\na_{U^\bul} M^\omb) \bB_\omb \big)_{\omb\in\Om^\bul}$
for any mould $M^\bul$ such that the families 
$(M^\omb \bB_\omb)_{\omb\in\Om^\bul}$ and
$(M^{\alb,\norm{\beb},\gab} U^\beb \bB_{\alb\conc\beb\conc\gab})_{\alb,\beb,\gab\in\Om^\bul}$ 
are formally summable, with
$$
\sum \left( \na_{U^\bul} M^\bul  \right) \bB_\bul = \sum M^\bul \bB'_\bul.
$$

If, moreover, there is an $\bA$-linear derivation $\gD \colon \gF \to \gF$ such
that $\gD B_\eta = B'_\eta$ for each $\eta\in\Om$, then $\bB'_\omb$ is nothing
but $\gD \bB_\omb$ and the previous identity takes the form
$$
\sum \left( \na_{U^\bul} M^\bul  \right) \bB_\bul
= \gD \left( \sum M^\bul \bB_\bul \right).
$$

%%%%%%%%%%%%%%%%%%%%%%%%%%%%%%%%%%%%%%%%%%%%%%%%%%%%%%%%%%%%%%%%%%%%%%%%%
%%%%%%%%%%%%%%%%%%%%%%%%%%%%%%%%%%%%%%%%%%%%%%%%%%%%%%%%%%%%%%%%%%%%%%%%%

\parag
For a given commutative algebra~$\bA$, we now consider the case where
$$
\Om \subset \Z^n, \quad \gA = \bA[[y_1,\ldots,y_n]],
$$
for a fixed $n\in\N^*$.

\begin{definition}	\label{defhomog}
Given $\eta\in\Z^n$ and $\Th\in\End_\bA(\bA[[y_1,\dotsc,y_n]])$, we say that $\Th$ is homogeneous of
degree~$\eta$ if $\Th y^m \in \bA y^{m+\eta}$ for every $m\in\N^n$ 
(with the usual notation $y^m=y_1^{m_1}\dotsm y_n^{m_n}$ for monomials).
\end{definition}

For example, any $\la\in\bA^n$ gives rise to an operator
\begin{equation}	\label{eqdefgXla}
\gX_\la = \la_1 y_1 \frac{\pa\;}{\pa y_1} +
\dotsb + \la_n y_n \frac{\pa\;}{\pa y_n}
\end{equation}
which is homogeneous of degree~$0$, since
$\gX_\la y^m = \lan m, \la\ran y^m$.

Suppose moreover that we are given a pseudovaluation $\val \colon \gA \to
\Z\cup\{\infty\}$ such that $(\gA,\val)$ is complete and 
$\frac{\pa\;}{\pa y_1},\dotsc,\frac{\pa\;}{\pa y_n}$ are continuous,
and consider $\gF = \gF_{\gA,\bA}$ as defined by~\eqref{eqdefgFexemp}.
We suppose $\Om\subset\Z^n$ because we are interested in $\gF$-valued {\em
homogeneous} comoulds, \ie\ $\gF$-valued comoulds~$\bB_\bul$ such that
$\bB_\omb$ is homogeneous of degree
$\norm{\omb} = \om_1 + \dotsb + \om_r \in \Z^n$
for every non-empty $\omb\in\Om^\bul$,
and homogeneous of degree~$0$ for $\omb=\est$; 
in fact, the multiplicativity property~\eqref{eqcaraccomould} will not be used
for what follows, the following proposition holds for any map $\bB_\bul\colon
\Om^\bul \to \gF$ provided it is homogeneous as just defined.

In the case of a comould satisfying the multiplicativity property as required in
Definition~\ref{defcomouldmult}, homogeneity is equivalent to the fact that each
$B_\eta = \bB_{(\eta)}$, $\eta\in\Om$, is homogeneous of degree~$\eta$.

\begin{prop}	\label{propcommutXla}
Let $\la\in\bA^n$ and $\bB_\bul$ be an $\gF$-valued homogeneous comould.
Then, for every $\bA$-valued mould~$M^\bul$ such that $(M^\omb
\bB_\omb)_{\omb\in\Om^\bul}$ is formally summable,
\begin{equation}	\label{eqXlaDph}
\big[ \gX_\la, \sum M^\bul \bB_\bul \big] = \sum (D_\ph M^\bul) \bB_\bul, 
\qquad \text{with} \ens 
\ph = \lan \, \cdot \,,\la \ran \colon \Om \to \bA.
\end{equation}
\end{prop}

Thus, this mould derivation~$D_\ph$ reflects the action of the derivation $\ad_{\gX_\la}$ of $\End_\bA(\gA)$.

\begin{proof}
We first check that, if $\Th\in\End_\bA(\bA[[y_1,\dotsc,y_n]])$ is homogeneous of
degree $\eta\in\Z^n$, then
$$
[ \gX_\la, \Th ] = \lan \eta,\la \ran \Th.
$$
By $\bA$-linearity and continuity, it is sufficient to check that both operators
act the same way on a monomial~$y^m$.
We have $\Th y^m = \be_{m} y^{m+\eta}$ with a $\be_{m} \in \bA$,
thus $\gX_\la \Th y^m = \lan m+\eta,\la \ran \be_{m}  y^{m+\eta} 
= \lan m+\eta,\la \ran \Th y^m$ 
while $\Th \gX_\la y^m = \lan m,\la \ran \Th y^m$, hence
$[ \gX_\la, \Th ] y^m = \lan\eta,\la\ran \Th y^m$ as required.

It follows that
$$
[ \gX_\la, \bB_\omb ] = \big(\ph(\om_1)+\dotsb+\ph(\om_r)\big) \bB_\omb, \qquad
\omb = (\om_1,\dotsc,\om_r) \in \Om^\bul.
$$
Let $N^\bul = D_\ph M^\bul$.
For any exhaustion of~$\Om^\bul$ by finite sets $I_k$, letting 
$\Th_k = \sum_{\omb\in I_k} M^\omb \bB_\omb$ and 
$\Th'_k = \sum_{\omb\in I_k} N^\omb \bB_\omb$, we get
$[\gX_\la,\Th_k] = \Th'_k$.
For every $f\in\gA$, we have $\Th'_k f \xrightarrow[k\to\infty]{}\sum N^\bul 
\bB_\bul f$ on the one hand,
while, by continuity of~$\gX_\la$, 
$[\gX_\la,\Th_k]f \xrightarrow[k\to\infty]{} \big[\gX_\la,\sum M^\bul \bB_\bul\big]f$
on the other hand.
\end{proof}

%%%%%%%%%%%%%%%%%%%%%%%%%%%%%%%%%%%%%%%%%%%%%%%%%%%%%%%%%%%%%%%%%%%%%%%%%
%%%%%%%%%%%%%%%%%%%%%%%%%%%%%%%%%%%%%%%%%%%%%%%%%%%%%%%%%%%%%%%%%%%%%%%%%

\parag
Notice that, in the above situation, 
any $\C$-linear derivation $d  \colon \bA \to \bA$ induces a derivation 
$\ti d \colon \gA \to \gA$ (defined by $\ti d \sum a_m y^m = \sum (d a_m) y^m$) and a mould
derivation~$D$ (defined just before Remark~\ref{remmouldVsol}). 
If $\ti d$ commutes with the $B_\eta$, $\eta\in\Om$, one easily gets
\begin{equation}	\label{eqcommuttid}
\big[ \ti d, \sum M^\bul \bB_\bul \big] = \sum (D M^\bul) \bB_\bul.
\end{equation}
On the other hand, $D_\ph M^\omb = \lan \norm{\omb},\la \ran M^\omb$ if $\ph = \lan \, \cdot,\la \ran$.
Thus $D_\ph = \na$ when $n=1$ and $\la=1$.
This is the relevant situation for the saddle-node:
\begin{cor}	\label{corEC}
Choose $\Om = \cN$ as in~\eqref{eqdefiancN}, $\bA = \C[[x]]$ and 
$(\gA,\val) = (\bA[[y]],\nu_4)$, $\gF = \gF_{\gA,\bA}$.
Let $\bB_\bul$ denote the $\gF$-valued comould generated by $B_\eta =
y^{\eta+1}\frac{\pa\,}{\pa y}$.
Let $(a_\eta)_{\eta\in\Om}$ be as in~\eqref{eqassan} and 
$$
X_0 = x^2 \frac{\pa\,}{\pa x} + y \frac{\pa\,}{\pa y},
\qquad
X = X_0 + \sum_{\eta\in\Om} a_\eta B_\eta.
$$
Then the mould-comould contraction $\Th=\sum \cV^\bul \bB_\bul \in \gF$,
where~$\cV^\bul$ is determined by Lemma~\ref{lemdefcV}, is solution of the
conjugacy equation~\eqref{eqconjugOp} $\Th X = X_0  \Th$ in $\End_\C(\gA)$.
\end{cor}

\begin{proof}
It was already observed that each $B_\eta$ is homogeneous of degree~$\eta$
and the formal summability of $(\cV^\omb \bB_\omb)_{\omb\in\Om^\bul}$ 
was checked in Lemma~\ref{lemCVformal}.
Let $d = x^2\frac{\dd\,}{\dd x} \colon \bA \to \bA$. 
The corresponding derivation of~$\gA$ is $\ti d = x^2\frac{\pa\,}{\pa x}$, which
commutes with the $B_\eta$'s.
On the other hand, with the notation of Proposition~\ref{propcommutXla},
$X_0 = \ti d + \gX_1$.
Since $X-X_0 = \sum J_a^\bul \bB_\bul$ with the notation of Remark~\ref{remmouldVsol},
equation~\eqref{eqconjugOp} is equivalent to 
$\big[\ti d + \gX_1,\Th\big] = \Th  \sum J_a^\bul \bB_\bul$;
plugging any formally convergent mould-comould expansion $\Th = \sum M^\bul
\bB_\bul$ into it, we find $\sum (D M^\bul+\na M^\bul)  \bB_\bul$ for the \lhs\
by~\eqref{eqXlaDph} and~\eqref{eqcommuttid} while, according to
Proposition~\ref{propmultiplimould}, the \rhs\ can be written $\sum
(J_a^\bul\times M^\bul) \bB_\bul$, hence the conclusion follows
from~\eqref{eqmouldeq}.
\end{proof}

\begin{remark}	\label{remcommentpf}
The symmetrality of the mould~$\cV^\bul$ obtained in Proposition~\ref{propcVsym}
shows us that $\Th$ is invertible, with inverse
$\Th\ii = \sum \cVt^\bul \bB_\bul$.
The proof of Theorem~\ref{thmSNformal} will thus be complete when we have
checked that $\Th$ is an algebra automorphism; this will follow from the results
of next section on the contraction of symmetral moulds into a comould generated
by derivations.
\end{remark}

%%%%%%%%%%%%%%%%%%%%%%%%%%%%%%%%%%%%%%%%%%%%%%%%%%%%%%%%%%%%%%%%%%%%%%%%%
%%%%%%%%%%%%%%%%%%%%%%%%%%%%%%%%%%%%%%%%%%%%%%%%%%%%%%%%%%%%%%%%%%%%%%%%%

\section{Contraction into a cosymmetral comould}	\label{secContrAltSym}

%%%%%%%%%%%%%%%%%%%%%%%%%%%%%%%%%%%%%%%%%%%%%%%%%%%%%%%%%%%%%%%%%%%%%%%%%
%%%%%%%%%%%%%%%%%%%%%%%%%%%%%%%%%%%%%%%%%%%%%%%%%%%%%%%%%%%%%%%%%%%%%%%%%

\parag 
For the interpretation of alternality and symmetrality of moulds in terms
of the corresponding mould-comould expansions, we focus on the case where the
comould~$\bB_\bul$ is generated by a family of $\bA$-linear derivations
$(B_\eta)_{\eta\in\Om}$ of a commutative algebra~$\gA$.

The main result of this section is Proposition~\ref{propaltsym} below, according to
which, in this case, 
\emph{the contraction of an alternal mould into~$\bB_\bul$ gives rise to a
derivation}
and \emph{the contraction of a symmetral mould gives rise to an algebra
automorphism}.

%%%%%%%%%%%%%%%%%%%%%%%%%%%%%%%%%%%%%%%%%%%%%%%%%%%%%%%%%%%%%%%%%%%%%%%%%
%%%%%%%%%%%%%%%%%%%%%%%%%%%%%%%%%%%%%%%%%%%%%%%%%%%%%%%%%%%%%%%%%%%%%%%%%

\parag 
We thus assume that $(\gA,\nu)$ is a complete pseudovaluation ring such that~$\gA$ is
a commutative $\bA$-algebra, and we define $\gF = \gF_{\gA,\bA}$ by~\eqref{eqdefgFexemp}.
Since we shall be interested in the way the elements of~$\gF$ act on products of
elements of~$\gA$, we consider the left $\gF$-module $\Bili$ of $\bA$-bilinear maps
from $\gA\times\gA$ to~$\gA$ 
(with ring multiplication $(\Th,\Phi)\in\gF\times\Bili \mapsto \Th\circ\Phi\in\Bili$)
and its filtration
$$
\gB_\de = \ao \Phi \in  \Bili \mid \nu\big( \Phi(f,g) \big) \ge \nu(f)+\nu(g)+\de
\;\text{for all $f,g\in\gA$} \af, \qquad \de\in\Z.
$$
By defining $\gB = \bigcup_{\de\in\Z} \gB_\de$ we get a left $\gF$-submodule of
$\Bili$, for which the filtration $(\gB_\de)_{\de\in\Z}$ is exhaustive,
separated and compatible with the filtration of~$\gF$ induced by~$\valn$: the
subgroups $\gF_\de = \ao \Th\in\gF \mid \vln{\Th}\ge\de \af$ satisfy 
$\gF_\de \gB_{\de'} \subset \gB_{\de+\de'}$ for all $\de,\de'\in\Z$.
The corresponding distance on~$\gB$ is complete, by completeness of
$(\gA,\nu)$. 

We now define a map $\sig \colon \gF \to \gB$ by
$\Th\in\gF \mapsto \sig(\Th) = \Phi$ such that
$$
\Phi \colon
(f,g) \in \gA\times\gA \mapsto \Phi(f,g) = \Th(fg) \in \gA.
$$
This map is to be understood as a kind of coproduct.
Observe that $\sig$ is $\gF$-linear, \ie\ $\sig(\Th \Th')=\Th\sig(\Th')$
(thus it boilds down to $\sig(\Th)=\Th\circ\sig(\ID)$, and $\sig(\ID)$ is just the
multiplication of~$\gA$)
and continuous because $\sig(\gF_\de) \subset \gB_\de$ for each $\de\in\Z$.

Viewing $\gF$ as an $\bA$-module, we also define an $\bA$-linear map
$$
\rho \colon \gF_2 = \gF  \otimes_\bA \gF \to \gB
$$ 
by its action on decomposable elements:
$$
\rho(\Th_1\otimes\Th_2)(f,g) = (\Th_1 f)(\Th_2 g), \qquad f,g\in\gA
$$
for any $\Th_1,\Th_2\in\gF$. 
(A remark parallel to the remark on Definition~\ref{defidec} applies: the kernel
of~$\rho$ is the torsion submodule of~$\gF_2$ when $\gA$ is an integral
domain; if moreover $\bA$ is principal, then $\rho$ is injective.)
Notice that, for $\Th_1\in\gF_{\de_1}$ and $\Th_2\in\gF_{\de_2}$, one has
$\rho(\Th_1\otimes\Th_2) \in \gB_{\de_1+\de_2}$, 
hence the map $\ti\rho \colon (\Th_1,\Th_2) \mapsto \rho(\Th_1\otimes\Th_2)$
from $\gF\times\gF$ to~$\gB$ is continuous.

Using the $\bA$-algebra structure of~$\gF_2$, we see that
\begin{equation}	\label{eqCorrespAlg}
\sig(\Th) = \rho(\xi),\; \sig(\Th') = \rho(\xi')
\quad \Rightarrow  \quad
\sig(\Th\Th') = \rho(\xi\xi')
\end{equation}
for any $\Th,\Th'  \in \gF$, $\xi,\xi'\in\gF_2$.

%%%%%%%%%%%%%%%%%%%%%%%%%%%%%%%%%%%%%%%%%%%%%%%%%%%%%%%%%%%%%%%%%%%%%%%%%
%%%%%%%%%%%%%%%%%%%%%%%%%%%%%%%%%%%%%%%%%%%%%%%%%%%%%%%%%%%%%%%%%%%%%%%%%

\parag
With the above notations, the set of all $\bA$-linear derivations of~$\gA$
having a valuation is
$$
\gL_\gF = \ao \Th\in\gF \mid \sig(\Th) = \rho(\Th\otimes\ID + \ID\otimes\Th) \af
$$
(it is a Lie algebra for the bracketting $[\Th_1,\Th_2] = \Th_1\Th_2 - \Th_2\Th_1$).
Letting $\gF^*$ denote the multiplicative group of invertible elements of~$\gF$,
we may also consider its subgroup
$$
G_\gF = \ao \Th\in\gF^* \mid \sig(\Th) = \rho(\Th\otimes\Th) \af,
$$
the elements of which are $\bA$-linear algebra automorphisms of~$\gA$.

\begin{lemma}	\label{lemcosym}
Assume that the generators $B_\eta$, $\eta\in\Om$, of an $\gF$-valued
comould~$\bB_\bul$ all belong to~$\gL_\gF$. 
Then
\begin{equation}	\label{eqmotivsh}
\sig(\bB_\omb) = \sum_{\omb^1,\omb^2\in\Om^\bul} \sh{\omb^1}{\omb^2}{\omb}
\rho( \bB_{\omb^1} \otimes \bB_{\omb^2} ), \qquad \omb\in\Om^\bul.
\end{equation}
\end{lemma}

Such a comould is said to be {\em cosymmetral}.

\begin{proof}
Let $\omb=(\om_1,\ldots,\om_r)\in\Om^\bul$.
We proceed by induction on~$r$.
Equation~\eqref{eqmotivsh} holds if $r=0$, since $\sig(\ID)=\rho(\ID\otimes\ID)$,
or $r=1$ (by assumption); we thus suppose $r\ge2$.
By~\eqref{eqcaraccomould}, we can write $\bB_\omb = \bB_{`\omb} B_{\omb_1}$ with
$`\omb = (\om_2,\dotsc,\om_r)$.
Using the induction hypothesis and~\eqref{eqCorrespAlg}, we get
$\sig(\bB_\omb)=\rho(\xi)$ with
\begin{multline*}
\xi = \sum_{\alb,\beb\in\Om^\bul} \sh{\alb}{\beb}{`\omb} 
(\bB_\alb\otimes\bB_\beb)(B_{\om_1}\otimes\ID + \ID\otimes B_{\om_1}) \\
= \sum_{\alb,\beb\in\Om^\bul} \sh{\alb}{\beb}{`\omb} 
(\bB_{\om_1\conc\alb}\otimes\bB_\beb + \bB_\alb\otimes\bB_{\om_1\conc\beb}).
\end{multline*}
This coincides with 
$\dst\sum_{\alb,\beb\in\Om^\bul} \tsh{\alb}{\beb}{\omb} \bB_\alb\otimes\bB_\beb$,
since
$$
\sh{\alb}{\beb}{\omb} = \sh{`\alb}{\beb}{`\omb} 1_{\{\al_1=\om_1\}}
+ \sh{\alb}{`\beb}{`\omb} 1_{\{\be_1=\om_1\}}
$$
(particular case of~\eqref{eqidentiteshdeux} with $\gab^1=\om_1$ and $\gab^2=`\omb$).
\end{proof}

%%%%%%%%%%%%%%%%%%%%%%%%%%%%%%%%%%%%%%%%%%%%%%%%%%%%%%%%%%%%%%%%%%%%%%%%%
%%%%%%%%%%%%%%%%%%%%%%%%%%%%%%%%%%%%%%%%%%%%%%%%%%%%%%%%%%%%%%%%%%%%%%%%%

\parag 
We are now ready to study the effect of alternality or symmetrality in this context.
\label{secaltsym} \label{secmotivsh}
\begin{prop}	\label{propaltsym}
Suppose that $\bB_\bul$ is an $\gF$-valued cosymmetral comould
and let $M^\bul  \in \hM^\bul(\Om,\bA)$ be such that 
$(M^\omb \bB_\omb)_{\omb\in\Om^\bul}$ is formally summable.
Then:
\begin{enumerate}[--]
\item If $M^\bul$ is alternal, then $\sum M^\bul \bB_\bul  \in \gL_\gF$.
\item If $M^\bul$ is symmetral, then $\sum M^\bul \bB_\bul  \in G_\gF$.
\item More generally, denoting by $\bM^\cbul$ the image of~$M^\bul$ by the
homomorphism~$\tau$ of Lemma~\ref{lemhomomtau} and assuming that the family 
$\big( \rho(\bM^{\alb,\beb} \bB_{\alb}\otimes\bB_{\beb}) 
\big)_{(\alb,\beb)\in\Om^\bul\times\Om^\bul}$
is formally summable in~$\gB$,
\begin{equation}	\label{eqsigtau}
\sig\left( \sum M^\bul \bB_\bul \right)
= \sum_{(\alb,\beb)\in\Om^\bul\times\Om^\bul} 
\rho(\bM^{\alb,\beb} \bB_{\alb}\otimes\bB_{\beb}).
\end{equation}
\end{enumerate}
\end{prop}

\begin{proof}
Let $\de_*\in\Z$ such that $v(\omb) = \vln{M^\omb\bB_\omb} \ge \de_*$ for all $\omb\in\Om^\bul$
and $\Th = \sum M^\bul \bB_\bul$.
We shall use the notation 
$\Phi_{\alb,\beb} = \bM^{\alb,\beb} \bB_{\alb}\otimes\bB_{\beb}$
for $(\alb,\beb)\in\Om^\bul\times\Om^\bul$.
Lemma~\ref{lemtaualtsym} yields
$$
\Phi_{\alb,\beb} = 1_{\{\beb=\est\}} M^\alb\bB_\alb \otimes \ID 
+ 1_{\{\alb=\est\}} \ID \otimes M^\beb\bB_\beb
$$ 
for $M^\bul$ alternal and
$\Phi_{\alb,\beb} = M^\alb \bB_\alb \otimes M^\beb \bB_\beb$ 
for $M^\bul$ symmetral.
In both cases, the set $\ao (\alb,\beb)\in\Om^\bul\times\Om^\bul \mid
\rho(\Phi_{\alb,\beb}) \notin \gB_\de \af$ is thus finite for any $\de\in\Z$, in
view of the formal summability hypothesis,
and the sum of the family $\big(\rho(\Phi_{\alb,\beb})\big)$ is respectively
$\rho(\Th\otimes\ID+\ID\otimes\Th)$ or $\rho(\Th\otimes\Th)$, by continuity of~$\ti\rho$.
Therefore it is sufficient to prove the third property
(the invertibility of~$\Th$ when $M^\bul$ is symmetral is a simple consequence
of Proposition~\ref{propmultiplimould} and of the invertibility of~$M^\bul$).

We thus assume $\big(\rho(\Phi_{\alb,\beb})\big)$ formally summable in~$\gB$.
As in the proof of Proposition~\ref{propmultiplimould}, we can suppose $\Om$
countable and choose an exhaustion of~$\Om^\bul$ by finite sets of the
form~$\Om^{K,R}$.
Then, by virtue of the definition of~$\tau$ and of Lemma~\ref{lemcosym},
\begin{multline*}
A_{K,R} :=
\sum_{(\alb,\beb)\in\Om^{K,R}\times\Om^{K,R}} \rho(\Phi_{\alb,\beb}) 
- \sig\Bigg( \sum_{\omb\in\Om^{K,R}} M^\omb \bB_\omb \Bigg) \\
= \Bigg( \sum_{ \alb,\beb\in\Om^{K,R}, \, \omb\in\Om^\bul }
- \sum_{ \alb,\beb\in\Om^\bul, \, \omb\in\Om^{K,R} } \Bigg)
\sh{\alb}{\beb}{\omb} M^\omb \rho(\bB_\alb\otimes\bB_\beb) \\
= \sum_{ \substack{\alb,\beb\in\Om^{K,R} \\ \text{s.t.}\, r(\alb)+r(\beb)>R} }
\rho(\Phi_{\alb,\beb})
\end{multline*}
(the last equality stems from the fact that, if $\omb\in\Om^{K,R}$, then
$\tsh{\alb}{\beb}{\omb}\neq0$ implies $\alb,\beb\in\Om^{K,R}$).
The formal summability of $\big(\rho(\Phi_{\alb,\beb})\big)$ yields
$A_{K,R}\to0$ as $K,R\to\infty$, which is the desired result since $\sig$ is
continuous. 
\end{proof}

\begin{remark}
The proof of Theorem~\ref{thmSNformal} is now complete:
in view of the symmetrality of~$\cV^\bul$ with $\Om=\cN$ and $\bA=\C[[x]]$ 
(Proposition~\ref{propcVsym})
and the cosymmetrality of~$\bB_\bul$ defined by~\eqref{eqdefbBn}
with $\gF=\gF_{\gA,\bA}$, $(\gA,\val) = (\C[[x,y]],\nu_4)$,
Proposition~\ref{propaltsym} shows that $\Th = \sum \cV^\bul \bB_\bul$ is an
automorphism of~$\gA$.
As noticed in Remark~\ref{remcommentpf}, this was the only thing which remained to be
checked.
\end{remark}
\label{secPfThm}

%%%%%%%%%%%%%%%%%%%%%%%%%%%%%%%%%%%%%%%%%%%%%%%%%%%%%%%%%%%%%%%%%%%%%%%%%
%%%%%%%%%%%%%%%%%%%%%%%%%%%%%%%%%%%%%%%%%%%%%%%%%%%%%%%%%%%%%%%%%%%%%%%%%

\parag 
Another way of checking that the contraction of an alternal mould into a
cosymmetral comould~$\bB_\bul$ is a derivation is to express it as a sum of iterated Lie
brackets of the derivations~$B_\eta$ which generate the comould.

For $\omb=(\om_1,\dotsc,\om_r)\in\Om^\bul$ with $r\ge2$, let
$$
\bB_{[\omb]} = [ B_{\om_r}, [ B_{\om_{r-1}}, [ \dotsb 
[B_{\om_2}, B_{\om_1}] \dotsb ]]].
$$
One can check that, for any alternal mould~$M^\bul$ and for any finite
subset~$\Om_{\!f}$ of~$\Om$,
$$
\sum_{\omb\in\Om_{\!f}^r} M^\omb \bB_\omb =
\frac{1}{r} \sum_{\omb\in\Om_{\!f}^r} M^\omb \bB_{[\omb]},
\qquad r\ge2
$$
(identifying $\Om_{\!f}^r$ with the sets of all words of length~$r$ the letters of
which belong to~$\Om_{\!f}$).
The proof is left to the reader.

%%%%%%%%%%%%%%%%%%%%%%%%%%%%%%%%%%%%%%%%%%%%%%%%%%%%%%%%%%%%%%%%%%%%%%%%%
%%%%%%%%%%%%%%%%%%%%%%%%%%%%%%%%%%%%%%%%%%%%%%%%%%%%%%%%%%%%%%%%%%%%%%%%%

\parag 
Let $\bB_\bul$ denote an $\gF$-valued comould.
Suppose that $(B_\eta)_{\eta\in\Om}$ is formally summable and consider
$Y=\sum_{\eta\in\Om} B_\eta \in \gF$.

We have seen that, by definition, the comould is cosymmetral iff each~$B_\eta$
is a derivation of~$\gA$; then $Y$ is itself a derivation.
This is the situation when there is an appropriate notion of homogeneity, as in
Definition~\ref{defhomog}, and we expand an $\bA$-linear derivation~$Y$ into a
sum of homogeneous components, each $B_\eta$ being homogeneous of degree~$\eta$.

\label{secAlludel}

Suppose now that the object to analyse is not a singular vector field, as in the
case of the saddle-node, but a local transformation; considering the associated
substitution operator, we are thus led to an automorphism of~$\gA$, typically of
the form $\phi=\ID+Y$. Then the homogeneous components~$B_\eta$ of~$Y$ are no longer
derivations; expanding $\sig(\phi) = \rho(\phi\otimes\phi)$, we rather get
\begin{equation}	\label{eqmodifLeib}
\sig(B_\eta) = \rho\Big( B_\eta \otimes \ID +
\sum_{\eta'+\eta''=\eta} B_{\eta'}\otimes B_{\eta''} 
+ \ID \otimes B_\eta \Big).
\end{equation}
The comould~$\bB_\bul$ they generate is then called {\em cosymmetrel}.
A cosymmetrel comould is characterized by identities similar
to~\eqref{eqmotivsh} but with the shuffling coefficients
$\tsh{\omb^1}{\omb^2}{\omb}$ replaced by new ones, denoted by
$\tcsh{\omb^1}{\omb^2}{\omb}$ and called ``contracting shuffling coefficients''.

Dually, using these new coefficients instead of the previous shuffling coefficients in
formulas~\eqref{eqdefaltal} and~\eqref{eqdefsymal}, one gets the definition of
{\em alternel} and {\em symmetrel} moulds, which were only briefly alluded to at
the beginning of Section~\ref{secAltSym}.

The contraction of alternel or symmetrel moulds into cosymmetrel comoulds
enjoy properties parallel to those that we just described in the cosymmetral
case.
This allows one to treat local vector fields and local discrete
dynamical systems with completely parallel formalisms.

%%%%%%%%%%%%%%%%%%%%%%%%%%%%%%%%%%%%%%%%%%%%%%%%%%%%%%%%%%%%%%%%%%%%%%%%%
%%%%%%%%%%%%%%%%%%%%%%%%%%%%%%%%%%%%%%%%%%%%%%%%%%%%%%%%%%%%%%%%%%%%%%%%%
%%%%%%%%%%%%%%%%%%%%%%%%%%%%%%%%%%%%%%%%%%%%%%%%%%%%%%%%%%%%%%%%%%%%%%%%%
%%%%%%%%%%%%%%%%%%%%%%%%%%%%%%%%%%%%%%%%%%%%%%%%%%%%%%%%%%%%%%%%%%%%%%%%%

\part{Resurgence, alien calculus and other applications}

%%%%%%%%%%%%%%%%%%%%%%%%%%%%%%%%%%%%%%%%%%%%%%%%%%%%%%%%%%%%%%%%%%%%%%%%%
%%%%%%%%%%%%%%%%%%%%%%%%%%%%%%%%%%%%%%%%%%%%%%%%%%%%%%%%%%%%%%%%%%%%%%%%%

\section{Resurgence of the normalising series}	\label{secResurSN}

%%%%%%%%%%%%%%%%%%%%%%%%%%%%%%%%%%%%%%%%%%%%%%%%%%%%%%%%%%%%%%%%%%%%%%%%%
%%%%%%%%%%%%%%%%%%%%%%%%%%%%%%%%%%%%%%%%%%%%%%%%%%%%%%%%%%%%%%%%%%%%%%%%%

%%%%%%%%%%%%%%%%%%%%%%%%%%%%%%%%%%%%%%%%%%%%%%%%%%%%%%%%%%%%%%%%%%%%%%%%%
%%%%%%%%%%%%%%%%%%%%%%%%%%%%%%%%%%%%%%%%%%%%%%%%%%%%%%%%%%%%%%%%%%%%%%%%%

\parag 
The purpose of this section is to use the mould-comould representation of the
formal normalisation of the saddle-node given by Theorem~\ref{thmSNformal} to
deduce ``resurgent properties''.
We begin by a few reminders about \'Ecalle's Resurgence theory. We
follow the notations of~\cite{kokyu}.

%%%%%%%%%%%%%%%%%%%%%%%%%%%%%%%%%%%%%%%%%%%%%%%%%%%%%%%%%%%%%%%%%%%%%%%%%

\medskip

\noindent -- 
The formal Borel transform is the $\C$-linear homomorphism
$$
\cB \,:\;
\wt\ph(z) = \sum_{n\ge0} c_n z^{-n-1} \in z\ii\C[[z\ii]]
\ens\mapsto\ens
\wh\ph(\ze) = \sum_{n\ge0} c_n \frac{\ze^n}{n!} \in \C[[\ze]].
$$
In the case of a convergent~$\wt\ph$, one gets a convergent series~$\wh\ph$
which defines an entire function of exponential type.
Namely, if $\ph(x) = \wt\ph(-1/x)\in x\C[[x]]$ has radius of convergence $>\rho$, then
there exists $K>0$ such that $|\ph(x)|\le K|x|$ for $|x|\le\rho$ and this
implies, by virtue of the Cauchy inequalities, that
$|c_n| \le K \rho^{-n}$, hence $\wh\ph$ is entire and
\begin{equation}	\label{ineqentire}
|\wh\ph(\ze)|  \le K\, \ee^{\rho\ii|\ze|}, \qquad \ze\in\C.
\end{equation}
We are particularly interested in the case where $\wh\ph(\ze)\in\C\{\ze\}$
without being necessarily entire; this is equivalent to Gevrey-$1$ growth
for the coefficients of~$\wt\ph$:
\begin{multline*}
\wt\ph(z) \in z\ii\C[[z\ii]]_1 
\quad \overset{\text{def}}{\Longleftrightarrow} \quad
\exists C,K>0 \; \text{s.t.}\; |c_n|\le K C^n n! \; \text{for all $n$} \\
\quad\Longleftrightarrow\quad
\cB\wt\ph(\ze) \in \C\{\ze\}.
\end{multline*}

%%%%%%%%%%%%%%%%%%%%%%%%%%%%%%%%%%%%%%%%%%%%%%%%%%%%%%%%%%%%%%%%%%%%%%%%%

\medskip

\noindent -- 
The counterpart in~$\C[[\ze]]$ of the multiplication (Cauchy product) of
$z\ii\C[[z\ii]]$ is called {\em convolution} and denoted by~$*$, thus
$\cB(\wt\ph\cdot\wt\psi) = \cB(\wt\ph) * \cB(\wt\psi)$.
Now, if $\wh\ph = \cB(\wt\ph)$ and $\wh\psi= \cB(\wt\psi)$ belong
to~$\C\{\ze\}$, then $\wh\ph * \wh\psi \in \C\{\ze\}$ and this germ of
holomorphic function is determined by
\begin{equation}	\label{eqdefconvol}
(\wh\ph * \wh\psi)(\ze) = \int_0^\ze \wh\ph(\ze_1) \wh\psi(\ze-\ze_1)
\qquad \text{for $|\ze|$ small enough.}
\end{equation}
We have an algebra $(\C\{\ze\},*)$ without unit, isomorphic via~$\cB$ to 
$(z\ii\C[[z\ii]]_1,\cdot)$.
By adjunction of unit, we get an algebra isomorphism
$$
\cB \colon \C[[z\ii]]_1 \overset{\sim}{\to} \C\, \de \oplus \C\{\ze\}
$$
(where $\de = \cB 1$ is a symbol for the unit of convolution).
We can even take into account the differential $\frac{\dd\,}{\dd z}$: its
counterpart via~$\cB$ is 
$\wh\pa \colon c\,\de + \wh\ph(\ze) \mapsto -\ze\wh\ph(\ze)$.

%%%%%%%%%%%%%%%%%%%%%%%%%%%%%%%%%%%%%%%%%%%%%%%%%%%%%%%%%%%%%%%%%%%%%%%%%

\medskip

\noindent -- 
Let us now consider all the rectifiable oriented paths of~$\C$ which start from
the origin and then avoid~$\Z$, \ie\ oriented paths represented by absolutely
continuous maps $\ga \colon [0,1] \to
\C\setminus\Z^*$ such that $\ga(0)=0$ and $\ga\ii(0)$ is connected.
We denote by~$\gR(\Z)$ the set of all homotopy classes~$[\ga]$ of
such paths~$\ga$ and by~$\pi$ the map $[\ga]\in\gR(\Z) \mapsto \ga(1)\in\C\setminus\Z^*$;
considering $\pi$ as a covering map, we get a Riemann surface structure
on~$\gR(\Z)$.

Observe that $\pi\ii(0)$ consists of a single point, the ``origin''
of~$\gR(\Z)$; this is the only difference between $\gR(\Z)$ and the universal
cover of~$\C\setminus\Z$.
The space $\wHR\Z$ of all holomorphic functions of~$\gR(\Z)$ can be identified
with the space of all $\wh\ph(\ze) \in \C\{\ze\}$ which admit an
analytic continuation along any representative of any element of~$\gR(\Z)$
(\cf \cite{kokyu}, Definition~3 and Lemma~2).

\begin{definition}	
We define the convolutive model of the algebra of resurgent functions over~$\Z$
as $\hRZ = \C\, \de  \oplus \wHR\Z$.
We define the formal model of the algebra of resurgent functions over~$\Z$
as $\tRZ = \cB\ii(\hRZ)$.
\end{definition}

It turns out that $\hRZ$ is a subalgebra of the convolution algebra
$\C\, \de  \oplus \C\{\ze\}$, \ie\ the aforementioned property of analytic
continuation is stable by convolution;
the proof of this fact relies on the notion of symmetrically contractile path
(see for instance {\em op.\ cit.}, \S1.3), which we shall not develop here.
Therefore $\tRZ$ is a subalgebra of $\C[[z\ii]]$ and $\cB$ induces an algebra
isomorphism $\tRZ \to \hRZ$.

%%%%%%%%%%%%%%%%%%%%%%%%%%%%%%%%%%%%%%%%%%%%%%%%%%%%%%%%%%%%%%%%%%%%%%%%%

\medskip

\noindent -- 
An obvious example of element of~$\wHR\Z$ is an entire function, or a
meromorphic function of~$\C$ without poles outside~$\Z^*$.
Indeed, for such a function~$\wh\ph$, we can define $\hat\phi\in\wHR\Z$ by
$\hat\phi(\ze) = \wh\ph\big(\pi(\ze)\big)$ for all $\ze\in\gR(\Z)$.
We usually identify~$\wh\ph$ and~$\hat\phi$.

For example, if $\om_1\neq0$, Remark~\ref{remheuri} shows that 
$\wt\cV^{\om_1}(z) = \cV^{\om_1}(-1/z) \in z\ii\C[[z\ii]]$
has formal Borel transform
$\wh\cV^{\om_1}(\ze) = \sum \om_1^{-r-1}(-\wh\pa\,)^r \wh a_{\om_1}
= \frac{\wh a_{\om_1}(\ze)}{\om_1-\ze}$, 
where $\wh a_{\om_1}$ denotes the formal Borel transform of $a_{\om_1}(-1/z)$,
which is an entire function (since $a_{\om_1}$ is convergent),
thus $\wh\cV^{\om_1}$ is meromorphic with at most one simple pole, located
at~$\om_1$.
On the other hand, $\wh\cV^0(\ze) = \frac{1}{\ze}\wh a_0(\ze)$ is entire.

We shall see that, for each non-empty word $\omb$, the formal Borel transform of
$\cV^\omb(-1/z)$ belongs to~$\wHR\Z$, but this function is usually not
meromorphic if $r(\omb)\ge2$.
For instance, for $\omb=(\om_1,\om_2)$, one gets
$\frac{1}{-\ze+\om_1+\om_2} \big(\wh a_{\om_1}*\wh\cV^{\om_2} \big)$
which is multivalued in general
(see formula~\eqref{eqcViter} below for the general case).

%%%%%%%%%%%%%%%%%%%%%%%%%%%%%%%%%%%%%%%%%%%%%%%%%%%%%%%%%%%%%%%%%%%%%%%%%

\medskip

\noindent -- 
A formal series~$\wt\ph(z)$ without constant term belongs to~$\tRZ$ iff its
formal Borel transform~$\wh\ph(\ze)$ converges to a germ of holomorphic function
which extends analytically to~$\gR(\Z)$. 
In particular, the principal branch\footnote{
The principal branch is defined as the analytic continuation of~$\wh\ph$ in the
maximal open subset of~$\C$ which is star-shaped \wrt~$0$; its domain is 
the cut plane obtained by removing the singular half-lines
$\left[1,+\infty\right[$ and $\left]-\infty,-1\right]$, unless $\wh\ph$ happens
to be regular at~$1$ or~$-1$.
}
of~$\wh\ph$ is holomorphic in sectors which extend up to infinity. If it has at
most exponential growth in a sector $\ao \ze\in\C \mid \th_1\le \arg\ze \le
\th_2 \af$
(as is the case of~$\wh\cV^{\om_1}(\ze)$ for instance),
then one can perform a Laplace transform and get a function 
$$
\tphan(z) = \int_0^{\ee^{\I\th}\infty} \wh\ph(\ze)\,\ee^{-z\ze}  \, \dd\ze,
\qquad \th \in [\th_1,\th_2],
$$
which is analytic for~$z$ belonging to a sectorial neighbourhood of
infinity. 
This is called {\em Borel-Laplace summation} (see
\eg~\cite{kokyu}, \S1.1). 

Since multiplication is turned into convolution by~$\cB$ and then turned again
into multiplication by the Laplace transform, and similarly with
$\frac{\dd\,}{\dd z}$ which is transformed into multiplication by $-\ze$
by~$\cB$, the Borel-Laplace process transforms the formal solution of a
differential equation like~\eqref{eqdiffeqzero}, \eqref{eqdiffeqn}
or~\eqref{eqdefcV} into an analytic solution of the same equation.

%%%%%%%%%%%%%%%%%%%%%%%%%%%%%%%%%%%%%%%%%%%%%%%%%%%%%%%%%%%%%%%%%%%%%%%%%
%%%%%%%%%%%%%%%%%%%%%%%%%%%%%%%%%%%%%%%%%%%%%%%%%%%%%%%%%%%%%%%%%%%%%%%%%

\parag 
The stability of~$\tRZ$ under multiplication together with the previous computation
explains to some extent why we can expect the solutions of a non-linear problem
like the formal classification of the saddle-node to be resurgent.
However, controlling products in~$\tRZ$ means controlling convolution products
in~$\hRZ$, and it is not so easy to extract from the stability statement the
quantitative information which would guarantee the convergence in~$\hRZ$ of a
method of majorant series for instance
(see the discussion at the end of the sketch of proof of Theorem~2 of~\cite{kokyu}).

Thanks to the mould-comould expansion given in Section~\ref{secMCexpSN}, we
shall be able to use much simpler arguments: the convolution product of an
element of~$\hRZ$ with an entire function belongs to~$\hRZ$ and efficient
bounds are available in this particular case of the stability statement (much
easier than the general one---see Lemma~~\ref{lemAnCont}
below).

\begin{thm}	\label{thmResur}
Consider the saddle-node problem, with hypotheses
\eqref{eqdefX}--\eqref{eqassA}.
Let $\th(x,y) = \left(x, y + \sum_{n\ge0} \ph_n(x) y^n\right)$ denote the formal
transformation, the substitution operator of which is $\Th = \sum \cV^\bul
\bB_\bul$, in accordance with Lemmas~\ref{lemdefcV} and~\ref{lemCVformal} and Theorem~\ref{thmSNformal}.
Let $\th\ii(x,y) = \left(x, y + \sum_{n\ge0} \psi_n(x) y^n\right)$ denote the inverse
transformation.

Then, for each $n\in\N$, the formal series $\wt\ph_n(z) = \ph_n(-1/z)$ and $\wt\psi_n(z) =
\psi_n(-1/z)$ belong to~$\tRZ$, and 
the analytic continuation of the formal Borel transforms
$\wh\ph_n(\ze),\wh\psi_n(\ze)\in\C\{\ze\}$ satisfy the following:
\begin{enumerate}[(i)]
\item All the branches of the analytic continuation of~$\wh\ph_n$ are regular at
the points of $n+\N = \{n, n+1, n+2, \dotsc\}$.
\item All the branches of the analytic continuation of~$\wh\psi_n$ are regular at
the points of $-\N^* = \{-1, -2, -3 \dotsc\}$, 
with the sole exception that the branches of~$\wh\psi_0$ may have simple poles
at~$-1$.
\item Given any $\rho\in\left]0,\demi\right[$ and $N\in\N^*$, there exist
positive constants $K,L,C$ which depend only on $\rho,N$ such that,
for any $(\rho,N,n-\N^*)$-adapted infinite path~$\ga$ issuing from the origin,
\begin{equation}	\label{ineqwhphn}
\big| \wh\ph_n\big( \ga(t) \big) \big| \le K L^n \,  \ee^{(n^2+1) C t} 
\ens\text{for all $t\ge0$ and $n\in\N$,}
\end{equation}
while, for any $(\rho,N,\N)$-adapted infinite path~$\ga$ issuing from the origin,
\begin{equation}	\label{ineqwhpsin}
\big| \wh\psi_n\big( \ga(t) \big) \big| \le K L^n \,  \ee^{C t} 
\ens\text{for $n\ge1$,}\quad
\big| \big(\ga(t)+1\big) \wh\psi_0\big( \ga(t) \big) \big| \le K \,  \ee^{C t}
\end{equation}
for all $t\ge0$.
\end{enumerate}
\end{thm}

\label{secResur}

What we call $(\rho,N,\gP^\pm)$-adapted infinite path, with $\gP^+ = \N$ or $\gP^- = n-\N^*$, is defined
below in Definition~\ref{defiadaptedpath}; see Figure~\ref{figrhoNadapt}.
These are arc-length parametrised paths $\ga \colon \left[0,+\infty\right[
\to \C$ (\ie\ $\ga$ is absolutely continuous and $|\dot\ga(t)|=1$ for almost
every~$t$) which start as rectilinear segments of length~$\rho$ issuing from the
origin and which then do not approach~$\gP^\pm$ nor $\pm\Sig(\rho,N)$ at a distance
$<\rho$, where $\pm\Sig(\rho,N)$ denotes the sector of half-opening
$\arcsin(\rho/N)$ bissected by $\pm\left[N,+\infty\right[$.

In particular, inequalities \eqref{ineqwhphn}--\eqref{ineqwhpsin} yield an
exponential bound at infinity for the principal branch of each~$\wh\ph_n$
or~$\wh\psi_n$ along all the half-lines issuing from~$0$
except the singular half-lines $\pm\left[0,+\infty\right[$ 
(the half-line $\left[0,+\infty\right[$ is not singular for $\wh\ph_0$ and the
half-line $-\left[0,+\infty\right[$ is not singular for any $\wh\psi_n$).

We recall that $\Th$ establishes a conjugacy between the saddle-node vector
field~$X$ and its normal form~$X_0$, thus the formal series $\wt\ph_n(z)$ are
the components of a formal integral~$\wt Y(z,u)$, as described
in~\eqref{eqdefYzu}--\eqref{eqdiffeqX}.
The resurgence statement contained in Theorem~\ref{thmResur} thus means that the
formal solutions of the singular differential
equations~\eqref{eqdiffeqzero}--\eqref{eqdiffeqn} may be divergent but that this
divergence is of a very precise nature.
We shall briefly indicate in Section~\ref{secBE} how alien calculus allows one
to take advantage of this information to study the problem of analytic
classification.

\begin{remark}	\label{remgetrid}
Theorem~\ref{thmResur} also permits the obtention of analytic solutions
of~\eqref{eqdiffeqzero}--\eqref{eqdiffeqn} via Borel-Laplace summation.
It is thus worth mentioning that one can get rid of the dependence on~$n$ in the
exponential which appears in~\eqref{ineqwhphn}, provided one restricts oneself
to paths which start from the origin and then do not approach at a distance $<\rho$ the set
$\Z\cup\Sig(\rho,N)\cup\big(-\Sig(\rho,N)\big)$, and which cross the cuts (the
segments between consecutive points of~$\Z$) at most~$N'$ times.
For instance, with $N'=0$, one obtains 
\begin{equation}	\label{inequniformphn}
\big| \wh\ph_n( \ze ) \big| \le K L^n \,  \ee^{C |\ze|} 
\end{equation}
for the principal branch of~$\wh\ph_n$, possibly with larger constants~$K,L,C$
but still independent of~$n$.
For the other branches, which correspond to $N'\ge1$ and $N\ge2$, one has to resort to
symmetrically contractile paths and the implied constants~$K,L,C$ depend only
on~$\rho,N,N'$.

Therefore, when performing Laplace transform,
inequalities~\eqref{inequniformphn} allow one to get the same domain of
analyticity for all the functions~$\tphan_n(z)$ solutions
of~\eqref{eqdiffeqzero}--\eqref{eqdiffeqn} (a sectorial neighbourhood of
infinity which depends only on~$C$; see \eg~\cite{kokyu}, \S1.1),
with explicit bounds which make it possible to study the domain of analyticity
of a {\em sectorial formal integral} 
$\Yan(z,u) = u\,\ee^z + \sum u^n \, \ee^{nz}\tphan_n(z)$
or of analytic normalising transformations $\phan(x,y)$, $\psian(x,y)$.

This will be used in Section~\ref{secMR}.
\end{remark}

The rest of this section is devoted to the proof of Theorem~\ref{thmResur} and
to the derivation of inequalities~\eqref{inequniformphn}.

%%%%%%%%%%%%%%%%%%%%%%%%%%%%%%%%%%%%%%%%%%%%%%%%%%%%%%%%%%%%%%%%%%%%%%%%%
%%%%%%%%%%%%%%%%%%%%%%%%%%%%%%%%%%%%%%%%%%%%%%%%%%%%%%%%%%%%%%%%%%%%%%%%%

\parag 
Using $\Om=\cN = \ao \eta\in\Z \mid \eta\ge -1 \af$ as an alphabet, we know that
the $\C[[x]]$-valued moulds~$\cV^\bul$ and~$\cVt^\bul$ are symmetral and
mutually inverse for mould multiplication. We recall that
$$
\Th\ii = \sum \cVt^\bul \bB_\bul, \qquad 
\cVt^{\om_1,\dotsc,\om_r} = (-1)^r \cV^{\om_r,\dotsc,\om_1}.
$$
With the notation $\norm{\omb}=\om_1+\dotsb+\om_r$ for any non-empty word
$\omb\in\Om^\bul$, % and $\norm{\est}=0$, 
equation~\eqref{eqdefbeomb} can be written
$\bB_\omb y = \be_\omb y^{\norm{\omb}+1}$,
with the coefficients~$\be_\omb$ defined at the end of Section~\ref{secMCexpSN}.
As was already observed, since $\Th y = \sum\ph_n(x)y^n$ and $\Th\ii y =
\sum\psi_n(x)y^n$ , the formal series we are interested in can be written as
formally convergent series in $\C[[x]]$:
\begin{equation}	\label{eqformulphnpsin}
\ph_n = \sum_{\norm{\omb}=n-1} \be_\omb \cV^\omb, \qquad
\psi_n = \sum_{\norm{\omb}=n-1} \be_\omb \cVt^\omb, \qquad n\in\N,
\end{equation}
with summation over all words~$\omb$ of positive length subject to the condition
$\norm{\omb}=n-1$. In fact, not all of these words contribute in these series:

%%%%%%%%%%%%%%%%%%%%%%%%%%%%%%%%%%%%%%%%%%%%%%%%%%

\begin{lemma}	\label{lemexo}
For any non-empty $\omb = (\om_1,\dotsc,\om_r)\in\Om^\bul$, 
using the notations
\begin{equation}	\label{eqnotahatcheck}
\wc\om_i = \om_1 + \dotsb + \om_i, \quad
\htb\om_i = \om_i + \dotsb + \om_r, \qquad
1 \le i \le r,
\end{equation}
we have
$$
\be_\omb \neq0 \ens\Rightarrow\ens
\norm{\omb} \ge -1, \ens
\wc\om_1,\dotsc,\wc\om_{r-1}\ge0 \ens\text{and}\ens 
\htb\om_1,\dotsc,\htb\om_r\le\norm{\omb}.
$$
\end{lemma}

%%%%%%%%%%%%%%%%%%%%%%%%%%%%%%%%%%%%%%%%%%%%%%%%%%

\begin{proof}
We have 
\begin{equation}	\label{eqdefibeomb}
\be_\omb = 1 \ens\text{if $r=1$},\qquad
\be_\omb = (\wc\om_1+1)(\wc\om_2+1)\dotsm(\wc\om_{r-1}+1) \ens\text{if $r\ge2$}.
\end{equation}
The property $\be_\omb\neq0 \;\Rightarrow\; \norm{\omb}\ge-1$ was already observed
at the end of Section~\ref{secMCexpSN}, as a consequence of $\bB_\omb y
\in\C[[y]]$
(one can also argue directly from formula~\eqref{eqdefibeomb}).

Now suppose $\be_\omb\neq0$ and $1\le i \le r-1$. 
The identity 
$$
\be_\omb = \be_{\om_1,\dotsc,\om_i}
(\wc\om_i+1)\dotsm(\wc\om_{r-1}+1)
$$
implies $\wc\om_i\neq-1$ and $\be_{\om_1,\dotsc,\om_i}\neq0$, hence
$\om_1+\dotsb+\om_i\ge-1$.
Therefore $\wc\om_i\ge0$ and $\htb\om_{i+1} = \norm{\omb}  - \wc\om_i \le
\norm{\omb}$, while $\htb\om_1=\norm{\omb}$.
\end{proof}

%%%%%%%%%%%%%%%%%%%%%%%%%%%%%%%%%%%%%%%%%%%%%%%%%%%%%%%%%%%%%%%%%%%%%%%%%
%%%%%%%%%%%%%%%%%%%%%%%%%%%%%%%%%%%%%%%%%%%%%%%%%%%%%%%%%%%%%%%%%%%%%%%%%

\parag 
We recall that the convergent series $a_\eta(x)$ were defined
in~\eqref{eqdefiancN} as Taylor coefficients \wrt~$y$ of the saddle-node vector
field~\eqref{eqdefX}. 
We define $\wt\ph_n(z)$, $\wt\psi_n(z)$, $\wt a_\eta(z)$, $\wt\cV^\omb(z)$, $\tcVto(z)$
from $\ph_n(x)$, $\psi_n(x)$, $a_\eta(x)$, $\cV^\omb(x)$, $\cVt^\omb(x)$ 
by the change of variable $z=-1/x$ (for any $n\in\N$, $\eta\in\Om$, $\omb\in\Om^\bul$),
and we denote by $\wh\ph_n(\ze)$, $\wh\psi_n(\ze)$, etc.\ the formal Borel transforms of these
formal series.

In view of Lemma~\ref{lemdefcV}, the formal series~$\wt\cV^\omb$ are uniquely
determined by the equations $\wt\cV^\est = 1$ and
\begin{equation*}
\big(\frac{\dd\,}{\dd z} + \norm{\omb}\big) \wt\cV^\omb
= \wt a_{\om_1} \wt\cV^{`\omb}, \qquad
\wt\cV^\omb \in z\ii\C[[z\ii]
\end{equation*}
for $\omb$ non-empty, with $`\omb$ denoting $\omb$ deprived from its first letter.
Since $\cB$ transforms $\frac{\dd\,}{\dd z}$ into multiplication by $-\ze$ and
multiplication into convolution, we get
$\wh\cV^\est = \de$ % = \hcVte = \de$
and
$$
\wh\cV^\omb(\ze) = -\frac{1}{\ze-\norm{\omb}} \big( \wh a_{\om_1}  * \wh\cV^{`\omb} \big), 
\qquad \omb\neq\est,
$$
where the \rhs\ belongs to $\C[[\ze]]$ even if $\norm{\omb}=0$, by the same
argument as in the proof of Lemma~\ref{lemdefcV}.
It belongs in fact to $\C\{\ze\}$, by induction on $r(\omb)$, and
\begin{equation}	\label{eqcViter}
\wh\cV^\omb = (-1)^r
\frac{1}{\ze-\htb\om_1} \Big( \wh a_{\om_1}  * 
\Big( \frac{1}{\ze-\htb\om_2} \Big( \wh a_{\om_2}  * 
\Big(  \dotsb 
\Big( \frac{1}{\ze-\htb\om_r} \wh a_{\om_r}
\Big) \dotsm \Big)\Big)\Big)\Big)
\end{equation}
with the notation of~\eqref{eqnotahatcheck}.
In view of the stability properties of~$\wHR\Z$ (stability by convolution with
another element of~$\wHR\Z$, a fortiori with an entire function, or by
multiplication with a meromorphic function regular on $\C\setminus\Z^*$), this
implies that \emph{the functions $\wh\cV^\omb$ are resurgent}, as announced in
the introduction to this section.
We shall give more details on this later.

%%%%%%%%%%%%%%%%%%%%%%%%%%%%%%%%%%%%%%%%%%%%%%%%%%%%%%%%%%%%%%%%%%%%%%%%%
%%%%%%%%%%%%%%%%%%%%%%%%%%%%%%%%%%%%%%%%%%%%%%%%%%%%%%%%%%%%%%%%%%%%%%%%%

\parag 
Here is a first consequence for the functions $\wh\ph_n$ and~$\wh\psi_n$:

\begin{lemma}	\label{lemtcSm}
For each $n\in\N$,
\begin{equation}	\label{eqwhphnwhpsin}
\wh\ph_n = \sum_{\norm{\omb}=n-1} \be_\omb  \wh\cV^\omb,
\qquad
\wh\psi_n = \sum_{\norm{\omb}=n-1} \be_\omb  \hcVto,
\end{equation}
with formally convergent series in $\C[[\ze]]$, and
for each non-empty~$\omb$ such that $\norm{\omb}=n-1$,
\begin{equation}	\label{eqbewhcV}
\be_\omb \wh\cV^\omb = \cS_{\shtb\om_1}\cA_{\om_1}
\dotsm \cS_{\shtb\om_r}\cA_{\om_r} \de, \qquad
\be_\omb \hcVto = \tfrac{1}{\ze-(n-1)} \cA_{\om_r} 
\tcS_{\swc\om_{r-1}}\cA_{\om_{r-1}}
\dotsm \tcS_{\swc\om_1}\cA_{\om_1} \de,
\end{equation}
with convolution operators
$$
\cA_\eta \colon \wh\ph \mapsto \wh a_\eta * \wh\ph, \qquad
\eta \in \Om
$$
and multiplication operators
\begin{equation}	\label{eqdefiStSm}
\cS_m \colon \wh\ph \mapsto -\tfrac{n-m}{\ze-m}\, \wh\ph, \quad
\tcS_m \colon \wh\ph \mapsto \tfrac{m+1}{\ze-m}\, \wh\ph, \qquad
m \in \Z
\end{equation}
\end{lemma}

%%%%%%%%%%%%%%%%%%%%%%%%%%%%%%%%%%%%%%%%%%%%%%%%%%%%%%%%%%%%%%%%%%%%%%%%%

\begin{proof}
Formula~\eqref{eqwhphnwhpsin} is a direct consequence
of~\eqref{eqformulphnpsin}.

In order to deal with $\hcVto$, we pass from $\omb=(\om_1,\dots,  \om_r)$ to 
$\wt\omb = (\om_r,\dots,  \om_1)$ and this exchanges $\htb\om_i$ and
$\wc\om_{r-i+1}$, thus \eqref{eqcViter} implies
$$
\hcVto = 
\frac{1}{\ze-\wc\om_r} \Big( \wh a_{\om_r}  * 
\Big( \frac{1}{\ze-\wc\om_{r-1}} \Big( \wh a_{\om_{r-1}}  * 
\Big(  \dotsb 
\Big( \frac{1}{\ze-\wc\om_1} \wh a_{\om_1}
\Big) \dotsm \Big)\Big)\Big)\Big).
$$
Since $\wc\om_r=n-1$, multiplying by $\be_\omb =
(\wc\om_{r-1}+1)\dotsm(\wc\om_1+1)$, we get the second part of~\eqref{eqbewhcV}.
The first part of this formula is obtained by multiplying~\eqref{eqcViter}
by~$\be_\omb$ written in the form 
$\be_\omb = (n-\htb\om_1)(n-\htb\om_2)\dotsm(n-\htb\om_r)$
(indeed, $n-\htb\om_1=1$ and $n-\htb\om_i = \wc\om_{i-1}+1$ for $2\le i\le r$).
\end{proof}

%%%%%%%%%%%%%%%%%%%%%%%%%%%%%%%%%%%%%%%%%%%%%%%%%%%%%%%%%%%%%%%%%%%%%%%%%
%%%%%%%%%%%%%%%%%%%%%%%%%%%%%%%%%%%%%%%%%%%%%%%%%%%%%%%%%%%%%%%%%%%%%%%%%

\parag 
The appearance of singularities in our problem is due to the multiplication
operators $\cS_{\shtb\om_i}$ or~$\tcS_{\swc\om_i}$. In view of
Lemma~\ref{lemexo} and formulas~\eqref{eqbewhcV}--\eqref{eqdefiStSm},
we are led to introduce subspaces of~$\wHR\Z$ formed of functions with smaller
sets of singularities.
We do this by considering Riemann surfaces $\gR(\gP)$ slightly more general than
$\gR(\Z)$.

Let $\gP$ denote a subset of~$\Z$. 
We define the Riemann surface $\gR(\gP)$ as the set of all homotopy classes of
rectifiable oriented paths which start from the origin and then avoid~$\gP$. 
%
% Thus $\gR(\Z)$ is nothing but~$\gR$.
%
The Riemann surface $\gR(\gP)$ and the universal cover of $\C\setminus\gP$
coincide if $0\not\in\gP$; there is a difference between them when $0\in\gP$:
there is no point which projects onto~$0$ in the second one, while the first one
still has an ``origin''.

The space $\wHR\gP$ of all holomorphic functions of $\gR(\gP)$ can
be identified with the space of all $\wh\ph(\ze)  \in  \C\{\ze\}$ which admit an
analytic continuation along any representative of any element of $\gR(\gP)$.
It can thus also be identified with the subspace of~$\wHR\Z$ consisting of
those functions holomorphic in~$\gR(\Z)$, the branches of which are regular at each
point of~$\Z\setminus\gP$.

We shall particularly be interested in two cases: $\gP^- = n-\N^*$ and $\gP^+ = \N$.
Indeed, our aim is to show that 
the functions $\wh\ph_n$ belong to $\wHR{n-\N^*}$ for any $n\in\N$
and that the functions $\wh\psi_n$ belong to $\wHR\N$ for any $n\ge1$,
while $(\ze+1)\wh\psi_1(\ze) \in \wHR\N$.

One could prove (with the help of symmetrically contractile paths) that the
spaces $\wHR\N$,
$\wHR{-\N^*}$ or $\wHR{-\N}$ are stable by convolution because the corresponding
sets~$\gP$ are stable by addition,
but beware that this is not the case of $\wHR{n-\N^*}$ if $n\ge2$.

%%%%%%%%%%%%%%%%%%%%%%%%%%%%%%%%%%%%%%%%%%%%%%%%%%%%%%%%%%%%%%%%%%%%%%%%%
%%%%%%%%%%%%%%%%%%%%%%%%%%%%%%%%%%%%%%%%%%%%%%%%%%%%%%%%%%%%%%%%%%%%%%%%%

\parag 
As previously mentioned, for each $\omb\neq\est$,
$\be_\omb\wh\cV^\omb$ and $\be_\omb\hcVto$ belong to~$\wHR\Z$ by virtue of
general stability properties.
But formula~\eqref{eqbewhcV} permits a more elementary argument and more precise
conclusions.

Indeed $\cA_{\om_r}\de = \wh a_{\om_r}$, resp.\ $\cA_{\om_1}\de = \wh a_{\om_1}$,
is an entire function, which vanishes at the origin if $\om_r=0$, resp.\
$\om_1=0$.
Thus
\begin{equation}	\label{eqinifcns}
\cS_{\shtb\om_r}\cA_{\om_r} \de = -\frac{n-\om_r}{\ze-\om_r} \, \wh a_{\om_r},
\qquad \text{resp.}\quad 
\tcS_{\swc\om_1}\cA_{\om_1} \de = \frac{\om_1+1}{\ze-\om_1} \, \wh a_{\om_1},
\end{equation}
is meromorphic on~$\C$ and regular at the origin if $\om_r\neq0$, resp.\ $\om_1\neq0$,
and entire if $\om_r=0$, resp.\ $\om_1=0$.
In fact, 
$$
\htb\om_r \le n-1 \ens\Rightarrow\ens \cS_{\shtb\om_r}\cA_{\om_r} \de \in \wHR{n-\N^*},
\qquad
\wc\om_1 \ge 0 \ens\Rightarrow\ens \cS_{\swc\om_1}\cA_{\om_1} \de \in \wHR\N.
$$
Therefore, one can apply $r-1$ times the following

%%%%%%%%%%%%%%%%%%%%%%%%%%%%%%%%%%%%%%%%%%

\begin{lemma}	\label{lemAnCont}
Suppose that $\gP\subset\Z$, $\wh\ph \in \wHR\gP$ and $\wh b$ is entire.
Then $\wh b*\wh\ph \in \wHR\gP$.
If furthermore $\wh s$ is a meromorphic function, the poles of which all belong
to~$\gP$ and with at most a simple pole at the origin, 
then $\wh s(\wh b*\wh\ph) \in \wHR\gP$.

Consider a rectifiable oriented path with arc-length parametrisation $\ga \colon
[0,T] \to \C$, such that $\ga(0)=0$ and $\ga(t) \in \C\setminus\gP$ for $0<t\le
T$.
Denoting the analytic continuation of~$\wh\ph$ along~$\ga$ by the same symbol~$\wh\ph$, we
suppose moreover that
$$
\big| \wh\ph\big( \ga(t) \big) \big| \le P(t) \,\ee^{Ct}, \qquad
0  \le t \le T,
$$
with a continuous function~$P$ and a constant $C\ge0$, and that there is a
continuous monotonic non-decreasing function~$Q$ such that $|\wh b(\ze)| \le
Q\big(|\ze|\big)\,\ee^{C|\ze|}$ for all $\ze\in\C$.
Then, for all $t\in  [0,T]$, the analytic continuation of $\wh b*\wh\ph$ at
$\ze=\ga(t)$ satisfies
\begin{equation}	\label{eqbstarwhph}
\wh b*\wh\ph(\ze) = \int_{\ga_\ze}
\wh b(\ze-\ze') \wh\ph(\ze') \,\dd\ze',
\qquad
\big|\wh b*\wh\ph\big( \ga(t)  \big)\big| \le
P*Q(t) \, \ee^{Ct},
\end{equation}
with $\ga_\ze$ denoting the restriction $\ga_{| [0,t]}$ and
$P*Q(t) = \int_0^t P(t')Q(t-t')\,\dd t'$.
\end{lemma}

%%%%%%%%%%%%%%%%%%%%%%%%%%%%%%%%%%%%%%%%%%

\begin{proof}
The first statement and the first part of~\eqref{eqbstarwhph} are obtained by
means of the Cauchy theorem:
if $\ga_1$ and~$\ga_2$ are two representatives of the same element of~$\gR(\gP)$ and
$\xi=\pi([\ga_1])=\pi([\ga_2])$, then $\int_{\ga_1}\wh b(\xi-\xi')
\wh\ph(\xi') \,\dd\xi'$ and $\int_{\ga_2}\wh b(\xi-\xi') \wh\ph(\xi')
\,\dd\xi'$ coincide; one can check that the function thus defined on~$\gR(\gP)$ is
holomorphic and this is clearly an extension of $\wh b*\wh\ph$.
Moreover $\wh b*\wh\ph$ vanishes at the origin, thus $\wh s(\wh
b*\wh\ph)\in\wHR{\gP}$ even if~$\wh s$ has a simple pole at~$0$.

We thus have
$$
\wh b*\wh\ph\big( \ga(t)  \big) = \int_0^t 
\wh b\big( \ga(t)-\ga(t') \big) \wh\ph\big( \ga(t') \big)
\dot\ga(t')\,\dd t'.
$$
For almost every $t'\in[0,t]$, $|\dot\ga(t')|=1$ and $|\ga(t)-\ga(t')|\le t-t'$, whence
$|\wh b\big( \ga(t)-\ga(t') \big)| \le Q(t-t')\,\ee^{C(t-t')}$ by monotonicity
of $\xi\mapsto Q(\xi)\,\ee^\xi$.
The conclusion follows.
\end{proof}

In view of Lemma~\ref{lemexo} and formula~\eqref{eqbewhcV}, the first part of
Lemma~\ref{lemAnCont} implies
\begin{cor}	\label{corbecVo}
Let $n\in\N$ and $\omb$ be a non-empty word such that $\norm{\omb}=n-1$.
Then the function $\be_\omb\wh\cV^\omb$ belongs to $\wHR{n-\N^*}$ and 
the function 
$$
\ze \mapsto \big(\ze-(n-1)\big) \be_\omb\hcVto(\ze)
$$
belongs to $\wHR\N$.
\end{cor}

%%%%%%%%%%%%%%%%%%%%%%%%%%%%%%%%%%%%%%%%%%%%%%%%%%%%%%%%%%%%%%%%%%%%%%%%%
%%%%%%%%%%%%%%%%%%%%%%%%%%%%%%%%%%%%%%%%%%%%%%%%%%%%%%%%%%%%%%%%%%%%%%%%%

\parag 
Our aim is now to exploit formula~\eqref{eqbewhcV} and the quantitative information
contained in Lemma~\ref{lemAnCont} to produce upper bounds for
$$
\big|\be_\omb\wh\cV^\omb(\ze)\big|, \quad \text{resp.}\ens
\big| \big(\ze-(n-1) \big) \be_\omb\hcVto(\ze) \big|
$$ 
which will ensure the uniform convergence of the series~\eqref{eqwhphnwhpsin}
(up to the factor $\ze-(n-1)$ for the second one) in any compact subset 
of $\gR(n-\N^*)$, resp.\ $\gR(\N)$.

We first choose positive constants $K,L,C$ such that
\begin{equation}	\label{ineqwhaeta}
\left|  \wh a_\eta(\ze) \right|  \le K L^\eta \, \ee^{C|\ze|},
\qquad \ze\in\C,\; \eta\in\Om.
\end{equation}
This is possible, since $\sum \dfrac{a_\eta(x)}{x} y^{\eta+1} =
\dfrac{A(x,y)-y}{y} \in \C\{x,y\}$ by assumption, thus one can find constants
such that
$\big| \frac{a_\eta(x)}{x} \big| \le K L^\eta$ for $|x|\le C\ii$
and use~\eqref{ineqentire}.
We can also assume, possibly at the price of increasing of~$K$, that
\begin{equation}	\label{ineqwhazero}
\left|  \wh a_0(\ze) \right|  \le K |\ze| \, \ee^{C|\ze|},
\qquad \ze\in\C,
\end{equation}
since $a_0(x) \in x^2\C\{x\}$.

%%%%%%%%%%%%%%%%%%%%%%%%%%%%%%%%%%%%%%%%%%%%%%%%%%%%%%%%%%%%%%%%%%%%%%%%%
%%%%%%%%%%%%%%%%%%%%%%%%%%%%%%%%%%%%%%%%%%%%%%%%%%%%%%%%%%%%%%%%%%%%%%%%%

\parag 
Next, we define exhaustions of $\gR(n-\N^*)$, resp.\ $\gR(\N)$, by subsets
$\gRN(n-\N^*)$, resp.\ $\gRN(\N)$, in which we shall be able to derive
appropriate bounds for our functions.
Let $\rho\in\left]0,\demi\right[$ and $N\in\N^*$.

We denote by $\bgRN(n-\N^*)$ the subset of~$\C$ obtained by removing the open discs
$D(m,\rho)$ with radius~$\rho$ and integer centres $m\le n-1$,
and removing also the points~$\ze$ such that the segment $[0,\ze]$ intersect the
open disc $D(-N,\rho)$ (\ie\ the points which are hidden by
$D(-N,\rho)$ to an observer located at the origin).

Similarly, we denote by $\bgRN(\N)$ the subset of~$\C$ obtained by removing the open discs
$D(m,\rho)$ with radius~$\rho$ and integer centres $m\ge 0$,
and removing also the points~$\ze$ such that the segment $[0,\ze]$ intersect the
open disc $D(N,\rho)$.
Thus, with the notations $\gP^- = n-\N^*$ and $\gP^+ = \N$,
\begin{multline*}
\bgRN(\gP^\pm) = \ao \ze\in\C \mid
\Dist(\ze,\gP^\pm)  \ge \rho \;\text{and}\; 
\Dist(\pm N,[0,\ze]) \ge \rho \af \\[.7ex]
= \ao \ze\in\C \mid 
\Dist\big( \ze, \gP^\pm \cup \pm\Sig(\rho,N) \big) \ge \rho \af,
\end{multline*}
with the notation~$\Sig$ introduced after the statement of
Theorem~\ref{thmResur}; see Figure~\ref{figrhoNadapt}.

Now, for $\gP=\gP^\pm$, consider the rectifiable oriented paths~$\ga$ which start at the origin and either
stay in the disc $D(0,\rho)$, or leave it and then stay in~$\bgRN(\gP)$. The
homotopy classes of such paths form a set~$\gRN(\gP)$ which we can identify with a
subset of~$\gR(\gP)$.

\begin{definition}	\label{defiadaptedpath}
If the arc-length parametrisation of a rectifiable oriented path 
$\ga \colon [0,T]  \to \C$ satisfies, for each $t\in[0,T]$,
\begin{align*}
0\le t\le \rho &\ens\Rightarrow\ens |\ga(t)|  = t, \\
t>\rho &\ens\Rightarrow\ens \ga(t) \in \bgRN(\gP),
\end{align*}
then we say that the parametrised path~$\ga$ is $(\rho,N,\gP)$-adapted.
We speak of infinite $(\rho,N,\gP)$-adapted path if $\ga$ is defined on
$\left[0,+\infty\right[$. 
\end{definition}

\begin{figure}

\begin{center}

\psfrag{n}{$\scriptstyle n$}
\psfrag{m}{$\scriptstyle n-1$}
\psfrag{0}{$\scriptstyle \mbox{}\hspace{.2em}0$}
\psfrag{N}{$\scriptstyle \mbox{}\hspace{.4em}N$}
\psfrag{M}{$\scriptstyle \mbox{}\hspace{-1em}-N$}

\psfrag{Pp}{$\gP^+ = \N$}
\psfrag{Pm}{$\gP^- = n-\N^*$}
\psfrag{S}{$\scriptstyle \Sig(N,\rho)$}
\psfrag{mS}{$\scriptstyle -\Sig(N,\rho)$}

\epsfig{file=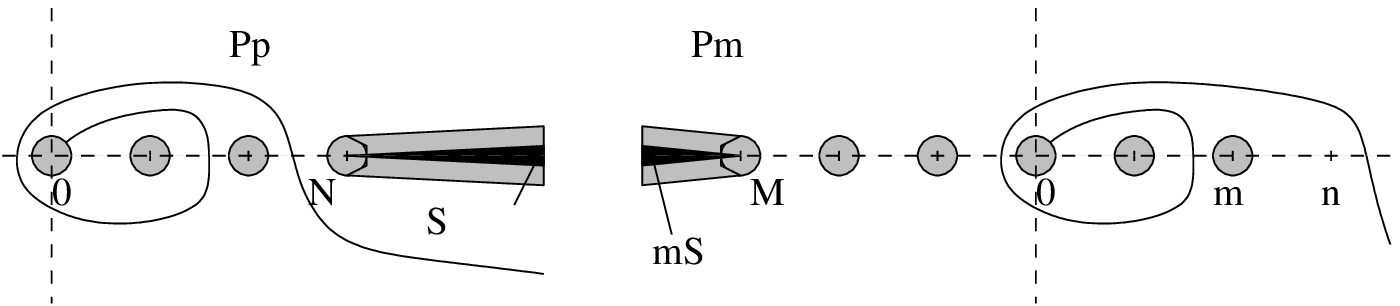,height=2.5cm,angle = 0}

\end{center}

%\vspace{-.5cm}

\caption{ \label{figrhoNadapt} 
The set $\protect\bgRN(\gP^\pm)$ 
and the image of a $(\rho,N,\gP^\pm)$-adapted path.}

\end{figure}

One can characterize~$\gRN(\gP)$ as follows:
a point of $\gR(\gP)$ belongs to~$\gRN(\gP)$ iff it can be represented by a
$(\rho,N,\gP)$-adapted path.

Observe that the projection onto~$\C$ of $\gRN(\gP)$ is $\bgRN(\gP)\cup
D(0,\rho)$ (only for $\gP=-\N^*$ is $D(0,\rho)$ contained in~$\bgRN(\gP)$) and
that $\gR(\gP) = \bigcup_{\rho,N} \gRN(\gP)$.

%%%%%%%%%%%%%%%%%%%%%%%%%%%%%%%%%%%%%%%%%%%%%%%%%%%%%%%%%%%%%%%%%%%%%%%%%
%%%%%%%%%%%%%%%%%%%%%%%%%%%%%%%%%%%%%%%%%%%%%%%%%%%%%%%%%%%%%%%%%%%%%%%%%

\parag 
We now show how to control the operators~$\cS_m$ and~$\tcS_m$ uniformly
in a set~$\gRN(\gP^\pm)$:

\begin{lemma}	\label{leminiSmtSm}
Let $n\in\N$ and $\cS_m,\tcS_m$ as in~\eqref{eqdefiStSm}, and consider the
meromorphic functions $S_m = \cS_m 1$ and $\tS_m = \tcS_m 1$.

Given $\rho,N$ as above, there exist $\la>0$ which depends only on $\rho,N$ and
$\la_n>0$ which depends only on $\rho,N,n$ such that, for
$m\in\gP\setminus\{0\}$,
\begin{alignat}{3}	
\label{ineqSm}
&\text{if $\gP = n-\N^*$:}& \qquad &
|S_m(\ze)| \le \la_n & \quad & 
\text{for}\ens \ze\in\bgRN(\gP)\cup D(0,\rho) \\
\tag{\ref{ineqSm}$'$}
&\text{if $\gP = \N$:}& \qquad &
|\tS_m(\ze)| \le \la & \quad &
\text{for}\ens \ze\in\bgRN(\gP)\cup D(0,\rho)
\end{alignat}
and
\begin{alignat}{3}
\label{ineqSzeroout}
&|S_0(\ze)| \le \la_n,& \qquad &|\tS_0(\ze)|  \le \la,& \qquad 
&\text{for}\ens \ze\in \C\setminus D(0,\rho) \\
\label{ineqSzeroin}
&|S_0(\ze)| \le \frac{\rho \la_n}{|\ze|},& \qquad &|\tS_0(\ze)|  \le \frac{\rho \la}{|\ze|},& \qquad 
&\text{for}\ens \ze\in D(0,\rho).
\end{alignat}
One can take $\la = (N+1)\rho\ii$ and $\la_n = (|n|+N)\rho\ii$.
\end{lemma}

\begin{proof}
Let $m\in\gP\setminus\{0\}$ and $\ze\in\bgRN(\gP)\cup D(0,\rho)$,
thus $|\ze-m|\ge\rho$.

Consider first the case $\gP=\N$.
If $m\ge N$, then $|\ze-m| \ge \frac{\rho |m|}{N}$ by Thales theorem; 
thus $\frac{1}{|\ze-m|} \le \rho\ii$ and $\big|\frac{m}{\ze-m}\big| \le
N\rho\ii$ for any $m\in\N^*$.
Therefore $|\tS_m(\ze)| = \big|\frac{m+1}{\ze-m}\big| \le \la = (N+1)\rho\ii$.
Since $\la\ge\rho\ii$, $\tS_0(\ze) = 1/\ze$ also satisfies the required inequalities.

When $\gP=n-\N^*$, one argues similarly except that the case $N\le m \le n-1$
must be treated separately.
\end{proof}

%%%%%%%%%%%%%%%%%%%%%%%%%%%%%%%%%%%%%%%%%%%%%%%%%%%%%%%%%%%%%%%%%%%%%%%%%
%%%%%%%%%%%%%%%%%%%%%%%%%%%%%%%%%%%%%%%%%%%%%%%%%%%%%%%%%%%%%%%%%%%%%%%%%

\parag 
Combining the previous two lemmas, we get

\begin{lemma}	\label{lemIneqfond}
Let us fix $n,\rho,N$ as above,
$K,L,C$ as in~\eqref{ineqwhaeta}--\eqref{ineqwhazero} and $\la,\la_n$ as in
Lemma~\ref{leminiSmtSm}. 
Suppose that $\gP = n-\N^*$ or~$\N$,
$\ga \colon [0,T] \to \C$ is $(\rho,N,\gP)$-adapted and
$\wh\ph \in \wHR\gP$ satisfies
$$
\big| \wh\ph\big( \ga(t) \big) \big| \le P(t) \,\ee^{Ct}, \qquad
0  \le t \le T,
$$
with a continuous monotonic non-decreasing function~$P$ and a constant $C\ge0$.
Assume $m\in\gP$, with the restriction $m\neq0$ if $n=0$ and $\gP=-\N^*$.

Then, for any $\eta\in\Om$, 
$$
\gP= n-\N^* \;\Rightarrow\;
\cS_m \cA_\eta\wh\ph \in \wHR{n-\N^*}, \quad
\gP= \N \;\Rightarrow\;
\tcS_m \cA_\eta\wh\ph \in \wHR\N,
$$
and, in the first case,
\begin{align}
m\neq0 \ens\text{or}\ens \eta=0 &\ens\Rightarrow\ens 
\big| \cS_m \cA_\eta\wh\ph\big( \ga(t) \big) \big| \le
\la_n K L^\eta (1*P)(t) \, \ee^{C t}
%
%\big| \tcS_m \cA_\eta\wh\ph\big( \ga(t) \big) \big| \le
%
%\la K L^\eta (1*P)(t) \, \ee^{C t}, 
%
\\
m=0 \ens\text{and}\ens \eta\neq0 &\ens\Rightarrow\ens 
\big| \cS_m \cA_\eta\wh\ph\big( \ga(t) \big) \big| \le
\la_n K L^\eta \big((\de+1)*P\big)(t) \, \ee^{C t}
%
%\big| \tcS_m \cA_\eta\wh\ph\big( \ga(t) \big) \big| \le
%
%\la K L^\eta (P+1*P)(t) \, \ee^{C t}
%
\end{align}
for all $t\in[0,T]$, while in the second case
the function~$\tcS_m \cA_\eta\wh\ph$ satisfies the same inequalities with~$\la$ replacing~$\la_n$.
\end{lemma}

\begin{proof}
We suppose $\gP=n-\N^*$ and show the properties for $\cS_m \cA_\eta\wh\ph$ only,
the other case being similar.
Since $\cS_m \cA_\eta\wh\ph = S_m (\wh a_\eta * \wh\ph)$, this function belongs to
$\wHR\gP$ by the first part of Lemma~\ref{lemAnCont}.
In view of~\eqref{ineqwhaeta}--\eqref{ineqwhazero}, the second part of this lemma yields
\begin{align}
\label{ineqAetawhph}
\big| \cA_\eta \wh\ph\big( \ga(t) \big) \big| &\le K L^\eta (1*P)(t) \,\ee^{Ct}
\\
\label{ineqAzerowhph}
\big| \cA_0 \wh\ph\big( \ga(t) \big) \big| &\le K (I*P)(t) \,\ee^{Ct}
\end{align}
for all $t\in[0,T]$, with $I(t)\equiv t$ (notice that the first inequality holds if $\eta=0$
as well).

If $m\neq0$, then \eqref{ineqSm} yields the desired inequality for $\big| \cS_m
\cA_\eta \wh\ph\big( \ga(t) \big) \big|$.

Suppose $m=0$; thus $n\neq0$ by assumption. 
We observe that, if $t>\rho$, then $\ga(t) \in \bgRN(\gP)$ has modulus $>\rho$ and 
\eqref{ineqSzeroout} yields 
$\big| S_0\big( \ga(t) \big) \big| \le \la_n$,
whereas if $t\le\rho$, then $|\ga(t)| = |t|$ and
\eqref{ineqSzeroin} yields 
$\big| S_0\big( \ga(t) \big) \big| \le \frac{\rho \la_n}{t}$.

Thus, if $m=0$ and $\eta=0$, then \eqref{ineqAetawhph} yields the desired inequality when
$t>\rho$ and \eqref{ineqAzerowhph} yields 
$\big| \cS_0 \cA_0 \wh\ph\big( \ga(t) \big) \big| \le K\rho\la_n
\frac{I*P(t)}{t} \,\ee^{Ct}$
for $t\le\rho$, which is sufficient since 
$\frac{I*P(t)}{t} = \frac{1}{t} \int_0^t t' P(t-t') \,\dd t' \le 1*P(t)$ and $\rho<1$.

We conclude with the case where $m=0$ and $\eta\neq0$.
Using \eqref{ineqAetawhph}, we obtain the result when $t>\rho$, since $1*P\le
P+1*P$.
When $t\le\rho$, we get 
$\big| \cS_0 \cA_\eta \wh\ph\big( \ga(t) \big) \big| \le K L^\eta\rho\la_n
\frac{1*P(t)}{t} \,\ee^{Ct}$,
which is sufficient since 
$\frac{1*P(t)}{t} = \frac{1}{t} \int_0^t P(t') \,\dd t' \le P(t)$. % \le P(t) + 1*P(t)$.
\end{proof}

%%%%%%%%%%%%%%%%%%%%%%%%%%%%%%%%%%%%%%%%%%%%%%%%%%%%%%%%%%%%%%%%%%%%%%%%%
%%%%%%%%%%%%%%%%%%%%%%%%%%%%%%%%%%%%%%%%%%%%%%%%%%%%%%%%%%%%%%%%%%%%%%%%%

\parag 
\emph{End of the proof of Theorem~\ref{thmResur}:} case of~$\wh\ph_n$.

\medskip

Let $n\in\N$. According to~\eqref{eqwhphnwhpsin}, the formal series $\wh\ph_n$
can be written as the formally convergent series
$\sum_{\norm{\omb}=n-1} \be_\omb  \wh\cV^\omb$.
Let $\gP = n-\N^*$; according to Corollary~\ref{corbecVo} each
$\be_\omb\wh\cV^\omb$ converges to a function of~$\wHR\gP$,
it is thus sufficient to check the uniform convergence of the above series
\emph{as a series of holomorphic functions} in each compact subset
of~$\gR(\gP)$ and to give appropriate bounds.
Let us fix $\rho\in\left]0,\demi\right[$, $N\in\N^*$ and $K,L,C,\la,\la_n$ as
in Lemma~\ref{lemIneqfond}.

%%%%%%%%%%%%%%%%%%%%%%%%%%%%%%%%%%%%%%%%%%%%%%%%%%%%%%%%%%%%%%%%%%%%%%%%%

\medskip

\noindent -- 
We first show that, for any $(\rho,N,\gP)$-adapted path~$\ga$ (infinite or not)
and for any $\omb=(\om_1,\dotsc,\om_r)\in\Om^r$ with $r\ge1$ and $\norm{\omb}=n-1$,
one has for all~$t$
\begin{equation}	\label{ineqbecV}
\big|  \be_\omb \wh\cV^\omb\big( \ga(t)  \big) \big| \le
(\la_n K)^r L^{n-1} \wh P_r(t) \, \ee^{Ct},
\qquad \wh P_r = (\de+1)^{*\flo{r/2}} * 1^{*\ceil{r/2}},
\end{equation}
with the same notation as in~\eqref{eqvalcV} for $\ceil{r/2}$ and
$\flo{r/2}=r-\ceil{r/2}$.
Observe that $\wh P_r$ is a polynomial with non-negative coefficients.

If $\om_r=\wc\om_r\neq0$, then \eqref{ineqwhaeta} and \eqref{ineqSm} yield
$|  \cS_{\shtb\om_r}  \wh a_{\om_r}(\ze) | \le \la_n K L^{\om_r}\, \ee^{C|\ze|}$
for all $\ze\in\gR(\gP)$.
The same inequality holds also if $\om_r=0$ (use \eqref{ineqwhaeta}
and~\eqref{ineqSzeroout} if $|\ze|>\rho$, and  \eqref{ineqwhazero}
and~\eqref{ineqSzeroin} if $|\ze|\le\rho$). Therefore
\begin{equation}	\label{ineqiniomr}
\big|  \cS_{\shtb\om_r}  \wh a_{\om_r}\big(\ga(t)\big) \big| \le \la_n K L^{\om_r}\, \ee^{Ct},
\qquad t\ge0
\end{equation}	
(since $|\ga(t)|\le t$).
Since Lemma~\ref{lemexo} implies $\htb\om_1,\dotsc,\htb\om_{r-1} \le n-1$, we can
apply $r-1$ times Lemma~\ref{lemIneqfond} and get
\begin{equation}	\label{ineqba}
\big|  \be_\omb \wh\cV^\omb\big( \ga(t)  \big) \big| \le
(\la_n K)^r L^{n-1} \big( (\de+1)^{*(r-a)} * 1^{*a}\big)(t) \, \ee^{Ct},
\end{equation}
with $a = \card \ao i\in[1,r]  \mid \htb\om_i\neq 0 \;\text{or}\; \om_i=0  \af$.
But $a\ge\ceil{r/2}$, as was shown in Lemma~\ref{lemdefcV}, hence the polynomial
expression in~$t$ appearing in the \rhs\ of~\eqref{ineqba} can be written
$(\de+1)^{*(r-a)} * 1^{*(a-\ceil{r/2})} * 1^{*\ceil{r/2}} \le 
(\de+1)^{*(r-\ceil{r/2})} * 1^{*\ceil{r/2}}$,
which yields~\eqref{ineqbecV}.

%%%%%%%%%%%%%%%%%%%%%%%%%%%%%%%%%%%%%%%%%%%%%%%%%%%%%%%%%%%%%%%%%%%%%%%%%

\medskip

\noindent -- 
We have
\begin{multline*}
\card\ao \omb\in\Om^r \mid \norm{\omb}=n-1 \af = 
\card\ao k\in\N^r \mid \norm{k}=n+r-1 \af \\[1ex]
=\binom{n+2(r-1)}{r-1} \le 2^{n+2(r-1)},
\end{multline*}
hence, for each $r\ge1$, 
\begin{equation}	\label{ineqCVUdeux}
\sum_{r(\omb)=r,\norm{\omb}=n-1} 
\big|  \be_\omb \wh\cV^\omb\big( \ga(t)  \big) \big| \le
2 \la_n K (2L)^{n-1} \La_n^{r-1} \wh P_r(t) \, \ee^{Ct}
\end{equation}
with $\La_n = 4\la_n K$.
But $\wt P_r(z) = \cB\ii\wh P_r = (1+z\ii)^{\flo{r/2}} z^{-\ceil{r/2}}$ gives
rise to
$$
\wt\Phi_n(z) = \sum_{r\ge1} \La_n^{r-1}  \wt P_r(z) =
z\ii \big( 1+\La_n(1+z\ii) \big) \big( 1 - \La_n^2(z\ii+z^{-2}) \big)\ii
$$
which is convergent (with non-negative coefficients), thus
$\dst\sum_{r\ge1}  \La_n^{r-1}  \wh P_r(t) = \cB\wt\Phi_n(t)$ is convergent for all~$t$.
Therefore $\wh\ph_n$ is the sum of a series of holomorphic functions uniformly
convergent in every compact subset of~$\gR_{\rho,N}(\gP)$ satisfying
$$
\big| \wh\ph_n\big(\ga(t)\big) \big| \le 2\la_n K (2L)^{n-1} \cB\wt\Phi_n(t)\,\ee^{C t}.
$$

%%%%%%%%%%%%%%%%%%%%%%%%%%%%%%%%%%%%%%%%%%%%%%%%%%%%%%%%%%%%%%%%%%%%%%%%%

\medskip

\noindent -- 
We conclude by using inequalities of the form~\eqref{ineqentire} to bound
$\cB\wt\Phi_n$:
one can check that $|z|\ge 4\La_n^2$ implies 
$|z\wt\Phi_n(z)| \le 2(2+\La_n)$, hence
$$
\cB\wt\Phi_n(t) \le 2(2+\La_n) \, \ee^{4\La_n^2 t}.
$$
In view of the explicit dependence of~$\la_n$ on~$n$ indicated in
Lemma~\ref{leminiSmtSm}, we easily get inequalities of the
form~\eqref{ineqwhphn} (possibly with larger constants $K,L,C$).

%%%%%%%%%%%%%%%%%%%%%%%%%%%%%%%%%%%%%%%%%%%%%%%%%%%%%%%%%%%%%%%%%%%%%%%%%
%%%%%%%%%%%%%%%%%%%%%%%%%%%%%%%%%%%%%%%%%%%%%%%%%%%%%%%%%%%%%%%%%%%%%%%%%

\parag 
\emph{End of the proof of Theorem~\ref{thmResur}:} case of~$\wh\psi_n$.

%%%%%%%%%%%%%%%%%%%%%%%%%%%%%%%%%%%%%%%%%%%%%%%%%%%%%%%%%%%%%%%%%%%%%%%%%

\medskip

We only indicate the inequalities that one obtains when adapting the previous arguments to
the case of~$\wh\psi_n$.
Let $\gP=\N$ and
$$
\wh\chi_n(\ze) = \big(\ze-(n-1)\big)\wh\psi_n(\ze), \qquad
\wh\cW^\omb(\ze) = \big(\ze-(n-1)\big)\be_\omb\hcVto(\ze).
$$
The initial bound corresponding to~\eqref{ineqiniomr} is
$\big|  \cS_{\swc\om_1}  \wh a_{\om_1}\big(\ga(t)\big) \big| \le 
\la K L^{\om_1}\, \ee^{Ct}$.
This yields 
\begin{equation}	\label{ineqCVUtrois}
\big| \wh\cW^\omb\big( \ga(t)  \big) \big| \le
K (\la K)^{r-1} L^{n-1} \wh Q_r(t) \, \ee^{Ct},
\qquad \wh Q_r = (\de+1)^{*\ceil{\frac{r}{2}-1}} * 1^{*\flo{\frac{r}{2}+1}}
\end{equation}
after $r-2$ applications of Lemma~\ref{lemIneqfond}
(with an intermediary inequality analogous to~\eqref{ineqba} but involving
$b = 1 + \card \ao i\in[1,r-1]  \mid \wc\om_i\neq 0 \;\text{or}\; \om_i=0  \af
\ge \flo{\frac{r}{2}+1}$ instead of~$a$).

Therefore 
$\big| \wh\chi_n\big(\ga(t)\big) \big| \le 2 K (2L)^{n-1} \cB\wt\Psi(t)\,\ee^{C t}$,
with
$$
\wt\Psi(z) = \sum_{r\ge1} \La^{r-1}  \cB\ii\wh Q_r(z) =
z\ii (1+\La z\ii) \big( 1 - \La^2(z\ii+z^{-2}) \big)\ii,
\qquad \La = 4\la K,
$$
whence $\big| \wh\chi_n\big( \ga(t) \big) \big|  \le K_1 (2L)^n \, \ee^{C_1 t}$,
with suitable constants $K_1,C_1$ independent of~$n$.

This is the desired conclusion when $n=0$. When $n\ge1$, we can pass
from~$\wh\chi_n$ to~$\wh\psi_n$ since $|\ga(t)-(n-1)|\ge\rho$, 
with only one exception; namely, if $n=1$ and $t<\rho$, then we only have a
bound for $|\ze\wh\psi_1(\ze)|$ with $\ze=\ga(t)\in D(0,\rho)$, but in that case
the analyticity of~$\wh\chi_1$ at the origin of~$\gR(\N)$ is sufficient since we
know that its Taylor series has no constant term.

%%%%%%%%%%%%%%%%%%%%%%%%%%%%%%%%%%%%%%%%%%%%%%%%%%%%%%%%%%%%%%%%%%%%%%%%%
%%%%%%%%%%%%%%%%%%%%%%%%%%%%%%%%%%%%%%%%%%%%%%%%%%%%%%%%%%%%%%%%%%%%%%%%%

\parag 
\emph{Proof of inequalities~\eqref{inequniformphn}.}
They will follow from a lemma which has its own interest.

%%%%%%%%%%%%%%%%%%%%%%%%%%%%%%%%%%%%%%%%%%%%%%%%%%%%%%%%%%%%%%%%%%%%%%%%%

\begin{lemma}	\label{lemLagrpsinphn}
For every $n\in\N$, the following identity holds in~$\C[[x]]$:
\begin{equation}	\label{eqLagrpsinphn}
\ph_n = \sum_{ \substack{s\ge1,\ n_1,\dotsc,n_s\ge0 \\
n_1 + \dotsb + n_s = n+s-1}}
\frac{(-1)^s}{s}  \binom{n+s-1}{s-1} 
\psi_{n_1}  \dotsm \psi_{n_s},
\end{equation}
where the \rhs\ is a formally convergent series.
\end{lemma}

%%%%%%%%%%%%%%%%%%%%%%%%%%%%%%%%%%%%%%%%%%%%%%%%%%%%%%%%%%%%%%%%%%%%%%%%%

\begin{proof}
This is the consequence of the following version of Lagrange inversion formula:
If $\chi(t,y)\in\C[[t,y]]$, then the formal transformation
$$ 
(t,x,y) \mapsto \big( t, x, y-x\chi(t,y) \big)
$$
has an inverse of the form
$(t,x,y) \mapsto \big( t, x, \gY(t,x,y) \big)$
with $\gY \in \C[[t,x,y]]$ given by
\begin{equation} 	\label{eqforminvLagr}
\gY(t,x,y) = y + \sum_{s\ge1} \, \frac{x^s}{s!} \, 
\Big(\frac{\pa\,}{\pa y}\Big)^{\!s-1} \big( \chi(t,y)^s \big).
\end{equation}
(Proof: The transformation is invertible, because its $1$-jet is, and
the inverse must be of the form $\big(t,x,\gY(t,x,y)\big)$ with $\gY(t,0,y)=y$
and $\pa_y\gY(t,0,y)=1$. 
It is thus sufficient to check the formula
$$
\pa_x^s \gY(t,x,y) = 
\Big(\frac{\pa\,}{\pa y}\Big)^{\!s-1} \bigg[ \Big( \chi\big(t, \gY(t,x,y) \big) \Big)^{\!s} 
\pa_y\gY(t,x,y) \bigg], \qquad s\ge1
$$
by induction on~$s$, which is easy.)

Since $\psi_n(x) \in x\C[[x]]$, we can apply this with $\chi(t,y) = \sum_{n\ge0}
\chi_n(t) y^n$ where $\chi_n(x) = -\frac{\psi_n(x)}{x}$:
this way $y-x\chi(x,y) = y + \sum_{n\ge0}  \psi_n(x) y^n = \psi(x,y)$, and
\eqref{eqforminvLagr} yields
$$
\ph(x,y) = y + \sum_{n\ge0} \ph_n(x) y^n = 
\sum_{s\ge1} \, \frac{(-1)^s}{s!} \, 
\Big(\frac{\pa\,}{\pa y}\Big)^{\!s-1} \bigg[ 
\bigg( \sum_{n\ge0} \psi_n(x) y^n \bigg)^{\!\!s}\, \bigg]
$$
by specialization to $t=x$, whence the result follows (one gets a formally
convergent series because $\psi_n(x) \in x\C[[x]]$).
\end{proof}

%%%%%%%%%%%%%%%%%%%%%%%%%%%%%%%%%%%%%%%%%%%%%%%%%%%%%%%%%%%%%%%%%%%%%%%%%

As a consequence, we get
\begin{equation*}	
\wh\ph_n = \sum_{ \substack{s\ge1,\ n_1,\dotsc,n_s\ge0 \\
n_1 + \dotsb + n_s = n+s-1}}
\frac{(-1)^s}{s}  \binom{n+s-1}{s-1} 
\wh\psi_{n_1} *  \dotsm * \wh\psi_{n_s}
\end{equation*}
a priori in~$\C[[\ze]]$, but the \rhs\ is also a series of holomorphic functions
and inequalities~\eqref{ineqwhpsin} will yield uniform convergence in every
compact subset of the principal sheet of~$\gR(\Z)$.

Indeed, let $\rho\in\left]0,\demi\right[$.
The domain considered in~\eqref{inequniformphn} consists of those $\ze\in\C$
such that the segment $[0,\ze]$ does not meet the open discs $D(-1,\rho)$ and
$D(1,\rho)$. 
All the $\wh\psi_n$'s are holomorphic in this domain~$\gD_\rho$ (we had to delete the disc
around~$-1$ only because of~$\wh\psi_0$).

Since~$\gD_\rho$ is star-shaped \wrt~$0$, the analytic continuation of the
convolution product of any two functions~$\wh\ph$ and~$\wh\psi$ holomorphic
in~$\gD_\rho$ is defined by formula~\eqref{eqdefconvol} regardless of the size
of~$|\ze|$. 
If moreover one has inequalities of the form
$|\wh\ph(\ze)| \le \Phi(|\ze|)\,\ee^{C|\ze|}$ and
$|\wh\psi(\ze)| \le \Psi(|\ze|)\,\ee^{C|\ze|}$ in~$\gD_\rho$, 
then the inequality
$|\wh\ph*\wh\psi(\ze)|  \le \Phi*\Psi(|\ze|)\,\ee^{C|\ze|}$ holds in~$\gD_\rho$.
Hence
$$
|\wh\ph_n(\ze)| \le \sum_{ \substack{s\ge1,\ n_1,\dotsc,n_s\ge0 \\
n_1 + \dotsb + n_s = n+s-1}}
\frac{1}{s}  \binom{n+s-1}{s-1} 
K^s L^{n+s-1} M_s(|\ze|)  \, \ee^{C|\ze|},
\qquad \ze\in\gD_\rho,
$$
with $M_s(\ze) = 1^{*s}(\ze) = \frac{\ze^{s-1}}{(s-1)!}$.
The conclusion follows since the \rhs\ is less than $K (4L)^n \,\ee^{(C+8KL)|\ze|}$.

%%%%%%%%%%%%%%%%%%%%%%%%%%%%%%%%%%%%%%%%%%%%%%%%%%%%%%%%%%%%%%%%%%%%%%%%%
%%%%%%%%%%%%%%%%%%%%%%%%%%%%%%%%%%%%%%%%%%%%%%%%%%%%%%%%%%%%%%%%%%%%%%%%%

\section{The $\wt\cV^\omb$'s as resurgence
monomials---introduction to alien calculus}	\label{secBESN} 

%%%%%%%%%%%%%%%%%%%%%%%%%%%%%%%%%%%%%%%%%%%%%%%%%%%%%%%%%%%%%%%%%%%%%%%%%
%%%%%%%%%%%%%%%%%%%%%%%%%%%%%%%%%%%%%%%%%%%%%%%%%%%%%%%%%%%%%%%%%%%%%%%%%

\parag
Resurgence theory means much more than Borel-Laplace summation.
It incorporates a study of the role of the singularities which appear in the
Borel plane (\ie\ the plane of the complex variable~$\ze$), 
which can be performed through the so-called {\em alien calculus}.

We shall now recall \'Ecalle's definitions in a particular case which will suffice for
the saddle-node problem.
We shall give less details than in the previous section; see \eg~\cite{kokyu},
\S2.3 for more information (and {\em op.\ cit.}, \S3 for an outline of the
general case and more references).

The reader will thus find in this section the definition of a subalgebra
$\tRsimpZ$ of~$\tRZ$, which is called the algebra of {\em simple resurgent
functions over~$\Z$},  
and of a collection of operators~$\De_{m}$, $m\in\Z^*$, which are
derivations of~$\tRsimpZ$ called {\em alien derivations}.
Alien calculus consists in the proper use of these derivations.

We shall see that the formal series $\wt\cV^{\om_1,\dotsc,\om_r}$ belong
to~$\tRsimpZ$ and study the effect of the alien derivations on them.

%%%%%%%%%%%%%%%%%%%%%%%%%%%%%%%%%%%%%%%%%%%%%%%%%%%%%%%%%%%%%%%%%%%%%%%%%
%%%%%%%%%%%%%%%%%%%%%%%%%%%%%%%%%%%%%%%%%%%%%%%%%%%%%%%%%%%%%%%%%%%%%%%%%

\parag
Let $\wh\ph$ be holomorphic in an open subset~$U$ of~$\C$ and $\om\in\pa U$.
We say that~$\wh\ph$ has a {\em simple singularity} at~$\om$ if there exist
$C\in\C$ and $\wh\chi(\ze),\mathrm{reg}(\ze)\in\C\{\ze\}$ such that
\begin{equation}	\label{eqsimplesing}
\wh\ph(\ze) =
\frac{C}{2\pi\I(\ze-\om)} + \frac{1}{2\pi\I}\wh\chi(\ze-\om)\log(\ze-\om) 
+ \mathrm{reg}(\ze-\om)
\end{equation}
for $\ze$ close enough to~$\om$. The {\em residuum}~$C$ and the
{\em variation}~$\wh\chi$ are then determined by~$\wh\ph$ (independently of the
choice of the branch of the logarithm):
$$
C =  2\pi\I %\ze\xrightarrow[\text{in}\,U]{}\om}
\lim_{\substack{\ze\to\om  \\  {\scriptscriptstyle \ze\in U}}}
(\ze-\om)\wh\ph(\ze), \qquad
\wh\chi(\ze) = \wh\ph(\om+\ze) - \wh\ph(\om+\ze\,\ee^{-2\pi\I}),
$$
where it is understood that considering $\om+\ze\,\ee^{-2\pi\I}$ means following
the analytic continuation of~$\wh\ph$ along the circular path 
$t\in[0,1]\mapsto \om+\ze\,\ee^{-2\pi\I t}$ 
(which is possible when starting from $\om+\ze\in U$ provided $|\ze|$ is small enough).
Let us use the notation
$$
\Sing_\om \wh\ph = C\,\de + \wh\chi \,\in\, \C\,\de \oplus \C\{\ze\}.
$$
in this situation.

We recall that $\hRZ = \C\, \de  \oplus \wHR\Z$.
\begin{definition}	
A simple resurgent function over~$\Z$ is any $c\,\de + \wh\ph \in \hRZ$ such that all
branches of the holomorphic fuction $\wh\ph\in\wh H\big(\gR(\Z)\big)$ only have
simple singularities (necessarily located at points of~$\Z$).
The space of simple resurgent functions over~$\Z$ will be denoted~$\hRsimpZ$.
\end{definition}

It turns out that \emph{$\hRsimpZ$ is stable by convolution:} it is a subalgebra
of~$\hRZ$.
This is the \emph{convolutive model of the algebra of simple resurgent functions}.
The \emph{formal model} is defined as $\tRsimpZ = \cB\ii(\hRsimpZ)$, which is a
subalgebra of~$\tRZ$.

%%%%%%%%%%%%%%%%%%%%%%%%%%%%%%%%%%%%%%%%%%%%%%%%%%%%%%%%%%%%%%%%%%%%%%%%%
%%%%%%%%%%%%%%%%%%%%%%%%%%%%%%%%%%%%%%%%%%%%%%%%%%%%%%%%%%%%%%%%%%%%%%%%%

\parag
For a simple resurgent function $c\,\de + \wh\ph$ and a path~$\ga$ which starts
from~$0$ and then avoids~$\Z$, we shall denote by $\cont_\ga\wh\ph$ the analytic
continuation of~$\wh\ph$ along~$\ga$: this function is analytic in a
neighbourhood of the endpoint of~$\ga$ and admits itself an analytic
continuation along all the paths which avoid~$\Z$.
If the endpoint of~$\ga$ is close to~$m$ (say at a distance
$<\demi$), then the singularity 
$\Sing_m(\cont_\ga\wh\ph) \in \C\,\de \oplus \C\{\ze\}$ 
is well-defined
(notice that it depends on the branch under consideration, \ie\ on~$\ga$, and
not only on~$m$ and~$\wh\ph$).
It is easy to see that $\Sing_m(\cont_\ga\wh\ph)$ is itself a simple resurgent
function; we thus have, for $\ga$ and~$m$ as above, a $\C$-linear operator
$c\,\de + \wh\ph \mapsto \Sing_m(\cont_\ga\wh\ph)$ from $\hRsimpZ$ to itself.

\begin{definition}	\label{defaliender}
Let $m\in\Z^*$. If $m\ge1$, we define an operator from $\hRsimpZ$ to
itself by using $2^{m-1}$ particular paths~$\ga$:
\begin{equation}    \label{eqdefDeomgen}
\De_m(c\,\de + \wh\ph) = \sum_{\eps\in\{+,-\}^{m-1}}
                \frac{p_\eps!q_\eps!}{m!} \Sing_m(\cont_{\ga_\eps}\wh\ph)
\end{equation}
where $p_\eps$ and $q_\eps=m-1-p_\eps$ denote the numbers of signs~`$+$' and
of signs~`$-$' in the sequence $\eps = (\eps_1,\dotsc,\eps_{m-1})$,
and the oriented path~$\ga_\eps$ connects~$0$ and~$m$ following the
segment~$\left]0,m\right[$ but circumventing the intermediary integer points~$k$
to the right if~$\eps_k=+$ and to the left if~$\eps_k=-$.

If $m\le-1$, the $\C$-linear operator $\De_m$ is defined similarly, using the
$2^{|m|-1}$ paths $\ga_\eps$ which follow~$\left]0,-|m|\right[$
but circumvent the intermediary integer points~$-k$ to the right
if~$\eps_k=+$ and to the left if~$\eps_k=-$.
\end{definition}

\begin{prop}	\label{propAlDer}
For each $m\in\Z^*$, the operator $\De_m$ is a $\C$-linear
derivation of~$\hRsimpZ$.
\end{prop}
\noindent
(For the proof, see \cite{Eca81} or \cite{kokyu}, \S2.3; see also
Lemma~\ref{lemrelDemDepm} and the comment on it below.)

By conjugacy by the formal Borel transform~$\cB$, we get a derivation
of~$\tRsimpZ$, still denoted~$\De_m$ since there is no risk of confusion.
The operator~$\De_m$ is called the \emph{alien derivation of index~$m$} (either in
the convolutive model $\hRsimpZ$ or in the formal model $\tRsimpZ$).

One can easily check from Definition~\ref{defaliender} that 
\begin{equation}	\label{eqcommutalien}
[\pa,\De_m] = m\De_m \ens\text{in $\tRsimpZ$},\qquad
[\wh\pa,\De_m] = m\De_m \ens\text{in $\hRsimpZ$},
\end{equation}
where $\pa$ denotes the natural derivation $\frac{\dd\,}{\dd z}$ of $\tRsimpZ$
and $\wh\pa$ is the corresponding derivation
$c\,\de + \wh\ph(\ze) \mapsto -\ze\wh\ph(\ze)$
of $\hRsimpZ$.

%%%%%%%%%%%%%%%%%%%%%%%%%%%%%%%%%%%%%%%%%%%%%%%%%%%%%%%%%%%%%%%%%%%%%%%%%
%%%%%%%%%%%%%%%%%%%%%%%%%%%%%%%%%%%%%%%%%%%%%%%%%%%%%%%%%%%%%%%%%%%%%%%%%

\parag
We shall see that the operators~$\De_m$ are independent in a strong sense
(see Theorem~\ref{thmfree} below).
This will rely on a study of the way the alien derivations act on the resurgent
functions $\wh\cV^{\om_1,\dotsc,\om_r}$. 

In this article, for the sake of simplicity, we shall not introduce the larger
commutative algebras $\hRsimp$ and $\tRsimp$ of simple resurgent functions
``over~$\C$'', \ie\ with simple singularities in the Borel plane which can be
located anywhere.
In these algebras act alien derivations indexed by any non-zero complex number.
One could easily adapt the arguments that we are about to develop to the study
of the alien derivations $\De_\om$, $\om \in \C^*$, in $\tRsimp$.

One can also define an even larger commutative algebra of resurgent functions, without any
restriction on the nature of the singularities to be encountered in the Borel
plane, on which act alien derivations $\De_\om$ indexed by points~$\om$ of the
Riemann surface of the logarithm, but there is no formal counterpart contained
in~$\C[[z\ii]]$ (see \eg~\cite{kokyu}, \S3, and the references therein).

%%%%%%%%%%%%%%%%%%%%%%%%%%%%%%%%%%%%%%%%%%%%%%%%%%%%%%%%%%%%%%%%%%%%%%%%%
%%%%%%%%%%%%%%%%%%%%%%%%%%%%%%%%%%%%%%%%%%%%%%%%%%%%%%%%%%%%%%%%%%%%%%%%%

\parag
We now check that the formal series $\wt\cV^{\om_1,\dotsc,\om_r}$ are simple
resurgent functions and slightly extend at the same time their definition.

\begin{lemma}	\label{defnewcVomb}
Let $\bA = \tRsimpZ$ and $\Om \subset \Z$.
Assume that $a = (\wh a_\eta)_{\eta\in\Om}$ is a family of entire functions;
if $0\in\Om$, we assume furthermore that $\wh a_0(0)=0$.
%
% $\ti a_0(z)\in z^{-2}\C\{z\ii\}$.
%
Let $\ti a_\eta = \cB\ii\wh a_\eta \in z\ii\C[[z\ii]]$.

Then the equations $\wt\cV_a^\est = \tcVae = 1$ and,
for $\omb \in \Om^\bul$ non-empty, 
\begin{equation}	\label{eqdefnewcV}
\big(\frac{\dd\,}{\dd z} + \norm{\omb}\big) \wt\cV_a^\omb
= \wt a_{\om_1} \wt\cV_a^{`\omb}, \qquad
\big(\frac{\dd\,}{\dd z} + \norm{\omb}\big) \tcVao
= -\wt a_{\om_r} \tcVaop
\end{equation}
(with $`\omb$ denoting $\omb$ deprived from its first letter,
$\ombp$ denoting $\omb$ deprived from its last letter and
$\norm{\omb}$ the sum of the letters of~$\omb$)
determine inductively two moulds $\wt\cV_a^\bul, \tcVab \in \hM^\bul(\Om,\bA)$, which
are symmetral and mutually inverse for mould multiplication.
\end{lemma} 

\begin{proof}
A mere adaptation of Lemma~\ref{lemdefcV} and Proposition~\ref{propcVsym} (in
which the fact that $\Om=\cN$ played no role) shows that $\wt\cV_a^\bul$ and
$\tcVab$ are well-defined by~\eqref{eqdefnewcV} as moulds on~$\Om$ {\em with values in
$\C[[z\ii]]$}, with $\wt\cV_a^\omb,\tcVao  \in z\ii\C[[z\ii]]$ as soon as
$\omb\neq\est$, that they are related by the involution~$S$ of Proposition~\ref{propinvsym}:
\begin{equation}	\label{eqrelStcVcV}
\tcVab = S\wt\cV_a^\bul
\end{equation}
and symmetral, hence  mutually inverse.

The formal Borel transforms are given by $\wh\cV_a^\est = \hcVae = \de$ and,
for $\omb\neq\est$,
\begin{equation}	\label{eqdefhcVhtcV}
\wh\cV_a^\omb(\ze) = -\frac{1}{\ze-\norm{\omb}} \big( \wh a_{\om_1}  * \wh\cV_a^{`\omb} \big), 
\qquad \hcVao(\ze) = \frac{1}{\ze-\norm{\omb}} \big( \wh a_{\om_r}  * \hcVaop \big), 
\end{equation}
where the right-hand sides belong to $\C[[\ze]]$ even if $\norm{\omb}=0$, by the same
argument as in the proof of Lemma~\ref{lemdefcV},
and in fact to $\C\{\ze\}$, by induction on $r(\omb)$.

Since $\norm{\omb}$ always lies in~$\Z$, we can apply Lemma~\ref{lemAnCont}: we
get
$\wh\cV_a^\omb, \hcVao \in \wh H\big(\gR(\Z)\big)$ for all $\omb\neq\est$ by
induction on~$r(\omb)$, hence our moulds take their values in $\hRZ$.

We see that the singularities of $\wh\cV_a^\omb$ and $\hcVao$ are all simple
singularities, because of the following addendum to Lemma~\ref{lemAnCont}:
with the hypotheses and notations of that lemma, if moreover
$\wh\ph\in\hRsimpZ$, then $\wh b*\wh\ph$ vanishes at the origin and only has
simple singularities with vanishing residuum
(this follows from the first formula in~\eqref{eqbstarwhph}),
hence $\wh s(\wh b*\wh\ph) \in \hRsimpZ$.
\end{proof}

Notice that, by iterating~\eqref{eqdefhcVhtcV}, one gets
\begin{gather}	\label{eqnewcViter}
\wh\cV_a^\omb = (-1)^r
\frac{1}{\ze-\htb\om_1} \Big( \wh a_{\om_1}  * 
\Big( \frac{1}{\ze-\htb\om_2} \Big( \wh a_{\om_2}  * 
\Big(  \dotsb 
\Big( \frac{1}{\ze-\htb\om_r} \wh a_{\om_r}
\Big) \dotsm \Big)\Big)\Big)\Big) \\
\label{eqnewtcViter}
\hcVao = 
\frac{1}{\ze-\wc\om_r} \Big( \wh a_{\om_r}  * 
\Big( \frac{1}{\ze-\wc\om_{r-1}} \Big( \wh a_{\om_{r-1}}  * 
\Big(  \dotsb 
\Big( \frac{1}{\ze-\wc\om_1} \wh a_{\om_1}
\Big) \dotsm \Big)\Big)\Big)\Big)
\end{gather}
with the notation of~\eqref{eqnotahatcheck}.
These are iterated integrals; for instance, the second formula can be written
\begin{multline}	\label{eqiteratedinthcVAo}
\hcVao(\ze) = \frac{1}{\ze-\wc\om_r} 
\underset{0<\ze_1<\dotsb<\ze_{r-1}<\ze}{\idotsint} 
\wh a_{\om_r}(\ze-\ze_{r-1})  \cdot \\[1ex]
\cdot \frac{\wh a_{\om_{r-1}}(\ze_{r-1}-\ze_{r-2})}{\ze_{r-1}-\wc\om_{r-1}}
\dotsm\frac{\wh a_{\om_{2}}(\ze_{2}-\ze_{1})}{\ze_{2}-\wc\om_{2}}
\frac{\wh a_{\om_{1}}(\ze_{1})}{\ze_{1}-\wc\om_{1}}
\,\dd\ze_{1}\dotsm\dd\ze_{r-1}
\end{multline}
and its analytic continuation along any parametrised path~$\ga$ which starts
from~$0$ and then avoids~$\Z$ is given by the same integral, but taken over all
$(r-1)$-tuples
$(\ze_1,\dotsc,\ze_{r-1}) = \big( \ga_\eps(t_1),\dotsc,\ga_\eps(t_{r-1}) \big)$
with $t_1 < \dotsb < t_{r-1}$.

%%%%%%%%%%%%%%%%%%%%%%%%%%%%%%%%%%%%%%%%%%%%%%%%%%%%%%%%%%%%%%%%%%%%%%%%%
%%%%%%%%%%%%%%%%%%%%%%%%%%%%%%%%%%%%%%%%%%%%%%%%%%%%%%%%%%%%%%%%%%%%%%%%%

\parag
We are now ready to study the alien derivatives of the resurgent functions
$\wt\cV_a^{\om_1,\dotsc,\om_r}$ and $\tcVaor$
(in the formal model as well as in the convolutive model, the difference is
immaterial here).

\begin{prop}	\label{propVam}
Let $\Om\subset\Z$ and $a = (\wh a_\eta)_{\eta\in\Om}$ as in Lemma~\ref{defnewcVomb}.
For each $m\in\Z^*$, denote by the same symbol~$\De_m$ the alien derivation of
index~$m$ on $\bA=\tRsimpZ$ and the mould derivation it induces
on~$\hM^\bul(\Om,\bA)$ by~\eqref{eqdefDd}.
Then there exists a scalar-valued alternal mould $V_a^\bul(m)
\in \hM^\bul(\Om,\C)$ such that
\begin{equation}	\label{eqDemtcVb}
\De_m \wt\cV_a^\bul = \wt\cV_a^\bul \times V_a^\bul(m), 
\qquad \De_m \tcVab = - V_a^\bul(m) \times \tcVab.
\end{equation}
Moreover, if $\omb\in\Om^\bul$ is non-empty,
\begin{equation}	\label{eqannulVm}
\norm{\omb}\neq m 
\ens\Rightarrow\ens
V_a^\omb(m) = 0.
\end{equation}
\end{prop}

\begin{proof}
Since $\wt\cV_a^\bul$ and $S \wt\cV_a^\bul = \tcVab$ are mutually inverse, we
get
$$
\De_m \wt\cV_a^\bul = \wt\cV_a^\bul \times \wt V_a^\bul(m), 
\qquad \De_m \tcVab = \tVa \times \tcVab
$$
by {\em defining} the moulds $\wt V_a^\bul(m)$ and $\tVa$ as
$$
\wt V_a^\bul(m) = \tcVab \times \De_m \wt\cV_a^\bul,
\qquad \tVa = \De_m \tcVab \times \wt\cV_a^\bul,
$$
but a priori all these moulds take their values in~$\bA$.
The operators~$\De_m$ and~$S$ clearly commute, thus $\tVa = S \wt V_a^\bul(m)$.
Proposition~\ref{propDerivSym} shows that $\wt V_a^\bul(m)$ and $\tVa$ are alternal;
Proposition~\ref{propinvsym} then shows that they are opposite of one another:
$\tVa = - \wt V_a^\bul(m)$.

It only remains to be checked that $\wt V_a^\bul(m)$ is scalar-valued and
satisfies~\eqref{eqannulVm}. 
This will follow from the equation
\begin{equation}	\label{eqpanamVa}
(\pa+\na-m) \wt V_a^\bul(m) = 0,
\end{equation}
where $\pa$ denotes the differential $\frac{\dd\,}{\dd z}$ as well as the mould
derivation it induces by~\eqref{eqdefDd} and $\na$ is the mould derivation~\eqref{eqdefna}.

Here is the proof of~\eqref{eqpanamVa}:
$\wt\cV_a^\bul$ is defined on non-empty words by the first equation
in~\eqref{eqdefnewcV}, which can be written
\begin{equation}	\label{eqdefnewcVmould}
(\pa+\na)\wt\cV_a^\bul = \wt J_a^\bul \times \wt\cV_a^\bul,
\end{equation}
with $\wt J_a^\bul \in \hM^\bul(\Om,\bA)$ defined exactly as in~\eqref{eqdefJa}.
Let us apply the derivation~$\De_m$ to both sides of equation~\eqref{eqdefnewcVmould}, using 
$\De_m(\pa+\na) = (\pa+\na-m)\De_m$ (consequence of~\eqref{eqcommutalien} and of
$[\na,\De_m]=0$) and $\De_m \wt J_a^\bul = 0$ (consequence of the vanishing
of~$\De_m$ on entire functions):
$$
(\pa+\na-m)\De_m\wt\cV_a^\bul = \wt J_a^\bul \times \De_m \wt\cV_a^\bul.
$$
Writing $\De_m \wt\cV_a^\bul$ as $\wt\cV_a^\bul \times \wt V_a^\bul(m)$ and using
the fact that $\pa+\na$ is a derivation, we get
$$
\big( (\pa+\na)\wt\cV_a^\bul \big) \times \wt V_a^\bul(m) 
+ \wt\cV_a^\bul \times (\pa+\na-m)\wt V_a^\bul(m)
= \wt J_a^\bul \times \wt\cV_a^\bul \times \wt V_a^\bul(m),
$$
whence $\wt\cV_a^\bul \times (\pa+\na-m)\wt V_a^\bul(m)=0$ by a further use
of~\eqref{eqdefnewcVmould}. Since $\wt\cV_a^\est=1$ and $\hM^\bul(\Om,\bA)$ is
an integral domain, this yields~\eqref{eqpanamVa}.

We conclude the proof by interpreting this relation in the convolutive model: we
already knew that $\wt V_a^\est(m) = 0$; now, for any non-empty~$\omb$, we have 
$$ 
\cB\wt V_a^\omb(m) = V_a^\omb(m)\de + \wh V_a^\omb(m)(\ze) 
$$
with $V_a^\omb(m) \in \C$ and $\wh V_a^\omb(m) \in \wh H\big(\gR(\Z)\big)$ satisfying
$$
(\norm{\omb}-m)V_a^\omb(m) = 0,
\qquad (-\ze+\norm{\omb}-m)\wh V_a^\omb(m) = 0,
$$
whence $V_a^\omb(m)=0$ for $\norm{\omb}\neq m$ and
$\wh V_a^\omb(m) = 0$ for all~$\omb$ (since both $\C$ and $\wh
H\big(\gR(\Z)\big)\subset\C\{\ze\}$ are integral domains).
\end{proof}

%%%%%%%%%%%%%%%%%%%%%%%%%%%%%%%%%%%%%%%%%%%%%%%%%%%%%%%%%%%%%%%%%%%%%%%%%
%%%%%%%%%%%%%%%%%%%%%%%%%%%%%%%%%%%%%%%%%%%%%%%%%%%%%%%%%%%%%%%%%%%%%%%%%

\parag
Formulas~\eqref{eqDemtcVb}, when evaluated in the convolutive model on $\omb\in\Om^\bul$, read 
$$
\De_m \wh\cV_a^\est = \De_m \hcVae = 0
$$ 
for $r(\omb)=0$, which is obvious since $\wh\cV_a^\est = \hcVae = \de$.
For $r(\omb)=1$, we get
\begin{equation*}	%\label{eqDemrun}
\De_m \wh\cV_a^{\om_1} = - \De_m \hcVaun = V_a^{\om_1}(m)\,\de
\end{equation*}
and the explicit value of the coefficient is
\begin{equation}	\label{eqVmrun}
\om_1=m \ens\Rightarrow\ens
 V_a^{\om_1}(m) = -2  \pi\I\, \wh a_m(m),
\qquad \om_1\neq m \ens\Rightarrow\ens
V_a^{\om_1}(m) = 0.
\end{equation}
This is a simple residuum computation for the meromorphic functions
$\wh\cV_a^{\om_1}(\ze) = -\hcVaun(\ze) = -\frac{\wh a_{\om_1}(\ze)}{\ze-\om_1}$
(observe that the value of~$\wh a_m$ at~$m$ and thus of $V_a^{m}(m)$ depend
transcendentally on the Taylor coefficients of~$\wh a_m$ at the origin).

For $r=r(\omb)\ge2$, we get
\begin{multline}
\De_m \wh\cV_a^\omb = V_a^\omb(m)\,\de  + \sum_{i=1}^{r-1}
V_a^{\om_{i+1},\dotsc,\om_r}(m) \wh\cV_a^{\om_1,\dotsc,\om_i}, \\
\label{eqVaoresidu}
\De_m \hcVao = - V_a^\omb(m)\,\de  - \sum_{i=1}^{r-1}
V_a^{\om_1,\dotsc,\om_i}(m) \hcVai.
\end{multline}
The number $V_a^\omb(m)$ thus appears as the {\em residuum} of a certain simple
singularity, which is a combination of the singularities at~$m$ of certain
branches of $\wh\cV_a^\omb$ or $\hcVao$; on the other hand, the fact that the
{\em variation} of this singularity can be expressed as a linear combination of the
functions $\wh\cV_a^{\om_1,\dotsc,\om_i}$ or $\hcVai$ is related to the very
origin of the name ``resurgent functions'':
the functions $\wh\cV_a^\omb(\ze)$ or $\hcVao(\ze)$, which were initially
defined for $\ze$ close to the origin by \eqref{eqnewcViter}--\eqref{eqnewtcViter},
``resurrect'' in the variation of the singularities of their analytic
continuations.

An even more striking instance of this ``resurgence phenomenon'' is the Bridge
Equation, to be discussed in the case of the saddle-node problem in
Section~\ref{secBE} below.

%%%%%%%%%%%%%%%%%%%%%%%%%%%%%%%%%%%%%%%%%%%%%%%%%%%%%%%%%%%%%%%%%%%%%%%%%
%%%%%%%%%%%%%%%%%%%%%%%%%%%%%%%%%%%%%%%%%%%%%%%%%%%%%%%%%%%%%%%%%%%%%%%%%

\parag
The computation of the number~$V_a^\omb(m)$ is not as easy when $r(\omb)\ge2$ as in the
case $r=1$.

First observe that the vanishing of~$V_a^\omb(m)$ when $\norm{\omb}
= \htb\om_1 = \wc\om_r \neq m$ could be obtained as a
consequence of the analytic continuation of formulas
\eqref{eqnewcViter}--\eqref{eqnewtcViter}
(for instance, the singularities of the analytic continuation of~$\hcVao$
can only be located at $\wc\om_1,\dots,\wc\om_r$ and, among them, only the one
at~$\wc\om_r$ can have a non-zero residuum---\cf the argument at the end of the
proof of Lemma~\ref{defnewcVomb}).

For $\norm{\omb}=m$, using the notations of Definition~\ref{defaliender}, one
can write $V_a^\omb(m)$ as a combination of iterated integrals:
\eqref{eqiteratedinthcVAo} and~\eqref{eqVaoresidu} yield
\begin{multline}	\label{eqformGaeps}
V_a^\omb(m) = -2\pi\I 
\sum_{\eps\in\{+,-\}^{|m|-1}} \frac{p_\eps!q_\eps!}{|m|!}
\int_{\Ga_\eps}
\wh a_{\om_r}(m-\ze_{r-1})  \cdot \\[1ex]
\cdot \frac{\wh a_{\om_{r-1}}(\ze_{r-1}-\ze_{r-2})}{\ze_{r-1}-\wc\om_{r-1}}
\dotsm\frac{\wh a_{\om_{2}}(\ze_{2}-\ze_{1})}{\ze_{2}-\wc\om_{2}}
\frac{\wh a_{\om_{1}}(\ze_{1})}{\ze_{1}-\wc\om_{1}}
\,\dd\ze_{1}\dotsm\dd\ze_{r-1},
\end{multline}
where $\Ga_\eps$ consists of all $(r-1)$-tuples
$(\ze_1,\dotsc,\ze_{r-1}) = \big( \ga_\eps(t_1),\dotsc,\ga_\eps(t_{r-1}) \big)$
with $t_1 < \dotsb < t_{r-1}$, for any parametrisation of the oriented
path~$\ga_\eps$ (which connects~$0$ and $m=\wc\om_r$).
In fact, one can restrict oneself to the paths which follow the segment
$\left]0,m\right[$ circumventing the points of
$\{\wc\om_1,\dotsc,\wc\om_{r-1}\}\cap\left]0,m\right[ =
\{k_1,\dots,k_s\}$ to the right or to the
left, labelled by sequences $\eps\in\{+,-\}^s$, with weights $p_\eps!q_\eps!/(s+1)!$.

The formula gets simpler when $\Om\subset\Z^*$ and $\wt a_\eta \equiv
z\ii$ for each $\eta\in\Om$, since each $\wh a_\eta$ is then the
constant function with value~$1$:
\begin{equation*}
\norm{\omb}=m \ens\Rightarrow\ens
V_a^\omb(m) = -2\pi\I 
\sum_{\eps\in\{+,-\}^{|m|-1}} \frac{p_\eps!q_\eps!}{|m|!}
\int_{\Ga_\eps}
\frac{\dd\ze_{1}\dotsm\dd\ze_{r-1}}{%
(\ze_{1}-\wc\om_{1})\dotsm(\ze_{r-1}-\wc\om_{r-1})}.
\end{equation*}
In this last case, the numbers~$V_a^\omb(m)$ are connected with multiple
logarithms. They are studied under the name ``canonical hyperlogarithmic
mould'' in \cite{Eca81}, chap.~7, without the restriction $\Om\subset\Z$ (which
we imposed here only to avoid having to define the larger algebra $\hRsimp$;
also the condition $0\notin\Om$ was imposed here only to simplify the discussion).

Observe that $V_a^\bul(m)$ is always a primitive element of the graded
cocommutative Hopf algebra $\gH^\bul(\Om,\C)$ defined in Section~\ref{secAltSym}
(this is just a rephrasing of the shuffle relations encoded by the alternality
of this scalar mould).

%%%%%%%%%%%%%%%%%%%%%%%%%%%%%%%%%%%%%%%%%%%%%%%%%%%%%%%%%%%%%%%%%%%%%%%%%
%%%%%%%%%%%%%%%%%%%%%%%%%%%%%%%%%%%%%%%%%%%%%%%%%%%%%%%%%%%%%%%%%%%%%%%%%

\parag
Formulas~\eqref{eqDemtcVb} can be iterated so as to express all the successive
alien derivatives of our resurgent functions~$\wt\cV_a^\omb$ or~$\tcVao$:
\begin{multline}	\label{eqiterAlDercVa}
\De_{m_s}\dotsm \De_{m_1} \wt\cV_a^\bul = \wt\cV_a^\bul \times 
V_a^\bul(m_s) \times \dotsm \times V_a^\bul(m_1), \\
\De_{m_s}\dotsm \De_{m_1} \tcVab = 
(-1)^s V_a^\bul(m_1) \times \dotsm \times V_a^\bul(m_s) \times \tcVab,
\end{multline}
for $s\ge1$ and $m_1,\dotsc,m_s\in\Z^*$.

We can consider the collection of resurgent functions
$(\wt\cV_a^\omb)_{\omb\in\Om^\bul}$ (or $(\tcVao)_{\omb\in\Om^\bul}$) as closed
under alien derivation (\ie\ all their alien derivatives can be expressed through
relations involving themselves and scalars);
it was already closed under multiplication (by symmetrality),
and even under ordinary differentiation, in view of~\eqref{eqdefnewcV}, if we
admit relations with coefficients in~$\C\{z\ii\}$ (but, after all, convergent
series can be considered as ``resurgent constants'': all alien derivations act
trivially on them).

This is why the $\wt\cV_a^\omb$'s are called ``resurgent monomials'': they
behave nicely under elementary operations such as multiplication and alien derivations. 
In fact, in Section~\ref{secResurMonom} below, we shall deduce from them another
family of resurgence monomials which behave even better under the action of
alien derivations (but the price to pay is that their ordinary derivatives are
not as simple as~\eqref{eqdefnewcV}).

Notice that the operator $\De_{m_s}\dotsm \De_{m_1}$ measures a combination of
singularities located at $m_1+\dotsb+m_s$.
For instance, the fact that $V_a^\bul(m_s) \times \dotsm \times V_a^\bul(m_1)$
vanishes on any word~$\omb$ such that $\norm{\omb}\neq m_1+\dotsb+m_s$ (easy
consequence of~\eqref{eqannulVm}) is consistent with the vanishing of the
residuum at any point $\neq\htb\om_1$ of any branch of~$\wh\cV_a^\omb$
(consequence of the analytic continuation of~\eqref{eqnewcViter}).

%%%%%%%%%%%%%%%%%%%%%%%%%%%%%%%%%%%%%%%%%%%%%%%%%%%%%%%%%%%%%%%%%%%%%%%%%
%%%%%%%%%%%%%%%%%%%%%%%%%%%%%%%%%%%%%%%%%%%%%%%%%%%%%%%%%%%%%%%%%%%%%%%%%

\parag
Let $\Om\subset\Z$ and $a = (\wh a_\eta)_{\eta\in\Om}$ be a family of entire
functions as in Lemma~\ref{defnewcVomb}, thus with $\wh a_0(0)=0$ if $0\in\Om$.
We end this section by illustrating mould calculus to derive quadratic shuffle
relations for the numbers
\begin{equation}	\label{eqdefLao}
L_a^\omb = 2\pi\I \int_{\Ga^+}
\wh a_{\om_r}(\norm{\omb}-\ze_{r-1})
\tfrac{\wh a_{\om_{r-1}}(\ze_{r-1}-\ze_{r-2})}{\ze_{r-1}-\wc\om_{r-1}}
\dotsm\tfrac{\wh a_{\om_{2}}(\ze_{2}-\ze_{1})}{\ze_{2}-\wc\om_{2}}
\tfrac{\wh a_{\om_{1}}(\ze_{1})}{\ze_{1}-\wc\om_{1}}
\,\dd\ze_{1}\dotsm\dd\ze_{r-1},
\end{equation}
for $\omb\in\Om^\bul$ non-empty,
where $\Ga^+ = \Ga_\eps$ with $\eps = (+,\dotsc,+)\in\{+,-\}^{|m|-1}$ for
$m=\norm{\omb}$ (notation of~\eqref{eqformGaeps}; 
if $r=1$, then $L_a^{\om_1} = 2\pi\I \wh a_{\om_1}(\om_1)$).
This includes the case of the multiple logarithms
$$
L^\omb = 2\pi\I 
\int_{\Ga^+}
\frac{\dd\ze_{1}\dotsm\dd\ze_{r-1}}{%
(\ze_{1}-\wc\om_{1})\dotsm(\ze_{r-1}-\wc\om_{r-1})},
$$
with $\om_1,\dotsc,\om_r\in\Om\subset\Z^*$
(obtained when $\wh a_\eta(\ze)\equiv 1$).\footnote{%
We recall that $\wc\om_1=\om_1, \wc\om_2=\om_1+\om_2, \dotsc,
\wc\om_{r-1} = \om_1+\dotsb+\om_{r-1}$
(thus $L^\omb$ depends on~$\om_r$ only through~$\Ga^+$
which connects the origin and~$\wc\om_r$).
}

It is convenient to use here the auxiliary operators~$\De^+_m$ of
$\hRsimpZ$ defined by the formulas $\De^+_0=\ID$ and, for $m\in\Z^*$,
\begin{equation}	\label{eqdefDemplus}
\De^+_m(c\,\de + \wh\ph) = \Sing_m(\cont_{\ga^+}\wh\ph),
\end{equation}
where $\ga^+ = \ga_\eps$ with $\eps = (+,\dotsc,+)\in\{+,-\}^{|m|-1}$.
Thus 
\begin{equation}	\label{eqlienLaDep}
L_a^\omb = \text{coefficient of~$\de$ in $\De^+_{\norm{\omb}}\hcVao$.}
\end{equation}
We shall consider $L_a^\omb$ as the value at~$\omb$ of a scalar
mould~$L_a^\bul$; we set $L_a^\est=1$, so that~\eqref{eqlienLaDep} still holds
when $\omb=\est$.

\begin{prop}	\label{propLabsym}
The numbers $L_a^\omb$ satisfy the shuffle relations
$$
\sum_{\omb\in\Om^\bul}  \sh{\omb^1}{\omb^2}{\omb} L_a^\omb \;=\; \left|
\begin{aligned}
& L_a^{\omb^1} L_a^{\omb^2} \quad \text{if $\norm{\omb^1}\cdot\norm{\omb^2}\ge0$}\\[1ex]
& \quad 0 \qquad\quad \text{if not}
\end{aligned} \right.
$$
for any non-empty $\omb^1,\omb^2\in\Om^\bul$.
Equivalently, the scalar moulds $\pmL{\bul}$ defined by
\begin{equation}	\label{eqdefpmL}
\pL{\omb} = 1_{\{\norm{\omb}\ge0\}} L_a^\omb, \qquad
\mL{\omb} = 1_{\{\norm{\omb}\le0\}} L_a^\omb
\end{equation}
(for any $\omb\in\Om^\bul$, with the convention $\norm{\est}=0$)
are symmetral.
\end{prop}

This can be rephrased by saying that $\pL\bul$ and $\mL\bul$ are
group-like elements of the graded cocommutative Hopf algebra $\gH^\bul(\Om,\C)$
defined in Section~\ref{secAltSym}.

The rest of this section is devoted to the proof of Proposition~\ref{propLabsym}.
We begin by a few facts about the operators~$\De^+_m$; these are not derivations, as the alien derivations $\De_m$,
but they are related to them and satisfy modified Leibniz rules analogous
to~\eqref{eqmodifLeib}:
%
%%%%%%%%%%%%%%%%%%%%%%%%%%%%%%%%%%%%%%%%
%
\begin{lemma}	\label{lemrelDemDepm}
The operators~$\De^+_m$ defined in~\eqref{eqdefDemplus} are related to the alien
derivations~\eqref{eqdefDeomgen} by the following relations:
for any $m\in\Z^*$,
\begin{equation}	\label{eqrelDemDepm}
\De^+_m = \sum \tfrac{1}{s!} \De_{m_s}\dotsm\De_{m_1}, \qquad
\De_m = \sum \tfrac{(-1)^{s-1}}{s} \De^+_{m_s}\dotsm\De^+_{m_1},
\end{equation}
with both sums taken over all $s\ge1$ and $m_1,\dotsc,m_s\in\Z^*$ of the same sign
as~$m$ such that $m_1+\dotsb+m_s=m$ (these are thus finite sums).
Moreover, for any $\wh\chi_1,\wh\chi_2\in\hRsimpZ$ and $m\in\Z$,
\begin{equation}	\label{eqLeibDep}
\De^+_m(\wh\chi_1*\wh\chi_2) = \sum \De^+_{m_1}\wh\chi_1 * \De^+_{m_2}\wh\chi_2
\end{equation}
with summation over all $m_1,m_2\in\Z$ of the same sign as~$m$ (but possibly
vanishing) such that $m_1+m_2=m$.
\end{lemma}

%%%%%%%%%%%%%%%%%%%%%%%%%%%%%%%%%%%%%%%%

Let us denote by the same symbols the operators of~$\tRsimpZ$ obtained
from the~$\De^+_m$'s by conjugacy by the formal Borel transform~$\cB$, as we did for
the~$\De_m$'s.
If we consider the algebras $\tRsimpZ[[\ee^{-z}]]$ and $\tRsimpZ[[\ee^{z}]]$,
formula~\eqref{eqrelDemDepm} can be written
\begin{equation}	\label{eqDepexpDe}
\sum_{m\ge0} \ee^{-mz} \De^+_m = \exp\bigg( \sum_{m>0} \ee^{-mz} \De_m \bigg),
\quad
\sum_{m\le0} \ee^{-mz} \De^+_m = \exp\bigg( \sum_{m<0} \ee^{-mz} \De_m \bigg).
\end{equation}
We do not give the proof of this lemma here; see \eg\ \cite{kokyu}, Lemmas~4 and~5 
(the coefficients $p_\eps!q_\eps!/|m|!$ in Definition~\ref{defaliender} were
chosen exactly so that~\eqref{eqrelDemDepm} hold; the standard properties of the
logarithm and exponential series then show that~\eqref{eqLeibDep} and
Proposition~\ref{propAlDer} are equivalent;
it is in fact easy to check first~\eqref{eqLeibDep} by deforming the contour of integration in
the integral giving~$\wh\chi_1*\wh\chi_2$, and then to deduce Proposition~\ref{propAlDer}).

\begin{lemma}	\label{lemW}
For any $m\in\Z^*$, define a scalar mould $L_a^\bul(m)$ by the formula
$$
L_a^\bul(m) = \sum \tfrac{(-1)^{s}}{s!} V_a^\bul(m_1)\times\dotsb\times V_a^\bul(m_s),
$$
with summation over all $s\ge1$ and $m_1,\dotsc,m_s\in\Z^*$ of the same sign
as~$m$ such that $m_1+\dotsb+m_s=m$.
Define also $L_a^\bul(0)=1^\bul$. Then, for every $m\in\Z$,
\begin{enumerate}[(i)]
\item 	\label{itemfirstpty}
$\De^+_m  \hcVab = L_a^\bul(m) \times \hcVab$,
\item  	\label{itemsecpty}
$\tau\big( L_a^\bul(m) \big) = \sum L_a^\bul(m_1)\otimes L_a^\bul(m_2)$, with
summation over all $m_1,m_2\in\Z$ of the same sign as~$m$ such that $m_1+m_2=m$,
\item  	\label{itemthirdpty}
$m=\norm{\omb} \;\Rightarrow\; L_a^\omb(m) = L_a^\omb$,  \quad
$m\neq\norm{\omb} \;\Rightarrow\; L_a^\omb(m) = 0$
\ens (for any $\omb\in\Om^\bul$, with the convention $\norm{\est}=0$).
\end{enumerate}
\end{lemma}

\begin{proof}
The first property follows from~\eqref{eqiterAlDercVa} and~\eqref{eqrelDemDepm}.
For the second, we write the symmetrality of~$\hcVab$ and~$\wh\cV_a^\bul$ as
identities in $\hM^\bbul(\Om,\hRsimpZ)$:
$$
\tau(\hcVab) = \hcVab \otimes \hcVab, \qquad
\tau(\wh\cV_a^\bul) = \wh\cV_a^\bul \otimes \wh\cV_a^\bul,
$$
the operator~$\De^+_m$ induces operators acting on moulds and dimoulds which clearly satisfy
$\De^+_m\circ\tau(\hcVab) = \tau(\De^+_m\hcVab)$ and relation~\eqref{eqLeibDep}
implies
$$
\tau(\De^+_m\hcVab) = \sum_{\substack{m=m_1+m_2 \\ m_i m\ge0}}
\De^+_{m_1}\hcVab \otimes \De^+_{m_2}\hcVab,
$$
whence the result follows since
$\tau\big(L_a^\bul(m)\big) = \tau(\De^+_m\hcVab) \times \tau(\wh\cV_a^\bul)$
by the homomorphism property of~$\tau$ applied to~(\ref{itemfirstpty}).

The second part of the third property is obvious when $m=0$ and follows
from~\eqref{eqannulVm} when $m\neq0$,
because $\norm{\omb}\neq m_1+\dotsb+m_s$ implies that
$V_a^\bul(m_1)\times\dotsb\times V_a^\bul(m_s)$ vanishes on~$\omb$ 
(even if $\omb=\est$).
The first part of the third property follows from~\eqref{eqlienLaDep}, since
property~(\ref{itemfirstpty}) yields
$$
\De^+_m\hcVao = L_a^\omb(m) \,\de +
\sum_{i=1}^{r-1}
L_a^{\om_1,\dotsc,\om_i}(m) \hcVai
$$
if $r=r(\omb)\ge1$ and $\De^+_m\hcVae = L_a^\est(m)\de$ if $\omb=\est$.
\end{proof}

\medskip

\noindent
\emph{Proof of Proposition~\ref{propLabsym}.}
We have $L_a^\est=1$.
Let $\omb^1,\omb^2\in\Om^\bul$ be non-empty. Property~(\ref{itemsecpty}) with
$m=\norm{\omb^1}+\norm{\omb^2}$ yields
$$
\sum_{\omb\in\Om^\bul} \sh{\omb^1}{\omb^2}{\omb} L_a^\omb(m) =
\sum_{\substack{m=m_1+m_2 \\ m_i m\ge0}} L_a^{\omb^1}(m_1) L_a^{\omb^2}(m_2).
$$
According to Property~(\ref{itemthirdpty}), the \lhs\ is
$\tau(L_a^\bul)^{\omb^1,\omb^2}$ (because any nonzero term in it has $\norm{\omb}=m$).
Among the $|m|+1$ terms of the \rhs, at most one may be nonzero:
if $\norm{\omb^1}$ and $\norm{\omb^2}$ have the same sign, then the term
corresponding to $m_1=\norm{\omb^1}$ is $L_a^{\omb^1} L_a^{\omb^2}$ while all
the others vanish; but in the opposite case, this term does not belong to the
summation and one gets~$0$ as \rhs.
This is the desired shuffle relation; we leave it to the reader to interpret it
in terms of symmetrality for the moulds~$\pmL{\bul}$ by distinguishing the four
possible cases: $\norm{\omb^1}\cdot\norm{\omb^2} \ge0$ or~$<0$, and
$\norm{\omb^1}+\norm{\omb^2} \ge0$ or~$<0$.
\qed

\medskip

In fact, we can write
\begin{gather}	\label{eqrelLaV}
\pL{\bul} = \sum_{m\ge0} L_a^\bul(m) = \exp\big(-\pV{\bul}\big),
\qquad \pV{\bul} = \sum_{m>0} V_a^\bul(m) \\
\mL{\bul} = \sum_{m\le0} L_a^\bul(m) = \exp\big( -\mV{\bul} \big),
\qquad \mV{\bul} = \sum_{m<0} V_a^\bul(m),
\end{gather}
with well-defined alternal moulds~$\pmV{\bul}$
(and using $\exp$ as a short-hand for $E_1$---see~\eqref{eqdefEt}),
since Lemma~\ref{lemW}~(\ref{itemthirdpty}) and property~\eqref{eqannulVm}
imply that, when evaluated on a given word~$\omb$, these formulas involve only
finitely many terms; one could thus have invoked
Proposition~\ref{propstructaltsym} to deduce the symmetrality of~$\pmL\bul$.

%%%%%%%%%%%%%%%%%%%%%%%%%%%%%%%%%%%%%%%%%%%%%%%%%%%%%%%%%%%%%%%%%%%%%%%%%
%%%%%%%%%%%%%%%%%%%%%%%%%%%%%%%%%%%%%%%%%%%%%%%%%%%%%%%%%%%%%%%%%%%%%%%%%

\section{The Bridge Equation for the saddle-node}	\label{secBE} 

%%%%%%%%%%%%%%%%%%%%%%%%%%%%%%%%%%%%%%%%%%%%%%%%%%%%%%%%%%%%%%%%%%%%%%%%%
%%%%%%%%%%%%%%%%%%%%%%%%%%%%%%%%%%%%%%%%%%%%%%%%%%%%%%%%%%%%%%%%%%%%%%%%%

In this section, returning to the saddle-node problem, we shall explain why the
formal series $\wt\ph_n(z) = \ph_n(-1/z)$ and $\wt\psi_n(z) = \psi_n(-1/z)$ of
Theorem~\ref{thmResur}, which were proved to belong to~$\tRZ$, are in fact
simple resurgent functions.
Moreover, we shall express their alien derivatives in terms of
themselves and of the numbers $V_a^\bul(m)$ of Proposition~\ref{propVam}.

%%%%%%%%%%%%%%%%%%%%%%%%%%%%%%%%%%%%%%%%%%%%%%%%%%%%%%%%%%%%%%%%%%%%%%%%%
%%%%%%%%%%%%%%%%%%%%%%%%%%%%%%%%%%%%%%%%%%%%%%%%%%%%%%%%%%%%%%%%%%%%%%%%%

\parag
We recall the hypotheses and the notations for the saddle-node:
$$ X = x^2 \frac{\pa\,}{\pa x} + A(x,y) \frac{\pa\,}{\pa y} $$
with $A(x,y) = y + \sum_{\eta\in\Om} a_\eta(x)y^{\eta+1} \in \C\{x,y\}$,
where $\Om = \ao \eta\in\Z \mid \eta\ge-1\af$,
$\wt a_\eta(z) = a_\eta(-1/z) \in z\ii\C\{z\ii\}$ and
$\wt a_0(z) \in z^{-2}\C\{z\ii\}$.

We also recall that $\eta\in\Om \mapsto B_\eta = y^{\eta+1} \frac{\pa\,}{\pa y}$
gives rise to a comould~$\bB_\bul$ such that
$\bB_\omb y = \be_\omb y^{\norm{\omb}+1}$, where the numbers $\be_\omb$,
$\omb\in\Om^\bul$, satisfy Lemma~\ref{lemexo} (we define $\be_\est=1$ and
$\norm{\est}=0$). 
We set $a = (\cB\wt a_\eta)_{\eta\in\Om}$, so as to be able to make use of the
constants $V_a^\omb(m)$, $(m,\omb)\in\Z^*\times\Om^\bul$ defined in
Proposition~\ref{propVam} and more explicitly by formulas~\eqref{eqVmrun}
and~\eqref{eqformGaeps}.
Later in this section we shall prove
%
%%%%%%%%%%%%%%%%%%%%%%%%%%%%%%%%%%%%%%
%
\begin{prop}	\label{propCm}
The family of complex numbers 
$\big( \be_\omb V_a^\omb(m) \big)_{\omb\in\Om^\bul,\,\norm{\omb}=m}$
is summable for each $m\in\Z^*$. Let
\begin{equation}	\label{eqdefCm}
C_m =  \sum_{\omb\in\Om^\bul,\,\norm{\omb}=m} \be_\omb V_a^\omb(m),
\qquad m\in\Z^*.
\end{equation}
Then $C_m=0$ for $m\le-2$.
\end{prop}

%%%%%%%%%%%%%%%%%%%%%%%%%%%%%%%%%%%%%%

We call \emph{\'Ecalle's invariants} of~$X$  the complex numbers
$C_{-1},C_1,C_2,\dotsc,C_m,\dotsc$ because of their role in the Bridge Equation
(Theorem~\ref{thmBE} below) and in the classification problem (Theorem~\ref{thmCmInv} and
Section~\ref{secMR} below).

The formal transformations $\th(x,y) = \big(x,\ph(x,y)\big)$ and
$\th\ii(x,y) = \big(x,\psi(x,y)\big)$ which conjugate~$X$ to its normal form
$X_0 = x^2 \frac{\pa\,}{\pa x} + y \frac{\pa\,}{\pa y}$ were constructed in the
first part of this article through mould-comould expansions for the
corresponding substitution operators~$\Th$ and~$\Th\ii$.
Passing to the resurgence variable $z=-1/x$, we set
$$
\wt\ph(z,y) = \ph(-1/z,y) =  y + \sum_{n\ge0} \wt\ph_n(z)y^n, \quad
\wt\psi(z,y) = \psi(-1/z,y) = y + \sum_{n\ge0} \wt\psi_n(z)y^n,
$$
where the coefficients $\wt\ph_n(z)$ and~$\wt\psi_n(z)$ are known to belong to
the algebra~$\tRZ$ of resurgent functions, by Theorem~\ref{thmResur}.
We also introduce the substitution operator
\begin{equation}	\label{defThti}
\wt\Th \colon
\ti f(z,y) \mapsto \ti f\big(z,\wt\ph(z,y)\big)
\end{equation}	
(a priori defined in $\C[[z\ii,y]]$).
Later in this section, we shall prove
\begin{thm}	\label{thmBE} 
The formal series $\wt\ph_n(z)$ and~$\wt\psi_n(z)$ are simple resurgent
functions, thus
$\wt\ph(z,y)$ and~$\wt\psi(z,y)$ belong in fact to $\tRsimpZ[[y]]$.

Moreover, for any $m\in\Z^*$, the formal series of~$\tRsimpZ[[y]]$
$$
\De_m \wt\ph := \sum_{n\ge0} (\De_m\wt\ph_n) y^n, \qquad
\De_m \wt\psi := \sum_{n\ge0} (\De_m\wt\psi_n) y^n
$$
are given by the formulas
\begin{equation}	\label{eqBEtigA}
\De_m \wt\ph = C_m y^{m+1} \frac{\pa\wt\ph}{\pa y},
\qquad
\De_m \wt\psi = - C_m \wt\psi^{m+1}, 
\qquad m\in\Z^*.
\end{equation}
\end{thm}

%%%%%%%%%%%%%%%%%%%%%%%%%%%%%%%%%%%%%%%%%%%%%%%%%%%%%%%%%%%%%%%%%%%%%%%%%
%%%%%%%%%%%%%%%%%%%%%%%%%%%%%%%%%%%%%%%%%%%%%%%%%%%%%%%%%%%%%%%%%%%%%%%%%

\parag
The two equations in~\eqref{eqBEtigA} are equivalent forms of the so-called {\em
Bridge Equation}, here expressed in~$\bA[[y]]$ with $\bA = \tRsimpZ$.
On the one hand, the \lhs s represent the action of the alien derivation~$\De_m$
of~$\bA[[y]]$ (we denote by the same symbol the alien derivation~$\De_m$
of~$\bA$ and the operator it induces in $\bA[[y]]$ by acting separately on each
coefficient).
On the other hand, both \rhs s can be expressed with the help of the ordinary
differential operator
$$
\gC(m) = C_m y^{m+1}\frac{\pa\,}{\pa y},
$$
yielding
\begin{align}
\De_m \wt\ph &= \gC(m)\wt\ph = \gC(m)\wt\Th y, \\[1ex]
\label{eqDempsiThiigCm}
\De_m \wt\psi &= -\wt\Th\ii\gC(m)y.
\end{align}
See the end of this section for more symmetric formulations of the Bridge
Equation, which involve only the operators $\wt\Th$ or~$\wt\Th\ii$ and~$\De_m$ for
the \lhs s, and $\gC(m)$ for the \rhs s.

The name ``Bridge Equation'' refers to the link thus established between alien and ordinary
differential calculus when dealing with the solutions~$\wt\ph$ and~$\wt\psi$ of
our formal normalisation problem (or with the operator~$\Th$ solution of the
conjugacy equation~\eqref{eqconjugOp}).

This is a very general phenomenon, in which one sees the advantage of measuring
the singularities in the Borel plane though {\em derivations}: 
we are dealing with the solutions of non-linear equations (\eg\ $(\pa + y
\frac{\pa\,}{\pa y})\wt\ph(z,y) = A(-1/z,\wt\ph(z,y))$ in
$\C[[z\ii,y]]$), 
and their alien derivatives must satisfy equations corresponding to the
linearisation of these equations; its is thus natural that these alien
derivatives can be expressed in terms of the ordinary derivatives of the
solutions.

The above argument could be used to derive the form of
equation~\eqref{eqBEtigA}\footnote{
Compare the linear equations
$L \pa_y\wt\ph = 0$ and
$(L-m-1) \De_m\wt\ph = 0$ 
where 
$$
L = \wt X_0 + \ti\la(z,y), \qquad
\ti\la(z,y) = 1 - \pa_y A(-1/z,\wt\ph(z,y)), \quad
\wt X_0 = \pa + y\frac{\pa\,}{\pa y}.
$$
The second equation follows from~\eqref{eqcommutalien} for the computation of
$\De_m(\pa + y\frac{\pa\,}{\pa y})\wt\ph(z,y)$, 
and from the relation
$\De_m A(-1/z,\wt\ph(z,y)) = \big(\pa_y A(-1/z,\wt\ph(z,y))\big)\De_m\wt\ph(z,y)$
deduced from Proposition~\ref{propResurzCVy} below
(indeed, $A(-1/z,y)\in\C\{z\ii,y\} \subset \bA\{y\}$).
Since $\pa_y\wt\ph = 1 + \gO(z\ii,y)$ is invertible, we can set
$\wt\chi = (\pa_y\wt\ph)\ii \De_m\wt\ph$;
the above linear equations imply that
$\wt\chi$ is annihilated by $\wt X_0-(m+1)$, thus proportional
to~$y^{m+1}$: there exists $c_m\in\C$ such that $\De_m\wt\ph = c_m y^{m+1}\pa_y\wt\ph(z,y)$.

The relation $\De_m\wt\psi = -c_m \wt\psi^{m+1}$ follows by the alien chain rule:
$y = \wt\ph\big(z,\wt\psi(z,y)\big) = \wt\Th\ii\wt\ph$
implies $(\De_m\wt\ph)(z,\wt\psi) +
\pa_y\wt\ph(z,\wt\psi)\De_m\wt\psi = 0$
by Proposition~\ref{propResurzCVy} below (using $\wt\ph\in\bA\{y\}$).
},
however, in the proof below, we prefer to use the
explicit mould representations involving~$\wt\cV^\bul$ and~$\tcVtb$ so as to
obtain formulas~\eqref{eqdefCm} for the coefficients~$C_m$.

%%%%%%%%%%%%%%%%%%%%%%%%%%%%%%%%%%%%%%%%%%%%%%%%%%%%%%%%%%%%%%%%%%%%%%%%%
%%%%%%%%%%%%%%%%%%%%%%%%%%%%%%%%%%%%%%%%%%%%%%%%%%%%%%%%%%%%%%%%%%%%%%%%%

\parag
Theorem~\ref{thmBE} could also have been formulated in terms of the formal
integral defined by~\eqref{eqdefYzu}:
$\wt Y(z,u) = \wt\ph(z,u\,\ee^z) \in \tRsimpZ[[u\,\ee^z]]$
and
$$
\dDem \wt Y = C_m u^{m+1} \frac{\pa\wt Y}{\pa u},
\qquad m\in\Z^*,
$$
where $\dDem = \ee^{-mz} \De_m$ is the {\em dotted alien derivation of index~$m$}, which already
appeared in formula~\eqref{eqDepexpDe}.

%%%%%%%%%%%%%%%%%%%%%%%%%%%%%%%%%%%%%%%%%%%%%%%%%%%%%%%%%%%%%%%%%%%%%%%%%
%%%%%%%%%%%%%%%%%%%%%%%%%%%%%%%%%%%%%%%%%%%%%%%%%%%%%%%%%%%%%%%%%%%%%%%%%

\parag
The Bridge Equations~\eqref{eqBEtigA} are a compact writing of infinitely many
``resurgence equations'' for the series $\De_m\wt\ph_n$ or~$\De_m\wt\psi_n$,
obtained by expanding them in powers of~$y$.

For instance, setting
\begin{equation}	\label{eqdefPhiph}
\wt\Phi_n = \left|
\begin{aligned}
1 + \wt\ph_1 & \qquad \text{if $n=1$}  \\
\wt\ph_n\ens\;     & \qquad \text{if $n\neq1$,}
\end{aligned}  \right.
\end{equation}
so that $\wt\ph(z,y) = \sum_{n\ge0} \wt\Phi_n(z) y^n$, we get
$$
\De_m\wt\Phi_n = \left|
\begin{aligned}
(n-m)C_m\wt\Phi_{n-m} & \qquad \text{if $-1\le m \le n-1$}  \\
0\qquad\qquad     & \qquad \text{if $m\le-2$ or $m\ge n$.}
\end{aligned}  \right.
$$
Thus \begin{itemize}
\item $\De_m\wt\ph_0=0$ for $m\neq-1$, 
while $\De_{-1}\wt\ph_0 = C_{-1}(1+\wt\ph_1)$;
\item $\De_m\wt\ph_1=0$ for $m\neq-1$, while $\De_{-1}\wt\ph_1 = 2 C_{-1} \wt\ph_2$;
\item $\De_m\wt\ph_2=0$ for $m\notin\{-1,1\}$, while $\De_{-1}\wt\ph_2 = 3 C_{-1} \wt\ph_3$
and $\De_1\wt\ph_2 = C_1 (1+\wt\ph_1)$;
\item $\De_m\wt\ph_3=0$ for $m\notin\{-1,1,2\}$, while\ldots
\end{itemize}
$\quad\qquad\vdots$
\medskip

\noindent 
Similarly, with
$$
\wt\Psi_n = \left|
\begin{aligned}
1 + \wt\psi_1 & \qquad \text{if $n=1$}  \\
\wt\psi_n\ens\;     & \qquad \text{if $n\neq1$,}
\end{aligned}  \right.
$$
we have
$ \sum (\De_m\wt\Psi_n) y^n = - C_m \big( \sum \wt\Psi_n y^n \big)^{m+1} $,
which means that $\De_m\wt\Psi_n = 0$ for all $n\in\N$ when $m\le -2$,
$$
\De_{-1}\wt\Psi_n = \left|
\begin{aligned}
-C_{-1} & \qquad \text{if $n=0$}  \\
0 \quad    & \qquad \text{if $n\neq0$}
\end{aligned}  \right.
$$
and
$\dst \De_m\wt\Psi_n = - C_m \sum_{n_1+\dotsc+n_{m+1}=n}
\wt\Psi_{n_1}  \dotsm \wt\Psi_{n_{m+1}}$ for any $n\in\N$ and $m\ge1$.

In particular, $C_m$ is the constant term in $\De_m\wt\ph_{m+1}$ or in
$-\De_m\wt\psi_{m+1}$.

%%%%%%%%%%%%%%%%%%%%%%%%%%%%%%%%%%%%%%%%%%%%%%%%%%%%%%%%%%%%%%%%%%%%%%%%%
%%%%%%%%%%%%%%%%%%%%%%%%%%%%%%%%%%%%%%%%%%%%%%%%%%%%%%%%%%%%%%%%%%%%%%%%%

\parag
\emph{Proof of Proposition~\ref{propCm} and Theorem~\ref{thmBE}.}
We have a Fr\'echet space structure on $\wHR\Z$, with
seminorms $\norm{\,.\,}_K$ indexed by the compact subsets of~$\gR(\Z)$:
$$
\norm{\wh\ph}_K = \max_{\ze\in K} \left|\wh\ph(\ze)\right|,
\qquad \wh\ph \in \wHR\Z, \quad K\in\gK.
$$
We thus naturally get Fr\'echet space structures on $\hRZ = \C\,\de\oplus\wHR\Z$,
by defining 
$\norm{c\,\de + \wh\ph}_K := \max\big( |c|, \norm{\wh\ph}_K \big)$,
and on $\tRZ = \cB\ii \hRZ$, with 
$\norm{\wt\chi}_K := \norm{\cB\wt\chi}_K$ for $\wt\chi = c+\wt\ph \in \tRZ$.

The space $\bA=\tRsimpZ$ of simple resurgent functions is a closed subspace
of~$\tRZ$ and the $\De_m$ are continuous operators.
Indeed, the map $\wh\ph \mapsto \Sing_m(\cont_\ga\wh\ph)$ is continuous on
$\wh\bA=\hRsimpZ$ because the variation can be expressed as a difference of branches
and the residuum as a Cauchy integral.

Consider now the formal series
$\wt\cV^\omb(z) = \wt\cV_a^\omb(z), \tcVto(z) = \tcVao(z) \in \bA$, 
and their formal Borel transforms, which belong to $\wh\bA$.
The end of the proof of Theorem~\ref{thmResur} shows that
$(\be_\omb \wh\cV^\omb)_{\omb\in\Om^\bul,\,\norm{\omb}=n-1}$ and 
$(\be_\omb \hcVto)_{\omb\in\Om^\bul,\,\norm{\omb}=n-1}$
are summable families of~$\wh\bA$ for each $n\in\N$; 
indeed, for any compact subset~$K$ of~$\gR(\Z)$, there exist $\rho$, $N$ and~$L$
such that any point of~$K$ is the endpoint of a $(\rho,N,n-\N^*)$-adapted path
of length~$\le L$ and also the endpoint of a $(\rho,N,\N)$-adapted path
of length~$\le L$, and one can use \eqref{ineqbecV}, \eqref{ineqCVUdeux}
and~\eqref{ineqCVUtrois}.
Hence the sums $\wh\ph_n$ and~$\wh\psi_n$ of these families belong to~$\wh\bA$.
Equivalently, the formal series $\wt\ph_n$ and~$\wt\psi_n$ appear as sums of
summable families of~$\bA$:
$$
\wt\ph_n = \sum_{\omb\in\Om^\bul,\,\norm{\omb}=n-1}
\be_\omb \wt\cV^\omb
\quad \text{and} \quad
\wt\psi_n = \sum_{\omb\in\Om^\bul,\,\norm{\omb}=n-1}
\be_\omb \tcVto
\quad \text{in $\bA$,}
$$
they are thus simple resurgent functions themselves.
To end the proof of Theorem~\ref{thmBE}, we thus only have to study the alien
derivatives $\De_m\wt\ph_n$ and $\De_m\wt\psi_n$.

%%%%%%%%%%%%%%%%%%%%%%%%%%%%%%%%%%%%%%%%%%%%%%%%%%%%%%%%%%%%%%%%%%%%%%%%%
%%%%%%%%%%%%%%%%%%%%%%%%%%%%%%%%%%%%%%%%%%%%%%%%%%%%%%%%%%%%%%%%%%%%%%%%%

\parag
\emph{End of the proof of Proposition~\ref{propCm}:}
Let $m\in\Z^*$.
In view of Lemma~\ref{lemexo}, we can suppose $m\ge-1$.
By continuity of~$\De_m$, 
$(\be_\omb \De_m\tcVto)_{\omb\in\Om^\bul,\,\norm\omb=m}$ is a summable family
of~$\bA$, of sum~$\De_m\wt\psi_{m+1}$.
In particular, the family obtained by extracting the constant terms is summable,
but the constant term in~$\De_m\tcVto$ is $-V_a^\omb(m)$
by~\eqref{eqDemtcVb}.
Hence we get the summability of
$$
C_m =  \sum_{\norm{\omb}=m} \be_\omb V_a^\omb(m) 
\quad \text{in $\C$,}
$$
which is the constant term in $-\De_m\wt\psi_{m+1}$.
\qed

%%%%%%%%%%%%%%%%%%%%%%%%%%%%%%%%%%%%%%%%%%%%%%%%%%%%%%%%%%%%%%%%%%%%%%%%%
%%%%%%%%%%%%%%%%%%%%%%%%%%%%%%%%%%%%%%%%%%%%%%%%%%%%%%%%%%%%%%%%%%%%%%%%%

\parag
As vector spaces, $\C[[y]]$ and $\bA[[y]]$ can be identified with~$\C^\N$
and~$\bA^\N$ and are thus also Fr\'echet spaces if we put the product topology on them.

As an intermediary step in the proof of Theorem~\ref{thmBE}, let us show
\begin{lemma}	\label{lemintermed} 
Let $m\in\Z^*$ and 
$$
\gC(m) = C_m y^{m+1}\frac{\pa\,}{\pa y}.
$$
Then, for each $n_0\in\N$, the families
$(\tcVto \bB_\omb y^{n_0})_{\omb\in\Om^\bul}$
and $(V_a^\omb(m) \bB_\omb y^{n_0})_{\omb\in\Om^\bul}$
are summable in~$\bA[[y]]$, of sums $\wt\Th\ii y^{n_0}$ and $\gC(m) y^{n_0}$.
\end{lemma}

\begin{proof}
Our aim is to show that $(\tcVto \bB_\omb)_{\omb\in\Om^\bul}$
and $(V_a^\omb(m) \bB_\omb)_{\omb\in\Om^\bul}$
are pointwise summable families of operators of~$\bA[[y]]$; in view of the above,
since $\bB_\omb y = \be_\omb y^{\norm{\omb}+1}$, we can already evaluate these operators
on~$y$ and write
\begin{equation}	\label{eqdeterminC}
\sum_{\omb\in\Om^\bul} V_a^\omb(m) \bB_\omb y =
\sum_{\omb\in\Om^\bul,\,\norm{\omb}=m} V_a^\omb(m) \bB_\omb y
= C_m y^{m+1}
\quad \text{in $\C[[y]]$}
\end{equation}
(the first identity stems from~\eqref{eqannulVm}) and
$$
\sum_{\omb\in\Om^\bul} \tcVto \bB_\omb y = 
y + \sum_{n\ge0} \wt\psi_n(z) y^n = \wt\Th\ii y
\quad \text{in $\bA[[y]]$.}
$$
Although similar to formula~\eqref{eqpsiThii}, the last equation is stronger in
that it gives the sum of a summable family of~$\bA[[y]]$ rather
than of a formally summable family of~$\C[[z\ii,y]]$.

When evaluating the operators~$\bB_\omb$ on~$y^{n_0}$, we get
coefficients~$\be_{\omb,n_0}$ which generalise the~$\be_\omb$'s:
$$
\bB_\omb y^{n_0} = \be_{\omb,n_0} y^{n_0+\norm{\omb}}
$$
with $\be_{\est,n_0} = 1$,
$\be_{(\om_1),n_0} = n_0$,
$\be_{\omb,n_0} = n_0 (n_0 + \wc\om_1) (n_0 + \wc\om_2) \dotsm (n_0 + \wc\om_{r-1})$ 
for $r\ge2$.
Notice that $\be_{\omb,n_0}\neq0 \;\Rightarrow\; \norm{\omb}\ge -n_0$.

A suitable modification of the proof of Theorem~\ref{thmResur} shows that the
families $(\be_{\omb,n_0} \hcVto)_{\omb\in\Om^\bul,\,\norm{\omb}=m}$
are summable in~$\wh\bA$ for all $m\ge-n_0$ 
(replace the functions $\tS_m(\ze) = \tfrac{m+1}{\ze-m}$ of Lemma~\ref{lemtcSm} by
$\tfrac{m+n_0}{\ze-m}$, for which the bounds are
only slightly worse than in Lemma~\ref{leminiSmtSm}).

This yields the first part of the lemma, since we can now write
$$
\sum_{\omb\in\Om^\bul} \tcVto \bB_\omb y^{n_0} = 
\sum_{m\ge -n_0} \bigg( \sum_{\omb\in\Om^\bul,\, \norm{\omb}=m}
\be_{\omb,n_0} \tcVto \bigg) y^{n_0+m}
= \wt\Th\ii y^{n_0}
\quad \text{in $\bA[[y]]$.}
$$
By continuity of~$\De_m$, we also get the summability of 
$(\be_{\omb,n_0} \De_m\hcVto)_{\omb\in\Om^\bul,\,\norm{\omb}=m}$
in~$\bA$,
hence of the family $(-\be_{\omb,n_0}
V_a^\omb(m))_{\omb\in\Om^\bul,\,\norm{\omb}=m}$ obtained by extracting the
constant terms.
Let 
$$
C_{m,n_0} = \sum_{\omb\in\Om^\bul,\,\norm{\omb}=m} \be_{\omb,n_0} V_a^\omb(m)
\quad \text{in $\C$}.
$$
Thus $(V_a^\omb(m) \bB_\omb y^{n_0})_{\omb\in\Om^\bul}$ is summable in
$\bA[[y]]$, with sum $C_{m,n_0} y^{n_0+m}$.

Let $\Om^{k,R}$ ($k,R\in\N^*$) denote an exhaustion of~$\Om^\bul$ by finite
sets as in the proof of Proposition~\ref{propmultiplimould}.
We conclude by showing that $C_{m,n_0} y^{n_0+m} = \gC(m) y^{n_0}$.
This follows from that fact that the operators
$$
\gC^{k,R}(m) = \sum_{\omb\in\Om^{k,R}} V_a^\omb(m) \bB_\omb
$$
are all derivations of $\C[[y]]$ because of the alternality of~$V_a^\bul(m)$ (the
Leibniz rule is easily checked with the help of the cosymmetrality
of~$\bB_\bul$), thus their pointwise limit is also a derivation, which cannot be
anything but~$\gC(m)$ by virtue of~\eqref{eqdeterminC}.
\end{proof}

%%%%%%%%%%%%%%%%%%%%%%%%%%%%%%%%%%%%%%%%%%%%%%%%%%%%%%%%%%%%%%%%%%%%%%%%%
%%%%%%%%%%%%%%%%%%%%%%%%%%%%%%%%%%%%%%%%%%%%%%%%%%%%%%%%%%%%%%%%%%%%%%%%%

\parag
\emph{End of the proof of Theorem~\ref{thmBE}:}
In $\bA[[y]]$, the families
$(\wt\cV^\omb \bB_\omb y)_{\omb\in\Om^\bul}$ and 
$(\tcVto \bB_\omb y)_{\omb\in\Om^\bul}$
are summable, of sums
\begin{equation}	\label{eqwtphpsisum}
\wt\ph(z,y) = \sum_{\omb\in\Om^\bul} \wt\cV^\omb(z) \bB_\omb y,
\qquad
\wt\psi(z,y) = \sum_{\omb\in\Om^\bul} \tcVto(z) \bB_\omb y.
\end{equation}
The derivation of~$\bA[[y]]$ induced by~$\De_m$ is clearly continuous;
applying $\De_m$ to both sides of the first equation in~\eqref{eqwtphpsisum} and
using~\eqref{eqcaraccomould} and~\eqref{eqDemtcVb}, we find
\begin{multline*}
\De_m\wt\ph = \sum_\omb (\De_m \wt\cV^\omb) \bB_\omb y
= \sum_{\omb^1, \omb^2} \wt\cV^{\omb^1} V_a^{\omb^2}(m) \bB_{\omb^2}  \bB_{\omb^1} y
\\[1ex]
= \sum_{\omb^2} V_a^{\omb^2}(m) \bB_{\omb^2}  \wt\Th y
= \gC(m)\wt\ph
\end{multline*}
(with the help of Lemma~\ref{lemintermed} for the last identities).
Similarly, 
\begin{multline*}
\De_m\wt\psi = \sum_\omb (\De_m \tcVto) \bB_\omb y
= -\sum_{\omb^1, \omb^2} V_a^{\omb^1}(m) \tcVod \bB_{\omb^2}  \bB_{\omb^1} y
\\[1ex]
= -\sum_{\omb^2} \tcVod \bB_{\omb^2} \gC(m) y
= -\wt\Th\ii(C_m y^{m+1}) = -C_m(\wt\Th\ii y)^{m+1}.
\end{multline*}
\qed

%%%%%%%%%%%%%%%%%%%%%%%%%%%%%%%%%%%%%%%%%%%%%%%%%%%%%%%%%%%%%%%%%%%%%%%%%
%%%%%%%%%%%%%%%%%%%%%%%%%%%%%%%%%%%%%%%%%%%%%%%%%%%%%%%%%%%%%%%%%%%%%%%%%

\parag
\emph{Operator form of the Bridge Equation.}
As announced after the statement of Theorem~\ref{thmBE}, the Bridge Equation can be
given a form which involves the operators~$\wt\Th$ or~$\wt\Th\ii$ in a more symmetric
way.
This will require a further construction. 

\begin{prop}	\label{propResurzCVy}
Let $\bA = \tRsimpZ$. The set 
$$
\bA\{y\} = \Big\{ \sum_{n\ge0} \ti f_n(z) y^n \in \bA[[y]] \mid
\forall K\in \gK,\exists c,\La>0 \;\text{s.t.}\; \norm{\ti f_n}_K \le c \La^n
\; \text{for all $n$}
\Big\}
$$
is a subalgebra of~$\bA[[y]]$, which contains $\wt\ph(z,y)$ and~$\wt\psi(z,y)$
and which is invariant by all the alien derivations~$\De_m$.
Moreover, the substitution operators~$\wt\Th$ and~$\wt\Th\ii$
leave $\bA\{y\}$ invariant and the operators they induce on $\bA\{y\}$ satisfy
the ``alien chain rule''
$$
\De_m \wt\Th \ti f = \wt\Th \De_m \ti f + (\wt\Th\pa_y\ti f) \De_m\wt\ph, \quad
\De_m \wt\Th\ii \ti f = \wt\Th\ii \De_m \ti f + (\wt\Th\ii\pa_y\ti f) \De_m\wt\psi.
$$
\end{prop}

\noindent
\emph{Idea of the proof:}
The fact that $\wt\ph,\wt\psi\in\bA\{y\}$ follows easily from \eqref{ineqwhphn}--\eqref{ineqwhpsin}.
The other statements require symmetrically contractile paths, first to control the seminorm
$\norm{\,.\,}_K$ of a product of simple resurgent functions ($\bA$ is in fact
a Fr\'echet algebra), 
and then to study $\pa_y^n\ti f(z,\wt\ph_0(z))$ which appears in the substitution
of~$\wt\ph$ inside a series with resurgent coefficients:
$$
\ti f(z,\wt\ph) = \ti f(z,\wt\ph_0) + y \pa_y\ti f(z,\wt\ph_0)\wt\Phi_1 +
y^2\Big(\pa_y\ti f(z,\wt\ph_0)\wt\Phi_ 2 + \frac{1}{2!}\pa_y^2\ti f(z,\wt\ph_0)\wt\Phi_1^2\Big)
+ \dotsb
$$
with the notation~\eqref{eqdefPhiph}.
See \cite{kokyu} (\eg\ \S2.3, formula~(41)).
\qed

\begin{thm}	\label{thmBrOp}
We have the following identities in $\End_\C(\bA\{y\})$:
\begin{equation}	\label{eqBrOp}
\big[  \De_m, \wt\Th \big] = \gC(m) \wt\Th,
\qquad
\big[  \De_m, \wt\Th\ii \big] = - \wt\Th\ii \gC(m),
\end{equation}
for all $m\in\Z^*$.
\end{thm}

\begin{proof}
We must prove that $\wt\Th\De_m\wt\Th\ii-\De_m = -\gC(m)$.

The operators $\wt\Th$ and $\wt\Th\ii$ are mutually inverse $\bA$-linear automorphisms
of $\gA=\bA\{y\}$
and $\gC(m)$ is an $\bA$-linear derivation.
The operator~$\De_m$ is a derivation, it is not $\bA$-linear, but
$D = \wt\Th\De_m\wt\Th\ii-\De_m$ is an $\bA$-linear derivation;
indeed, if $\mu(z)\in\bA$ and $ f(z,y)\in\gA$, then 
\begin{multline*}
D(\mu f) = 
\wt\Th\De_m(\mu\wt\Th\ii f) - \De_m(\mu f) = \\
\wt\Th\big( \mu \De_m\wt\Th\ii f + (\De_m\mu) \wt\Th\ii f \big)
-\big( \mu\De_m f + (\De_m\mu) f \big) = \\
\mu\wt\Th\De_m\wt\Th\ii  f + (\De_m\mu) f 
- \mu\De_m  f - (\De_m\mu)  f
= \mu Df.
\end{multline*}
It is thus sufficient to check that the operator $D+\gC(m)$ vanishes on~$y$
(being a continuous $\bA$-linear derivation of~$\gA$, it'll have
to vanish everywhere).

But, in view of~\eqref{eqDempsiThiigCm}, $D y = \wt\Th\De_m\wt\psi =
-C_m(\wt\Th\wt\psi)^{m+1} = - C_m y^{m+1}$, as required.
\end{proof}

%%%%%%%%%%%%%%%%%%%%%%%%%%%%%%%%%%%%%%%%%%%%%%%%%%%%%%%%%%%%%%%%%%%%%%%%%
%%%%%%%%%%%%%%%%%%%%%%%%%%%%%%%%%%%%%%%%%%%%%%%%%%%%%%%%%%%%%%%%%%%%%%%%%

\parag
\emph{The Bridge Equation and the problem of analytic classification.}
We now explain why the coefficients~$C_m$ implied in the Bridge Equation are
``analytic invariants'' of the vector field~$X$.

Suppose we are given two saddle-node vector fields, $X_1$ and~$X_2$, of the
form~\eqref{eqdefX} and satisfying~\eqref{eqassA}. 
Both of them are formally conjugate to the normal form~$X_0$, hence they are
mutually formally conjugate. 
Namely, we have formal subsitution automorphisms $\Th_i$ (or~$\wt\Th_i$, when
using the variable~$z$ instead of~$x$) conjugating $X_i$ with~$X_0$, for $i=1,2$, hence
$$
\Th X_1 = X_2 \Th, \qquad \Th = \Th_2\ii \Th_1.
$$
The operator~$\Th$ is the substitution operator associated with 
$$
\th \colon (x,y) \mapsto \big( x, \ph(x,y) \big),
\qquad \ph(x,y) = \Th y = \ph_1\big(x,\psi_2(x,y)\big),
$$
which is the unique formal transformation of the form~\eqref{eqdefthph} such
that $X_1 = \th^* X_2$.

One can check that, when passing to the variable~$z$, one gets as a consequence
of Proposition~\ref{propResurzCVy} and Theorem~\ref{thmBrOp}:
$$
\wt\Th \in \End_\C(\bA\{y\}), \qquad
\big[  \De_m, \wt\Th \big] = \wt\Th_2\ii \big( \gC_1(m) - \gC_2(m) \big) \wt\Th_1,
\qquad m\in\Z^*,
$$
where $\gC_i(m) = C_{i,m} y^{m+1}\frac{\pa\,}{\pa y}$ is the derivation
appearing in the \rhs\ of the Bridge Equations~\eqref{eqBrOp} for~$X_i$.

If $X_1$ and~$X_2$ are holomorphically conjugate, then the unique formal
conjugacy~$\th$ is given by a convergent series~$\th(x,y)$, thus all the alien
derivatives of~$\wt\ph$ vanish and $\gC_1(m) = \gC_2(m)$ for all~$m$.
We thus have proved half of 
\begin{thm}	\label{thmCmInv}
Two saddle-node vector fields of the
form~\eqref{eqdefX} and satisfying~\eqref{eqassA} 
are analytically conjugate if and only if their Bridge
Equations~\eqref{eqBEtigA} share the same collection of coefficients
$(C_m)_{m\in\Z^*}$.
\end{thm}

According to this theorem, the numbers~$C_m$ thus constitute a complete system
of analytic invariants for a saddle-node vector field.

To complete the proof of Theorem~\ref{thmCmInv}, one needs to show the reverse
implication, 
\ie\ that the identities $\gC_1(m)=\gC_2(m)$ imply the convergence of
$\ph_1\big(x,\psi_2(x,y)\big)$.
This will follow from the results of next section, according to which the
coefficients~$C_m$ are related to another complete system of analytic
invariants, which admits a more geometric description.

%%%%%%%%%%%%%%%%%%%%%%%%%%%%%%%%%%%%%%%%%%%%%%%%%%%%%%%%%%%%%%%%%%%%%%%%%
%%%%%%%%%%%%%%%%%%%%%%%%%%%%%%%%%%%%%%%%%%%%%%%%%%%%%%%%%%%%%%%%%%%%%%%%%

\parag
We end this section with a look at simple cases of the general theory.

``Euler equation'' corresponds to $A(x,y) = x+y$, as mentioned in
Section~\ref{secSN}.
We may call Euler-like equations those which correspond to the case in which
$a_\eta=0$ for $\eta\ge1$, thus $A(x,y) = a_{-1}(x) + \big(1+a_0(x)\big) y$.
For them, the formal integral is explicit.

Set $\wt a_0(z) = a_0(-1/z) \in  z^{-2}\C\{z\ii\}$ and 
$\wt a_{-1}(z) = a_{-1}(-1/z) \in z\ii\C\{z\ii\}$ as usual.
Let $\wt\al(z)$ be the unique series such that $\pa_z\wt\al = \wt
a_0$ and $\wt\al \in z\ii\C\{z\ii\}$.
Set also $\wt\be = \wt a_{-1} \,\ee^{-\wt\al} \in z\ii\C\{z\ii\}$ and $\wh\be =
\cB\wt\be$ (which is an entire function of exponential type).
One finds
$$
\wt Y(z,u) = \wt\ph_0(z) + u\,\ee^{z+\wt\al(z)}, \qquad
\wt\ph_0 = - \ee^{\wt\al} \, \cB\ii \Big( \ze \mapsto \frac{\wh\be(\ze)}{\ze+1} \Big).
$$
Correspondingly, $\ph(x,y) = \Phi_0(x) + \Phi_1(x) y$ with 
$\Phi_0(x) = \wt\ph_0(-1/x)$ generically divergent and $\Phi_1(x) =
\ee^{\wt\al(-1/x)}$ convergent.

One has $C_m=0$ for every $m\in\Z\setminus\{-1\}$, but
$$
C_{-1} = \ee^{-\wt\al}  \De_{-1}\wt\ph_0 = -2\pi\I \, \wh\be(-1).
$$

%%%%%%%%%%%%%%%%%%%%%%%%%%%%%%%%%%%%%%%%%%%%%%%%%%%%%%%%%%%%%%%%%%%%%%%%%
%%%%%%%%%%%%%%%%%%%%%%%%%%%%%%%%%%%%%%%%%%%%%%%%%%%%%%%%%%%%%%%%%%%%%%%%%

\parag
Another particular case, much less trivial, is that of Riccati equations (see
\cite{Eca84}, \cite[Vol.~2]{Eca81} or \cite{BSSV}):
when $a_1\neq0$ and $a_\eta=0$ for $\eta\ge2$, hence $A(x,y) = a_{-1}(x) +
\big(1+a_0(x)\big) y + a_1(x) y^2$, one can check that the formal integral has a linear
fractional dependence upon the parameter~$u$:
$$
\wt Y(z,u) = \frac{\wt\ph_0(z) + u\,\ee^z \wt\chi(z)}{
1 + u\,\ee^z \wt\chi(z)\wt\ph_\infty(z)},
$$
where $\wt\ph_0$, $\wt\ph_\infty$ and $-1+\wt\chi$ belong to $z\ii\C[[z\ii]]$;
$\wt\ph_0$ and $1/\wt\ph_\infty$ can be found as the unique solutions of the
differential equation~\eqref{eqdiffeqX} in the fraction field $\C(\!(z\ii)\!)$.
Correspondingly, the normalising series $\ph(x,y)$ and~$\psi(x,y)$ have a linear
fractional dependence upon~$y$.

In the Riccati case, only $C_{-1}$ and~$C_1$ may be nonzero. Indeed, 
$$
\De_m\wt\ph_0 \neq 0 \;\Rightarrow\; m=-1,
\quad
\De_m\wt\ph_\infty \neq 0 \;\Rightarrow\; m=1,
\quad
\De_m\wt\chi \neq 0 \;\Rightarrow\; m=\pm1.
$$

%%%%%%%%%%%%%%%%%%%%%%%%%%%%%%%%%%%%%%%%%%%%%%%%%%%%%%%%%%%%%%%%%%%%%%%%%
%%%%%%%%%%%%%%%%%%%%%%%%%%%%%%%%%%%%%%%%%%%%%%%%%%%%%%%%%%%%%%%%%%%%%%%%%

\parag
We may call ``canonical Riccati equations'' the equations corresponding to a
function~$A$ of the form
$A(x,y) = y + \frac{1}{2\pi\I}B_- x + \frac{1}{2\pi\I}B_+ x y^2$, with $B_-,B_+\in\C$.
Thus, for them, the differential equation~\eqref{eqdiffeqX} reads
$$
\pa_z \wt Y = \wt Y - \frac{1}{2\pi\I z}(B_- + B_+ \wt Y^2).
$$
A direct mould computation based on~\eqref{eqdefCm}
is given in \cite{Eca81}, Vol.~2, pp.~476--480, yielding
$$
C_{-1} = B_- \sig(B_- B_+), \quad
C_{1} = - B_+ \sig(B_- B_+),
$$
with $\sig(b) = \frac{2}{b^{1/2}} \sin \frac{b^{1/2}}{2}$
(see \cite{BSSV} for a computation by another method).

%%%%%%%%%%%%%%%%%%%%%%%%%%%%%%%%%%%%%%%%%%%%%%%%%%%%%%%%%%%%%%%%%%%%%%%%%
%%%%%%%%%%%%%%%%%%%%%%%%%%%%%%%%%%%%%%%%%%%%%%%%%%%%%%%%%%%%%%%%%%%%%%%%%

\section{Relation with Martinet-Ramis's invariants}	\label{secMR} 

%%%%%%%%%%%%%%%%%%%%%%%%%%%%%%%%%%%%%%%%%%%%%%%%%%%%%%%%%%%%%%%%%%%%%%%%%
%%%%%%%%%%%%%%%%%%%%%%%%%%%%%%%%%%%%%%%%%%%%%%%%%%%%%%%%%%%%%%%%%%%%%%%%%

In this section, we continue to investigate the consequences of the resurgence
of the solution of the conjugacy equation for a saddle-node~$X$.
We shall now connect the ``alien computations'' of the previous section with
Martinet-Ramis's solution of the problem of analytic classification \cite{MR},
completing at the same time the proof of Theorem~\ref{thmCmInv}.

This will be done by comparing sectorial solutions of the conjugacy problem
obtained by Borel-Laplace summation on the one hand, and by deriving geometric
consequences of the Bridge Equation through exponentiation and summation on
the other hand
(this amounts to a resurgent description of the ``Stokes phenomenon'' for the
differential equation~\eqref{eqdiffeqX}).

%%%%%%%%%%%%%%%%%%%%%%%%%%%%%%%%%%%%%%%%%%%%%%%%%%%%%%%%%%%%%%%%%%%%%%%%%
%%%%%%%%%%%%%%%%%%%%%%%%%%%%%%%%%%%%%%%%%%%%%%%%%%%%%%%%%%%%%%%%%%%%%%%%%

\parag
Let us call \emph{Martinet-Ramis's invariants} of~$X$ the numbers
$\xi_{-1},\xi_1,\xi_2,\dotsc$ defined in terms of \'Ecalle's
invariants by the formulas
\begin{align}
\label{defxiCminus}
\xi_{-1} &= - C_{-1}, \\
\label{defxiCplus}
\xi_m &= \sum_{r\ge1}
\sum_{\substack{m_1,\dotsc,m_r\ge1  \\  m_1+\dotsb+m_r = m}}
\frac{(-1)^r}{r!} \be_{m_1,\ldots,m_r} C_{m_1}\dots C_{m_r},
\qquad m\ge1,
\end{align}
where, as usual, $\be_{m_1} = 1$ and
$\be_{m_1,\ldots,m_r} = (m_1+1)(m_1+m_2+1)\dotsm(m_1+\dotsb+m_{r-1})$
for $r\ge2$.

Observe that they are obtained by integrating backwards the vector fields
$$
\gC_- = \gC(-1) = C_{-1} \frac{\pa\,}{\pa u}, \qquad
\gC_+ = \sum_{m>0} \gC(m) = \sum_{m>0} C_m u^{m+1} \frac{\pa\,}{\pa u}.
$$
Indeed, the time-$(-1)$ maps of~$\gC_-$ and~$\gC_+$ are
\begin{equation}	\label{eqdefxipm}
u \mapsto \xi_-(u) = u + \xi_{-1}, \qquad
u \mapsto \xi_+(u) = u + \sum_{m>0}  \xi_m u^{m+1}
\end{equation}
(as can be checked by viewing $-\gC_+$ as an elementary mould-comould expansion
on the alphabet $\N^*$; the reason for changing the variable~$y$ into~$u$ will
appear later).\footnote{
Thus one always has $\xi_-(u) = u - C_{-1}$, and in the Riccati case as at the
end of the previous section
$\xi_+(u) = \frac{u}{1 - C_{1} u}$.
}

These numbers can also be defined directly from the iterated
integrals~$L_a^\omb$ of~\eqref{eqdefLao}:
\begin{prop}
The family
$\big( \be_\omb L_a^\omb \big)_{\omb\in\Om^\bul,\,\norm{\omb}=m}$
is summable in~$\C$ for each $m\in\Z^*$ and
$$
\xi_m = \sum_{\omb\in\Om^\bul,\, \norm{\omb} = m} \be_\omb L_a^\omb,
$$
with the convention $\xi_m = 0$ for $m\le-2$.
\end{prop}

\noindent
\emph{Idea of the proof.} The relations $\xi_\pm(u) = \sum_{\omb\in\Om^\bul} \pmL\omb
\bB_\omb u$ (where $\pmL\bul$ is defined by~\eqref{eqdefpmL}) formally follow
from the formula $\pmL{\bul} = \exp\big(-\pmV{\bul}\big)$ and
Lemma~\ref{lemintermed}, according to which 
$(\pmV{\omb} \bB_\omb)_{\omb\in\Om^\bul}$
is a pointwise summable family of operators of~$\bA[[u]]$ with sum~$\gC_\pm$.
The summability can be justified by the same kind of arguments as in the proof
of Proposition~\ref{propCm} and Theorem~\ref{thmBE}.
\qed

%%%%%%%%%%%%%%%%%%%%%%%%%%%%%%%%%%%%%%%%%%%%%%%%%%%%%%%%%%%%%%%%%%%%%%%%%
%%%%%%%%%%%%%%%%%%%%%%%%%%%%%%%%%%%%%%%%%%%%%%%%%%%%%%%%%%%%%%%%%%%%%%%%%

\parag
The formulas~\eqref{defxiCminus}--\eqref{defxiCplus} can be inverted so as to
express the $C_m$'s in terms of the $\xi_m$'s.
Theorem~\ref{thmCmInv} is thus equivalent to the fact that the $\xi_m$'s
constitute themselves a complete system of analytic invariants for the
saddle-node classification problem.
We shall now prove this fact directly.

In fact, we shall obtain more: the pair $(\xi_-,\xi_+)$ is a complete system of
analytic invariants and \emph{$\xi_+$ is necessarily convergent}.
Thus, not all collections of numbers $(C_m)_{m\in\{-1\}\cup\N^*}$ 
can appear as analytic invariants,
only those for which the corresponding $\xi_m$'s admit geometric bounds
$|\xi_m| \le K^m$ for $m\ge1$
(hence they have to satisfy Gevrey bounds themselves: $|C_m| \le K_1^m m!$ for
$m\ge1$).

This information will follow from the geometric interpretation of~$\xi_\pm$.
Martinet and Ramis have also showed that any collection 
$(\xi_m)_{m\in\{-1\}\cup\N^*}$ subject to the
previous growth constraint can be obtained as a system of analytic invariants
for some saddle-node vector field, but we shall not consider this question here.

%%%%%%%%%%%%%%%%%%%%%%%%%%%%%%%%%%%%%%%%%%%%%%%%%%%%%%%%%%%%%%%%%%%%%%%%%
%%%%%%%%%%%%%%%%%%%%%%%%%%%%%%%%%%%%%%%%%%%%%%%%%%%%%%%%%%%%%%%%%%%%%%%%%

\parag
Let us consider the saddle-node vector field~$X$ and its normal form~$X_0$ in
the variable $z=-1/x$ instead of~$x$:
$$
\wt X = \frac{\pa\,}{\pa z} + A(-1/z,y) \frac{\pa\,}{\pa y}, \qquad
\wt X_0 = \frac{\pa\,}{\pa z} + y \frac{\pa\,}{\pa y}.
$$
For $\eps\in \left]0,\pi/2\right[$ and $R>0$, we set 
\begin{align*}
\zDu &= \ao z \in\C \mid 
-\tfrac{\pi}{2}+\eps \le \arg z \le \tfrac{3\pi}{2}-\eps, \;
|z|  \ge R \af, \\
\zDl &= \ao z \in\C \mid 
-\tfrac{3\pi}{2}+\eps \le \arg z \le \tfrac{\pi}{2}-\eps, \;
|z|  \ge R \af,
\end{align*}
which are ``sectorial neighbourhoods of infinity'' in the $z$-plane
(corresponding to certain sectorial neighbourhoods of the origin in the $x$-plane).
Their intersection has two connected components:
\begin{align*}
\zDm &= 
\ao z \in\C \mid 
\tfrac{\pi}{2}+\eps \le \arg z \le \tfrac{3\pi}{2}-\eps, \;
|z|  \ge R \af
\subset \ao\RE z<0\af, \\
\zDp &= 
\ao z \in\C \mid 
-\tfrac{\pi}{2}+\eps \le \arg z \le \tfrac{\pi}{2}-\eps, \;
|z|  \ge R \af
\subset \ao\RE z>0\af. \\
\end{align*}

%%%%%%%%%%%%%%%%%%%%%%%%%%%%%%%%%%%%%%%%%%%%%%%%%%%%%%%%
%%%%%%%%%%%%%%%%%%%%%%%%%%%%%%%%%%%%%%%%%%%%%%%%%%%%%%%%

\begin{thm}	\label{thmximInv} 
Let $\eps\in \left]0,\pi/2\right[$. Then there exist $R,\rho>0$ such that:
\begin{enumerate}[(i)]
\item
By Borel-Laplace summation, the formal series~$\wt\ph_n(z)$ give
rise to functions $\wt\ph_n^{up}(z)$, resp.\ $\wt\ph_n^{low}(z)$, which are analytic
in $\zDu$, resp.\ $\zDl$, such that the formulas
$$
\wt\ph^{up}(z,y) = \sum_{n\ge0} \wt\ph_n^{up}(z) y^n, \quad
\wt\ph^{low}(z,y) = \sum_{n\ge0} \wt\ph_n^{low}(z) y^n
$$
define two functions $\wt\ph^{up}$ and~$\wt\ph^{low}$ analytic in 
$\zDu\times \ao |y|\le \rho\af$,
resp.\ $\zDl\times \ao |y|\le \rho\af$,
and each of the transformations 
$$
\thu(z,y) = \big(z,\wt\ph^{up}(z,y)\big), \quad
\thl(z,y) = \big(z,\wt\ph^{low}(z,y)\big)
$$
is injective in its domain and establishes there a conjugacy between the
normal form~$\wt X_0$ and the saddle-node vector field~$\wt X$.
\item
The series~$\xi_+$ of~\eqref{eqdefxipm} has positive radius of convergence and
the upper and lower normalisations are connected by the formulas
\begin{alignat*}{3}
&\thu(z,y) &=& \thl\big(z,\xi_-(y\,\ee^{-z})\,\ee^z\big)
&=& \thl(z,y + \xi_{-1}\ee^z)
\\
\intertext{for $z\in\zDm$ and $|y|\le \rho$, whereas}
&\thl(z,y) &=& \thu\big(z,\xi_+(y\,\ee^{-z})\,\ee^z\big)
&=& \thu\big(z, y + \xi_1 y^2\ee^{-z} + \xi_2 y^3\ee^{-2z} + \dotsb)\\
\intertext{for $z\in\zDp$ and $|y|\le \rho$.}
\end{alignat*}
\item
The pair $(\xi_-,\xi_+)$ is a complete system of
analytic invariants for~$X$.
\end{enumerate}
\end{thm}

%%%%%%%%%%%%%%%%%%%%%%%%%%%%%%%%%%%%%%%%%%%%%%%%%%%%%%%%
%%%%%%%%%%%%%%%%%%%%%%%%%%%%%%%%%%%%%%%%%%%%%%%%%%%%%%%%

As already mentioned, Theorem~\ref{thmximInv} contains Theorem~\ref{thmCmInv}.
The rest of this section is devoted to the proof of Theorem~\ref{thmximInv}.

%%%%%%%%%%%%%%%%%%%%%%%%%%%%%%%%%%%%%%%%%%%%%%%%%%%%%%%%%%%%%%%%%%%%%%%%%
%%%%%%%%%%%%%%%%%%%%%%%%%%%%%%%%%%%%%%%%%%%%%%%%%%%%%%%%%%%%%%%%%%%%%%%%%

\parag
In view of inequalities~\eqref{ineqwhpsin}--\eqref{inequniformphn}, the
principal branches of the Borel transforms~$\wh\ph_n(\ze)$ and~$\wh\psi_n(\ze)$
admit exponential bounds of the form $K L^n \,  \ee^{C |\ze|}$ in the sectors
$\ao \ze\in\C \mid \tfrac{\eps}{2} \le \arg\ze \le \pi-\tfrac{\eps}{2} \af$
and $\ao \ze\in\C \mid \pi+\tfrac{\eps}{2} \le \arg\ze \le 2\pi-\tfrac{\eps}{2} \af$.
Using the directions of the first sector for instance, we can
define analytic functions by gluing the Laplace transforms
corresponding to various directions
$$
\wt\ph_n^{low}(z) = \int_0^{\ee^{\I\th}\infty} \wh\ph_n(\ze)\,\ee^{-z\ze}  \, \dd\ze,
\qquad
\wt\psi_n^{low}(z) = \int_0^{\ee^{\I\th}\infty} \wh\psi_n(\ze)\,\ee^{-z\ze}  \, \dd\ze,
$$
with $\th \in [\tfrac{\eps}{2},\pi-\tfrac{\eps}{2}]$.
If we take $R$ large enough, then the union of the
half-planes $\ao
\RE(z\,\ee^{\I\th}) > C \af$ contains $\zDl$ and the functions
$$
\wt\ph^{low}(z,y) = \sum_{n\ge0} \wt\ph_n^{low}(z) y^n,
\qquad
\wt\psi^{low}(z,y) = \sum_{n\ge0} \wt\psi_n^{low}(z) y^n
$$
are analytic for $z\in\zDl$ and $|y|\le\rho$ as soon as $\rho<1/L$.

The standard properties of Borel-Laplace summation ensure that the relations
$y = \wt\ph\big( z, \wt\psi(z,y) \big) = 
\wt\psi\big( z, \wt\ph(z,y) \big)$
and
$\wt X_0 \wt\ph(z,y) = A\big( -1/z, \wt\ph(z,y) \big)$
yield similar relations for $\wt\ph^{low}$ and $\wt\psi^{low}$, possibly in smaller
domains (because $\wt\ph^{low}(z,y)-y$ and $\wt\psi^{low}(z,y)-y$ can be made
uniformly small by increasing~$R$ and diminishing~$\rho$).
Hence the transformations
$$
(z,y) \mapsto \big(z,\wt\ph^{low}(z,y)\big), \qquad
(z,y) \mapsto \big(z,\wt\psi^{low}(z,y)\big)
$$
(or rather the sectorial germs they represent)
are mutually inverse and establish a conjugacy between~$\wt X_0$ and~$\wt X$.

We define similarly $\wt\ph^{up}(z,y)$ and $\wt\psi^{up}(z,y)$ with the desired
properties, by means of Laplace transforms in directions belonging to
$[\pi+\tfrac{\eps}{2},2\pi-\tfrac{\eps}{2}]$.
This yields the first statement in Theorem~\ref{thmximInv}.

%%%%%%%%%%%%%%%%%%%%%%%%%%%%%%%%%%%%%%%%%%%%%%%%%%%%%%%%%%%%%%%%%%%%%%%%%
%%%%%%%%%%%%%%%%%%%%%%%%%%%%%%%%%%%%%%%%%%%%%%%%%%%%%%%%%%%%%%%%%%%%%%%%%

\parag
We now have at our disposal two sectorial normalisations
$$
\thl \colon (z,y) \mapsto \big(z,\wt\ph^{low}(z,y)\big), \qquad
\thu \colon (z,y) \mapsto \big(z,\wt\ph^{up}(z,y)\big),
$$
which are defined in different but overlapping domains, and which admit the same
asymptotic expansion \wrt~$z$ (when one first expands in powers of~$y$).
If we consider $\big(\thu\big)\ii\circ\,\thl$ or $\big(\thl\big)\ii\circ\,\thu$ in
one of the two components of $\zDl\cap\zDu$, we thus get a transformation of the
form
\begin{equation}	\label{eqdefchizy}
(z,y) \mapsto \big( z, \chi(z,y) \big)
\end{equation}
which conjugates the normal form~$\wt X_0$ with itself, to which one can apply the following:

\begin{lemma}
Let $\cD$ be a domain in~$\C$.
Suppose that the transformation $(z,y) \mapsto \big( z, \chi(z,y) \big)$ is
analytic and injective for
$z \in \cD$ and $|y| \le \rho$,
and that it conjugates~$\wt X_0$ with itself.
Then there exists $\xi(u)\in\C\{u\}$ such that
\begin{equation}	\label{eqchixi}
\chi(z,y) = \xi(y\,\ee^{-z}) \ee^z.
\end{equation}
\end{lemma}

Such transformations are called \emph{sectorial isotropies} of the normal form.

\begin{proof}
By assumption $\chi = \wt X_0\chi$.
Since $y = \wt X_0 y$, this implies that $\frac{1}{y}\chi(z,y)$ is a first
integral of~$\wt X_0$.
Thus $\frac{1}{u\,\ee^z}\chi(z,u\,\ee^z)$ is independent of~$z$ and can be
written $\frac{\xi(u)}{u}$, where obviously $\xi(u)\in\C\{u\}$.
\end{proof}

When $\chi(z,y)$ comes from
$\big(\thu\big)\ii\circ\,\thl$ or $\big(\thl\big)\ii\circ\,\thu$, we have a
further piece of information:
in the Taylor expansion $\chi(z,y)- y = \sum_{n\ge0} \chi_n(z) y^n$,
each component $\chi_n(z)$ admits the null series as asymptotic expansion in
$\zDpm$
(the transformation~\eqref{eqdefchizy} is asymptotic to the identity because
$\thu$ and~$\thl$ share the same asymptotic expansion).
This has different implications according to whether the domain $\cD$ is $\zDm$
or $\zDp$.

Indeed, if we expand $\xi(u)-u = \sum_{n\ge0} \al_n u^n$, we get
$$
\chi_0(z) = \al_0\, \ee^{z}, \quad
\chi_1(z) = \al_1, \quad
\chi_2(z) = \al_2\, \ee^{-z}, \quad
\chi_3(z) = \al_3\, \ee^{-2z}, \, 
\dotsc
$$
hence
\begin{align*}
\cD = \zDm \subset \ao \RE z <0 \af &\ens\Rightarrow\ens
\al_n = 0 \ens\text{for $n\neq0$},\\
\cD = \zDp \subset \ao \RE z >0 \af &\ens\Rightarrow\ens
\al_0 = \al_1 = 0.
\end{align*}
The upshot is that there exist $\al_0\in\C$ and
$\xi(u) = u + \al_2 u^2 + \al_3 u^3 \in \C\{u\}$
such that
\begin{alignat*}{3}
&\big(\thl\big)\ii\circ\,\thu(z,y) &=& (z,y + \al_0\ee^z),
&\qquad& z\in\zDm,
\\
&\big(\thu\big)\ii\circ\thl(z,y) &=& 
(z, y + \al_2 y^2\ee^{-z} + \al_3 y^3\ee^{-2z} + \dotsb),
&\qquad& z\in\zDp.
\end{alignat*}
%

%%%%%%%%%%%%%%%%%%%%%%%%%%%%%%%%%%%%%%%%%%%%%%%%%%%%%%%%%%%%%%%%%%%%%%%%%
%%%%%%%%%%%%%%%%%%%%%%%%%%%%%%%%%%%%%%%%%%%%%%%%%%%%%%%%%%%%%%%%%%%%%%%%%

\parag
It is elementary to check that the pair of sectorial isotropies
$\Big( 
\big(\thl\big)\ii \circ\, \thu_{\hspace{1.2em}|\zDm}, 
\big(\thu\big)\ii \circ\, \thl_{\hspace{.9em}|\zDp}
\Big)$
is a complete system of analytic invariants for~$X$:
suppose indeed that two saddle-node vector fields~$X_1$ and~$X_2$ are given and
that we wish to know whether the unique formal transformation~$\th$ of the
form~\eqref{eqdefthph} which conjugate them is convergent,
then 
$\thu_2 \circ \big(\thu_1\big)\ii$ and
$\thl_2 \circ \big(\thl_1\big)\ii$ are two sectorial conjugacies between~$X_1$
and~$X_2$ defined in different but overlapping domains and admitting~$\th$ as
asymptotic expansion (up to the change $x=-1/z$);
they coincide and define an analytic conjugacy iff
$\big(\thl_2\big)\ii  \circ \thu_2 = 
\big(\thl_1\big)\ii  \circ \thu_1$
in both components of the intersection of the domains.

%%%%%%%%%%%%%%%%%%%%%%%%%%%%%%%%%%%%%%%%%%%%%%%%%%%%%%%%%%%%%%%%%%%%%%%%%
%%%%%%%%%%%%%%%%%%%%%%%%%%%%%%%%%%%%%%%%%%%%%%%%%%%%%%%%%%%%%%%%%%%%%%%%%

\parag
Therefore, it only remains to be checked that
$\al_0=\xi_1$ and $\xi=\xi_+$.
This will follow from the interpretation of the operators~$\De_m^+$ as
components of the ``Stokes automorphism''.
For this part, the reader may consult the end of \S2.4 in \cite{kokyu}.

Suppose that a simple resurgent functions $c\,\de+\wh\ph\in\hRsimpZ$ has the
following property:
the functions~$\wh\chi_m$ defined by $\De_m^+(c\,\de+\wh\ph) = \ga_m\,\de +
\wh\chi_m$ and~$\wh\ph$ itself have at most exponential growth in each
non-horizontal directions, so that one can consider the Laplace transforms
$\cL^\th\wh\ph(z) = \int_0^{\ee^{\I\th}\infty} \wh\ph(\ze)\,\ee^{-z\ze}  \, \dd\ze$
or $\cL^\th\wh\chi_m(z)$ for $\th\in\left]\eps,\pi-\eps\right[$ or
$\th\in\left]\pi+\eps,2\pi-\eps\right[$, which are analytic in
sectorial neighbourhoods of infinity of the form $\zDl$ or $\zDu$.
Let $\th<0<\th'$, with $\th$ and~$\th'$ both close to~$0$;
by deforming a contour of integration, 
one deduces from the definition~\eqref{eqdefDemplus} that, 
for any $M\in\N^*$ and $\sig\in\left]0,1\right[$,
$$
c + \cL^\th\wh\ph(z) = c + \cL^{\th'}\wh\ph(z) + \sum_{m=1}^M \ee^{-mz}
\big( \ga_m + \cL^{\th'}\wh\chi_m(z) \big) + O(|\ee^{-(M+\sig)z}|)
$$
in the sectorial neighbourhood of infinity obtained by imposing that both
$\RE(z\,\ee^{\I\th})$ and $\RE(z\,\ee^{\I\th'})$ be large enough, which is
contained in the right half-plane $\ao \RE z > 0 \af$. 

Let us denote this by: 
$\cL^\th(c\,\de+\wh\ph) \sim 
\sum_{m\ge0}  \ee^{-mz}\cL^{\th'}\De_m^+(c\,\de+\wh\ph)$ in $\ao \RE z > 0 \af$. 
Similarly, if $\th<\pi<\th'$ with $\th$ and~$\th'$ both close to~$\pi$, one gets
$\cL^\th(c\,\de+\wh\ph) \sim 
\sum_{m\le0}  \ee^{-mz}\cL^{\th'}\De_m^+(c\,\de+\wh\ph)$ in the left half-plane 
$\ao \RE z < 0 \af$.

We can even write $\cL^\th \sim \cL^{\th'}\circ \sum_{m \ge 0} \dDep$ in $\ao \RE z > 0 \af$
and $\cL^\th \sim \cL^{\th'}\circ \sum_{m \le 0} \dDep$ in $\ao \RE z < 0 \af$, if
we define properly $\dDep$ in the convolutive model. See \cite{kokyu}: 
$\dDep = \tau_m \circ \De_m^+$, with a shift operator $\tau_m \colon \hRsimpZ \to
\tau_m(\hRsimpZ)$, the target space being the set of simple resurgent
functions ``based at~$m$'' (instead of being based at the origin).
On the other hand, we can rephrase~\eqref{eqDepexpDe} as
$\sum_{m \ge 0} \dDep = \exp\Big( \sum_{m>0} \dDem \Big)$,
$\sum_{m\le0} \dDep = \exp\Big( \sum_{m<0} \dDem \Big)$.

Apply this to $\wt Y(z,u)$ (or, rather, to each of its components):
when $\th$ and~$\th'$ are close to~$0$, we have $\cL^\th\wh Y = \wt Y^{up}$ and
$\cL^{\th'}\wh Y = \wt Y^{low}$ in~$\zDp$, hence, in view of the Bridge Equation,
$\wt Y^{up} \sim (\cL^{\th'}\circ \exp\gC_+) \wh Y$,
which yields $\wt Y^{up}(z,u) \sim \wt Y^{low}\big(z,(\xi_+)\ii(u)\big)$ in~$\zDp$.
Similarly, $\wt Y^{low}(z,u) \sim \wt Y^{up}\big(z,(\xi_-)\ii(u)\big)$ in the domain~$\zDm$.
When interpreting these relations componentwise with respect to~$u$ and modulo
$O(|\ee^{\pm(M+\sig)z}|)$ in $\zDpm$ with arbitrarily large~$M$, we get the desired
relations between $\wt\ph^{up}(z,y) = \wt Y^{up}(z,y  \ee^{-z})$ and
$\wt\ph^{low}(z,y) = \wt Y^{low}(z,y  \ee^{-z})$.

%%%%%%%%%%%%%%%%%%%%%%%%%%%%%%%%%%%%%%%%%%%%%%%%%%%%%%%%%%%%%%%%%%%%%%%%%
%%%%%%%%%%%%%%%%%%%%%%%%%%%%%%%%%%%%%%%%%%%%%%%%%%%%%%%%%%%%%%%%%%%%%%%%%

\section{The resurgence monomials $\wt\cU_a^\omb$'s and the freeness of alien
derivations}	\label{secResurMonom} 

%%%%%%%%%%%%%%%%%%%%%%%%%%%%%%%%%%%%%%%%%%%%%%%%%%%%%%%%%%%%%%%%%%%%%%%%%
%%%%%%%%%%%%%%%%%%%%%%%%%%%%%%%%%%%%%%%%%%%%%%%%%%%%%%%%%%%%%%%%%%%%%%%%%

%%%%%%%%%%%%%%%%%%%%%%%%%%%%%%%%%%%%%%%%%%%%%%%%%%%%%%%%%%%%%%%%%%%%%%%%%
%%%%%%%%%%%%%%%%%%%%%%%%%%%%%%%%%%%%%%%%%%%%%%%%%%%%%%%%%%%%%%%%%%%%%%%%%

\parag
The first goal of this section is to construct families of simple resurgent functions
which form closed systems for multiplication and alien derivations in the
following sense:

\begin{definition}	
We call $\De$-friendly monomials the members of any family of simple resurgent functions
$(\wt\cU^{\om_1,\dotsc,\om_r})_{r\ge0,\,\om_1,\dotsc,\om_r\in\Z^*}$, such that on
the one hand
\begin{equation}	\label{eqsystDefriend}
\De_m \wt\cU^{\om_1,\dotsc,\om_r} = \left| \begin{alignedat}{2} 
&\wt\cU^{\om_2,\dots,\om_r} &\quad &\text{if $r\ge1$ and $\om_1=m$,} \\
& \ens\quad 0 &\quad &\text{if not,}
\end{alignedat}  \right.
\end{equation}
for every $m\in\Z^*$, and on the other hand
$\wt\cU^\est=1$ and
$$
\wt\cU^\alb \wt\cU^\beb =
\sum_{\omb\in\Om^\bul}  \sh{\alb}{\beb}{\omb} \wt\cU^\omb,
\qquad \alb,\beb\in (\Z^*)^\bul,
$$
\ie, when viewed as a mould, $\wt\cU^\bul \in \hM^\bul(\Z^*,\tRsimpZ)$ is symmetral.
\end{definition}

J.~\'Ecalle calls $\De$-friendly such resurgent functions by contrast with the functions
$\wt\cV_a^{\om_1,\dotsc,\om_r}$, which can be termed ``$\pa$-friendly monomials'' because
of~\eqref{eqdefnewcV} (using $\pa$ as short-hand for $\frac{\dd\,}{\dd z}$).

As a matter of fact, $\De$-friendly monomials will be defined with the help of
the moulds $\wt\cV_a^\bul$, $V_a^\bul$ of Section~\ref{secBESN} and mould composition,
but we first need to enlarge slightly the definition of mould composition.

%%%%%%%%%%%%%%%%%%%%%%%%%%%%%%%%%%%%%%%%%%%%%%%%%%%%%%%%%%%%%%%%%%%%%%%%%
%%%%%%%%%%%%%%%%%%%%%%%%%%%%%%%%%%%%%%%%%%%%%%%%%%%%%%%%%%%%%%%%%%%%%%%%%

\parag
We thus begin with a kind of addendum to Sections~\ref{secAlgMoulds} and~\ref{secAltSym}.
Assume that~$\bA$ is a commutative $\C$-algebra, the unit of which is
denoted~$1$, and~$\Om$ is a commutative semigroup, the operation of which is denoted
additively. 
We still use the notations $\norm{\omb} = \om_1 + \dotsb + \om_r$ and $\norm{\est}=0$.

Let us call {\em restricted moulds} the elements of $\hMst$, where
$\Om^* = \Om \setminus \{0\}$.
The example we have in mind is $\Om = \Z$ and $\bA = \C$ or~$\tRsimpZ$.

\begin{definition}
We call {\em licit mould} any restricted mould~$U^\bul$ such that
$$
\norm{\omb} = 0 \ens\Rightarrow\ens 
U^\omb = 0
$$
for any $\omb \in (\Om^*)^\bul$.
The set of licit moulds will be denoted $\hMlic$.
\end{definition}

The set $\hMlic$ is clearly an $\bA$-submodule of~$\hMst$, but not an
$\bA$-subalgebra.
Notice that $U^\bul\in\hMlic$ implies $U^\est=0$.

% Composition with licit moulds: definition and basic properties.

We now define the {\em composition of a restricted mould and a licit mould} as follows: 
$$
(M^\bul,U^\bul) \in \hMst\times\hMlic \mapsto
C^\bul = M^\bul \circ U^\bul \in \hMst,
$$
with $C^\est = M^\est$ and, for $\omb\neq\est$,
$$
C^\omb = \sum_{ \substack{s\ge1,\,\omb = \omb^1 \concsm \omb^s \\
\norm{\omb^1},\dotsc,\norm{\omb^s}\neq0} }
M^{(\norm{\omb^1},\dotsc,\norm{\omb^s})} 
U^{\omb^1}  \dotsm U^{\omb^s}.
$$
The map $M^\bul \mapsto M^\bul \circ U^\bul$ is clearly $\bA$-linear; 
we leave it to the reader\footnote{
The verification of most of the properties indicated in this paragraph can be
simplified by observing that the canonical restriction map
$\rho \colon \hM^\bul(\Om,\bA) \to \hMst$ 
is an $\bA$-algebra homomorphism which satisfies
$\rho(M^\bul\circ U^\bul) = \rho(M^\bul) \circ \rho(U^\bul)$
for any two moulds~$M^\bul$ and~$U^\bul$ such that $\rho(U^\bul)$ is licit
and which preserves alternality and symmetrality.
}
to check that it is an $\bA$-algebra homomorphism, that
$$
U^\bul, V^\bul \in \hMlic \ens\Rightarrow\ens U^\bul \circ V^\bul \in \hMlic,
$$
and that
$$
M^\bul\in \hMst \;\text{and}\; U^\bul, V^\bul \in \hMlic 
\ens\Rightarrow\ens
(M^\bul \circ U^\bul) \circ V^\bul = M^\bul \circ (U^\bul \circ V^\bul).
$$
The {\em restricted identity mould} is 
$$
I_*^\bul \colon \omb\in(\Om^*)^\bul \mapsto I_*^\omb = 
\left| \begin{aligned} 
1 \quad &\text{if $r(\omb) = 1$,}\\
0 \quad &\text{if $r(\omb) \neq 1$.}
\end{aligned}  \right.
$$
It is a licit mould, which satisfies
$M^\bul \circ I_*^\bul = M^\bul$ for any restricted mould~$M^\bul$
and  $I_*^\bul \circ U^\bul = U^\bul $ for any licit mould~$U^\bul$.
One can check that a licit mould~$U^\bul$ admits an inverse for composition iff
$U^\omb$ is invertible in~$\bA$ whenever $r(\omb)=1$.

A proposition analogous to Proposition~\ref{propcomposalt} holds. 
In particular, \emph{alternal invertible licit moulds form a subgroup of the
composition group of invertible licit moulds}.\footnote{
This can be checked by means of the restriction homomorphism of the previous
footnote: 
if $U^\bul$ is licit and $U^\omb$ is invertible whenever $r(\omb)=1$, then any
$U_0^\bul\in\hM^\bul(\Om,\bA)$ such that $\rho(U_0^\bul) = U^\bul$ and
$U_0^{(0)}=1$ is an invertible mould, the composition inverse of which has a
restriction~$V^\bul$ which satisfies $U^\bul\circ V^\bul = V^\bul\circ U^\bul =
I_*^\bul$; if moreover $U^\bul$ is alternal, then one can choose $U_0^\bul$
alternal (take $U_0^\omb = 0$ whenever $r(\omb)\ge2$ and one of the letters
of~$\omb$ is~$0$), thus its inverse and the restriction of its inverse are alternal.
}
The property that inversion of licit moulds preserves alternality will be used
in the next paragraph.

%%%%%%%%%%%%%%%%%%%%%%%%%%%%%%%%%%%%%%%%%%%%%%%%%%%%%%%%%%%%%%%%%%%%%%%%%
%%%%%%%%%%%%%%%%%%%%%%%%%%%%%%%%%%%%%%%%%%%%%%%%%%%%%%%%%%%%%%%%%%%%%%%%%

\parag
We now take $\Om = \Z$ and $\bA = \tRsimpZ$.
Assume that $a = (\wh a_\eta)_{\eta\in\Z^*}$ is any family of entire functions
such that $\wh a_\eta(\eta)\neq0$ for each $\eta\in\Z^*$.
We still use the notations
$\wt a_\eta = \cB\ii\wh a_\eta \in z\ii\C[[z\ii]]$ and
$J_a^\omb = \wt a_\eta$ if $\omb=(\eta)$, $0$ if not.
We recall that, according to Section~\ref{secBESN}, the equation
\begin{equation}	\label{eqdeftV}
(\pa+\na)\tcVab = -  \tcVab \times J_a^\bul
\end{equation}
defines a symmetral mould $\tcVab \in \hMst$, and that, for each $m\in\Z^*$, we
have an alternal scalar mould $\Vt^\bul(m) = -V_a^\bul(m) \in \hMstC$ which satisfies
\begin{gather}	\label{eqaldertcV}
\De_m\tcVab = \Vt^\bul(m) \times  \tcVab,\\
\label{eqnultvVm}
\Vt^\omb(m) \neq 0 \ens\Rightarrow\ens \norm{\omb}=m.
\end{gather}
Moreover $\Vt^{(\eta)}(\eta) = 2\pi\I\, \wh a_\eta(\eta)$.

\begin{thm}
The formula $\Vt^\bul = \sum_{m\in\Z^*}  \Vt^\bul(m)$ defines an alternal scalar
licit mould, which admits a composition inverse~$U_a^\bul$.
The formula 
\begin{equation}	\label{eqdefcUa}
\wt\cU_a^\bul = \tcVab \circ U_a^\bul \,\in\, \hMstR
\end{equation}
defines a family of $\De$-friendly monomials~$\wt\cU_a^\omb$.
\end{thm}

\begin{proof}
In view of~\eqref{eqnultvVm}, the definition of~$\Vt^\bul$ makes sense and its
alternality follows from the alternality of each $\Vt^\bul(m)$.
This mould is clearly licit, and 
$\Vt^{(\eta)} = 2\pi\I\, \wh a_\eta(\eta)  \neq0$, hence its invertibility.

The general properties of the composition of a restricted mould and a licit
mould ensure that \eqref{eqdefcUa} defines a symmetral mould.
Its alien derivatives are easily computed since $U_a^\bul$ is a scalar mould:
$$
\De_m \wt\cU_a^\bul = (\De_m\tcVab)  \circ U_a^\bul = 
(\Vt^\bul(m)\times\tcVab) \circ U_a^\bul = I_m^\bul \times \wt\cU_a^\bul,
$$
with $I_m^\bul = \Vt^\bul(m)\circ U_a^\bul$ (the last identity follows from the
$\bA$-algebra homomorphism property of post-composition with~$U_a^\bul$).
The conclusion follows from the fact that 
\begin{equation}	\label{eqdefIm}
I_m^\omb = 1 \ens\text{if $\omb=(m)$,} \quad 0 \ens\text{if not.}
\end{equation}
This formula can be checked by introducing the map
$\rho_m \colon M^\bul  \in  \hMst \mapsto M_m^\bul  \in  \hMst$ defined by $M_m^\omb
= M^\omb$ if $\norm{\omb}=m$, $0$ if not, and observing that
$\rho_m(M^\bul\circ U^\bul) = \rho_m(M^\bul) \circ U^\bul$ for any licit
mould~$U^\bul$;
thus $I_m^\bul = \rho_m(\Vt^\bul) \circ U_a^\bul = \rho_m(I_*^\bul)$.
\end{proof}

\begin{remark}
An analogous computation yields
$$
(\pa+\na) \wt\cU_a^\bul = -\wt\cU_a^\bul \times \wt K^\bul
$$
with a licit alternal mould $\wt K^\bul \in \hM^\bul(\Z^*,\C[[z\ii]])$ defined by
$\wt K^\omb = U_a^\omb \, \wt a_{\norm{\omb}}$ if $\norm{\omb}\neq0$.
\end{remark}

%%%%%%%%%%%%%%%%%%%%%%%%%%%%%%%%%%%%%%%%%%%%%%%%%%%%%%%%%%%%%%%%%%%%%%%%%
%%%%%%%%%%%%%%%%%%%%%%%%%%%%%%%%%%%%%%%%%%%%%%%%%%%%%%%%%%%%%%%%%%%%%%%%%

\parag
As an application of the existence of $\De$-friendly monomials, we now show

\begin{thm}	\label{thmfree}
Let $\bA = \RsimpZ$.
The subalgebra of $\End_\C \bA$ generated by the
operators~$\De_m$, $m\in\Z^*$, is isomorphic to the free associative algebra on~$\Z^*$.
\end{thm}

\noindent 
In fact, we shall prove a stronger statement:
for any non-commutative polynomial {\em with coefficients in~$\bA$},
$$
P = \sum_{(m_1,\dots,m_r)\in\gF} \wt\ph^{m_1,\dots,m_r}
\De_{m_r}\dotsm\De_{m_1},
\qquad \text{$\gF$ finite subset of~$(\Z^*)^\bul$,}
$$
there exists $\wt\psi\in\bA$ such that $P\wt\psi\neq0$, unless all the
coefficients $\wt\ph^{m_1,\dots,m_r}$ are zero.
Thus there is no non-trivial polynomial relation between the alien derivations~$\De_m$.

\begin{proof}
Assume that not all the coefficients are zero.
We may suppose $\gF\neq\est$ and $\wt\ph^\omb\neq0$ for each $\omb\in\gF$.
%
% we set $\wt\ph^\omb = 0$ for $\omb\in (\Z^*)^\bul \setminus \gF$.
%
Choose $\mb=(m_1,\dotsc,m_r)\in\gF$ with minimal length; then, for any
family of $\De$-friendly monomials $\wt\cU^\bul$, we find
$P\wt\cU^\mb = \wt\ph^{m_1,\dots,m_r} \neq 0$
as a consequence of 
\begin{equation}	\label{eqiterDeDefr}
\De_{m_s}\dotsm\De_{m_1} \wt\cU^\omb = \left| \begin{alignedat}{2} 
&\wt\cU^{\nb} &\quad &\text{if $\omb = (m_1,\dotsc,m_s) \conc \nb$ with $\nb\in(\Z^*)^\bul$,} \\
& \; 0 &\quad &\text{if not.}
\end{alignedat}  \right.
\end{equation}
\end{proof}

%%%%%%%%%%%%%%%%%%%%%%%%%%%%%%%%%%%%%%%%%%%%%%%%%%%%%%%%%%%%%%%%%%%%%%%%%
%%%%%%%%%%%%%%%%%%%%%%%%%%%%%%%%%%%%%%%%%%%%%%%%%%%%%%%%%%%%%%%%%%%%%%%%%

\parag
Let us call \emph{resurgence constant} any $\wt\ph\in\tRsimpZ$ such that
$\De_m\wt\ph = 0$ for any $m\in\Z^*$.
This is equivalent to saying that $\cB\wt\ph = c\,\de + \wh\ph(\ze)$ with
$c\in\C$ and $\wh\ph$ entire
(in particular every convergent series $\wt\ph(z)\in\C\{z\ii\}$ is a resurgence
constant, but the converse is not true since we did not require the Borel
transform to be of exponential type: the entire function~$\wh\ph$ might have
order~$>1$).

Resurgence constants form a subalgebra~$\wt\gP_0$ of~$\tRsimpZ$.

\begin{prop}
Let $\wt\cU_1^\bul$ and $\wt\cU_2^\bul$ be two moulds in $\hMstR$ and suppose
that $\wt\cU_1^\bul$ is a family of $\De$-friendly monomials.
Then $\wt\cU_2^\bul$ is a family of $\De$-friendly monomials iff if
there exists a symmetral mould $\wt M^\bul \in \hM^\bul(\Z^*,\wt\gP_0)$ such that
\begin{equation}	\label{eqrelcUcU}
\wt\cU_2^\bul = \wt\cU_1^\bul \times \wt M^\bul.
\end{equation}
\end{prop}

\noindent
Thus all the families of $\De$-friendly monomials can be deduced from one of them.

\begin{proof}
Let $\wt M^\bul = (\wt\cU_1^\bul)\ii \times \wt\cU_2^\bul \in \hMstR$. This mould is symmetral iff
$\wt\cU_2^\bul$ is symmetral. 
Let $m\in\Z^*$.
We have $\De_m\wt\cU_1^\bul = I_m^\bul \times  \wt\cU_1^\bul$, with the
mould~$I_m^\bul$ defined by~\eqref{eqdefIm}.
The Leibniz rule applied to~\eqref{eqrelcUcU} yields
$\De_m\wt\cU_2^\bul = I_m^\bul \times  \wt\cU_2^\bul + 
\wt\cU_1^\bul \times  \De_m \wt M^\bul$.
Thus $\wt\cU_2^\bul$ satisfies \eqref{eqsystDefriend} for all~$m$ iff 
$\wt\cU_1^\bul \times  \De_m \wt M^\bul=0$ for all~$m$, which is equivalent to
$\wt M^\omb\in\wt\gP_0$ since $\wt\cU_1^\bul$ admits a multiplicative inverse.
\end{proof}

%%%%%%%%%%%%%%%%%%%%%%%%%%%%%%%%%%%%%%%%%%%%%%%%%%%%%%%%%%%%%%%%%%%%%%%%%
%%%%%%%%%%%%%%%%%%%%%%%%%%%%%%%%%%%%%%%%%%%%%%%%%%%%%%%%%%%%%%%%%%%%%%%%%

\parag
Define $\bDe_\est=\ID$ and
$\bDe_\omb = \De_{\om_r}\dotsm\De_{\om_1}$ for $\omb=(\om_1,\dotsc,\om_r)\in(\Z^*)^\bul$.
We call \emph{resurgence polynomial} any $\wt\ph\in\tRsimpZ$ such that
$\bDe_{\omb}\wt\ph = 0$ for all but finitely many $\omb \in(\Z^*)^\bul$.
Resurgence polynomials form a subalgebra~$\wt\gP$ of~$\tRsimpZ$ (which contains~$\wt\gP_0$).

\begin{prop}
Let $\wt\cU^\bul$ be any family of $\De$-friendly monomials and $\wt\ph$ be any
simple resurgent function.
Then $\wt\ph$ is a resurgence polynomial iff $\wt\ph$ can be written as
\begin{equation}	\label{eqreprespol}
\wt\ph = \sum_{\omb\in\gF} \wt\cU^\omb \wt\ph_\omb, \qquad
\text{$\gF$ finite subset of $(\Z^*)^\bul$,}
\end{equation}
with $\wt\ph_\omb\in\wt\gP_0$ for every $\omb\in\gF$.
Moreover, such a representation of a resurgence polynomial is unique and the formula
$\gE = \sum S \wt\cU^\bul \bDe_\bul$
(with $S$ defined by~\eqref{eqdefinvol}, thus $S\wt\cU^\bul$ is the
multiplicative inverse of~$\wt\cU^\bul$)
defines an algebra homomorphism $\gE\colon\wt\gP \to \wt\gP_0$ such that
$$
\wt\ph_\omb = \gE \bDe_\omb\wt\ph, \qquad \omb \in (\Z^*)^\bul.
$$
\end{prop}

\begin{proof}
In view of~\eqref{eqiterDeDefr}, formula~\eqref{eqreprespol} defines a
resurgence polynomial whenever the $\wt\ph_\omb$'s are resurgence constants.

The formula $\gE = \sum S \wt\cU^\bul \bDe_\bul$ makes sense as an operator
$\wt\gP\to\tRsimpZ$ since the sum is locally finite;
an easy adaptation of the arguments of Section~\ref{secContrAltSym} shows that
$\gE$ is an algebra homomorphism because $\wt\cU^\bul$ is symmetral and
$\bDe_\bul$ can be viewed as a cosymmetral comould (the $\De_m$'s which generate
it are derivations of~$\wt\gP$).

Let us check that $\gE\Big(\wt\gP\Big) \subset \wt\gP_0$.
Let $\wt\ph\in\wt\gP$ and $m\in\Z^*$; we can write $\De_m = \sum
I_m^\bul\bDe_\bul$ with the notation~\eqref{eqdefIm}.
A computation analogous to the proof of Proposition~\ref{propmultiplimould}, but
taking into account the fact that~$\De_m$ does not commute with the
multiplication by~$(S\wt\cU)^\omb$, shows that
$$
\De_m\gE\wt\ph = \sum \big( (S\wt\cU \times I_m^\bul) + \De_m S\wt\cU^\bul \big)\bDe_\bul\wt\ph.
$$
Since $\De_m\wt\cU^\bul = I_m^\bul\times\wt\cU^\bul$ and $S$ is an
anti-homomorphism such that $S I_m^\bul = -I_m^\bul$ and $S\De_m=\De_mS$, we have
$\De_m S\wt\cU^\bul = -S\wt\cU \times I_m^\bul$, hence $\De_m\gE\wt\ph=0$.

We conclude by considering $\wt\ph\in\wt\gP$ and setting $\wt\ph_\alb =
\gE\bDe_\alb\wt\ph$ for every word $\alb  \in (\Z^*)^\bul$ (but only finitely many 
words may yield a nonzero result).
We have
$\wt\ph_\alb = \sum_{\beb} (S\wt\cU^\bul)^\beb \bDe_{\alb\conc\beb}\wt\ph$, 
thus
$\sum_{\alb} \wt\cU^\alb \wt\ph_\alb = \sum_{(\alb,\beb)} 
\wt\cU^\alb (S\wt\cU^\bul)^\beb \bDe_{\alb\conc\beb}\wt\ph$, 
and the identity $\wt\cU^\bul \times S\wt\cU^\bul = 1^\bul$ implies
$\sum_{\alb} \wt\cU^\alb \wt\ph_\alb = \wt\ph$.
\end{proof}

%%%%%%%%%%%%%%%%%%%%%%%%%%%%%%%%%%%%%%%%%%%%%%%%%%%%%%%%%%%%%%%%%%%%%%%%%
%%%%%%%%%%%%%%%%%%%%%%%%%%%%%%%%%%%%%%%%%%%%%%%%%%%%%%%%%%%%%%%%%%%%%%%%%

\section{Other applications of mould calculus}	\label{secOtherAppli}

%%%%%%%%%%%%%%%%%%%%%%%%%%%%%%%%%%%%%%%%%%%%%%%%%%%%%%%%%%%%%%%%%%%%%%%%%
%%%%%%%%%%%%%%%%%%%%%%%%%%%%%%%%%%%%%%%%%%%%%%%%%%%%%%%%%%%%%%%%%%%%%%%%%

%%%%%%%%%%%%%%%%%%%%%%%%%%%%%%%%%%%%%%%%%%%%%%%%%%%%%%%%%%%%%%%%%%%%%%%%%
%%%%%%%%%%%%%%%%%%%%%%%%%%%%%%%%%%%%%%%%%%%%%%%%%%%%%%%%%%%%%%%%%%%%%%%%%

% \parag
% %
% Cosymmetrel comould associated with a tangent-to-identity germ of $(\C,0)$ and
% resurgence of the direct and inverse iterators via mould-comould expansions?
% (brief allusion)

%%%%%%%%%%%%%%%%%%%%%%%%%%%%%%%%%%%%%%%%%%%%%%%%%%%%%%%%%%%%%%%%%%%%%%%%%
%%%%%%%%%%%%%%%%%%%%%%%%%%%%%%%%%%%%%%%%%%%%%%%%%%%%%%%%%%%%%%%%%%%%%%%%%

\parag
In this last section, we wish to indicate how mould calculus can be applied to
another classical normal form problem: the linearisation of a vector field with
non-resonant spectrum.

Let $\gA = \C[[y_1,\dotsc,y_n]]$ with $n\in\N^*$, and consider a vector field
with diagonal linear part:
$$
X = \sum_{i=1}^n a_i(y) \frac{\pa\;}{\pa y_i}, \qquad
a_i(y) = \la_i y_i + \sum_{k\in\N^n,\, |k|\ge2} a_{i,k} y^k
$$
(with standard notations for the multi-indices: 
$y^k = y_1^{k_1}\dotsm y_n^{k_n}$ and $|k| = k_1+\dotsb+k_n$
if $k = (k_1,\dotsc,k_n)$).

The first problem consists in finding a formal transformation which
conjugates~$X$ and its linear part
$$
\Xl = \sum_{i=1}^n \la_i y_i \frac{\pa\;}{\pa y_i}.
$$
This linear part is thus considered as a natural candidate to be a normal form;
it is determined by the spectrum 
$\la = (\la_1,\dotsc,\la_n)$.
In fact $\Xl = \gX_\la$ with the notation~\eqref{eqdefgXla}.

It is not always possible to find a formal conjugacy between~$X$ and~$\Xl$,
because elementary calculations let appear rational functions of the spectrum,
the denominators of which are of the form
\begin{equation}	\label{eqdefdiv}
\lan m,\la \ran = m_1 \la_1 + \dots + m_n \la_n
\end{equation}
with certain multi-indices $m\in\Z^n$.
Let us make the following {\em strong non-resonance assumption}:
\begin{equation}	\label{eqStNRass}
\lan m,\la \ran \neq 0 \quad
\text{for every $m\in\Z^n\setminus\{0\}$.}
\end{equation}
We shall now indicate how to construct a formal conjugacy via mould-comould
expansions under this assumption.

%%%%%%%%%%%%%%%%%%%%%%%%%%%%%%%%%%%%%%%%%%%%%%%%%%%%%%%%%%%%%%%%%%%%%%%%%
%%%%%%%%%%%%%%%%%%%%%%%%%%%%%%%%%%%%%%%%%%%%%%%%%%%%%%%%%%%%%%%%%%%%%%%%%

\parag
We are in the framework of Section~\ref{secGenMcM} with $\bA=\C$.
Let us use the standard monomial valuation on~$\gA$, defined by $\nu(y^k)=|k|$.
We shall manipulate operators of~$\gA$ having a valuation \wrt~$\nu$; they form
a subspace~$\gF$ of $\End_\C\gA$ which was denoted $\gF_{\gA,\bA}$
in~\eqref{eqdefgFexemp}. 

We first decompose~$X$ as a sum of homogeneous components, in the sense of
Definition~\ref{defhomog}:
$\Xl$ is homogeneous of degree~$0$ and we can write
$$
X - \Xl = \sum_{i=1}^n \sum_{k\in\Z^n} a_{i,k} y^k \frac{\pa\;}{\pa y_i},
$$
thus extending the definition of the $a_{i,k}$'s:
$$
a_{i,k}  \neq0 \ens\Rightarrow\ens k\in\N^n \ens\text{and}\ens |k|\ge2.
$$
Using the canonical basis $(e_1,\dotsc,e_n)$ of~$\Z^n$, we can write
\begin{equation}	\label{eqdefBnei}
X-\Xl = \sum_{m\in\Z^n} B_m, \qquad
B_m = \sum_{i=1}^n a_{i,m+e_i} y^m \cdot y_i\frac{\pa\;}{\pa y_i}.
\end{equation}
Observe that each $B_m$ is homogeneous of degree $m\in\Z^n$ and that
\begin{multline}
\notag
\quad B_m\neq0 \ens\Rightarrow\ens m\in\cN, \\
\label{eqdefcNalph}
\cN = \ao m\in\Z^n \mid \exists i \;\text{such that}\; m+e_i\in\N^m
\;\text{and}\; |m|\ge1 \af.\quad 
\end{multline}

We thus view~$\cN$ as an alphabet and consider
$$
\bB_\est = \ID, \qquad
\bB_{m_1,\dots,m_r} = B_{m_r}  \dotsm B_{m_1}
$$
as a comould on~$\cN$ with values in~$\gF$.
For instance, $X-\Xl = \sum I^\bul \bB_\bul$.
The inequalities
$$
\vln{\bB_{m_1,\dots,m_r}} \ge | m_1 + \dotsb + m_r |
$$
show that, for any scalar mould $M^\bul \in \hM^\bul(\cN,\C)$, the family
$(M^\mb \bB_\mb)_{\mb\in\cN^\bul}$ is formally summable in~$\gF$
(indeed, for any $\de\in\Z$, 
$\vln{M^{m_1,\dots,m_r}\bB_{m_1,\dots,m_r}} \le \de$ implies 
$r\le |m_1|+\dotsb+|m_r|\le\de$ and there are only finitely many $\eta\in\cN$
such that $|\eta|\le\de$).

%%%%%%%%%%%%%%%%%%%%%%%%%%%%%%%%%%%%%%%%%%%%%%%%%%%%%%%%%%%%%%%%%%%%%%%%%
%%%%%%%%%%%%%%%%%%%%%%%%%%%%%%%%%%%%%%%%%%%%%%%%%%%%%%%%%%%%%%%%%%%%%%%%%

\parag
According to the general strategy of mould-comould expansions, we now look for a
formal conjugacy~$\th$ bewteen $X$ and~$\Xl$ through its substitution
automorphism~$\Th$, which should satisfy $X = \Th\ii\Xl\Th$.
This conjugacy equation can be rewritten 
$$
\left[ \Xl, \Th \right] = \Th \left( X-\Xl \right).
$$
Propositions~\ref{propmultiplimould} and~\ref{propcommutXla} show that, given
any $M^\bul\in\hM^\bul(\cN,\C)$,
$\Th = \sum M^\bul \bB_\bul$ is solution as soon as 
\begin{equation}	\label{eqConjugM}
D_\ph M^\bul = I^\bul \times M^\bul,
\end{equation}
with $D_\ph M^\mb = \lan \norm{\mb},\la \ran M^\mb$ for $\mb\in\cN^\bul$.

Assumption~\eqref{eqStNRass} allows us to find a unique solution of
equation~\eqref{eqConjugM} such that $M^\est = 1$; it is inductively determined
by
$$
M^{m_1,\dotsc,m_r} = \frac{1}{\lan \norm{\mb},\la \ran} M^{m_2,\dotsc,m_r},
$$
hence
\begin{equation}	\label{eqdefsollin}
M^{\mb} = \frac{1}{\lan {m_1+\dotsb+m_r},\la \ran}
\frac{1}{\lan {m_2+\dotsb+m_r},\la \ran}
\dotsm \frac{1}{\lan {m_r},\la \ran}
\end{equation}
The symmetrality of this solution can be obtained by mimicking the proof of
Proposition~\ref{propcVsym}.

Since $\bB_\bul$ is cosymmetral, we thus have an automorphism $\Th = \sum
M^\bul  \bB_\bul$; 
since $\Th$ is continuous for the Krull topology, $\th = (\th_1,\dots,\th_n)$
with $\th_i=\Th y_i$ yields a formal tangent-to-identity transformation which
conjugates~$X$ and~$\Xl$.

%%%%%%%%%%%%%%%%%%%%%%%%%%%%%%%%%%%%%%%%%%%%%%%%%%%%%%%%%%%%%%%%%%%%%%%%%
%%%%%%%%%%%%%%%%%%%%%%%%%%%%%%%%%%%%%%%%%%%%%%%%%%%%%%%%%%%%%%%%%%%%%%%%%

\parag
As was alluded to at the end of Section~\ref{secContrAltSym}, the formalism of
moulds can be equally applied to the normalisation of discrete dynamical
systems. 
A problem parallel to the previous one is the linearisation of a formal
transformation with multiplicatively non-resonant spectrum.

Suppose indeed that $f = (f_1,\dotsc,f_n)$ is a $n$-tuple of formal series
of~$\gA$ without constant terms, with diagonal linear part 
$\fl \colon (y_1,\dotsc,y_n) \mapsto (\ell_1 y_1, \dotsc, \ell_n y_n)$.
Conjugating~$f$ and~$\fl$ is equivalent to finding a continuous
automorphism~$\Th$ which conjugates the corresponding susbtitution automorphisms:
$F = \Th\ii \Fl \Th$.

This is possible under the following {\em strong multiplicative non-resonance
assumption} on the spectrum $\ell = (\ell_1,\dotsc,\ell_n)$:
\begin{equation}	\label{eqStNRassMult}
\ell^m -1 \neq 0 \quad
\text{for every $m\in\Z^n\setminus\{0\}$.}
\end{equation}
An explicit solution is obtained by expanding $F (\Fl)\ii$ in homogeneous
components
$$
F = \Big( \ID + \sum_{m\in\cN} B_m \Big) \Fl,
$$
where the homogeneous operators~$B_m$ are no longer derivations; instead, they
satisfy the modified Leibniz rule~\eqref{eqmodifLeib} and generate a {\em
cosymmetrel} comould~$\bB_\bul$.
Correspondingly, the scalar mould
\begin{equation}	\label{eqdefsollinMult}
M^\est = 1, \qquad
M^{\mb} = \frac{1}{(\ell^{m_1+\dotsb+m_r}-1)
(\ell^{m_2+\dotsb+m_r}-1)
\dotsm (\ell^{m_r}-1)}
\end{equation}
is {\em symmetrel} and $\Th = \sum M^\bul\bB_\bul$ is the desired automorphism
(see \cite{Eca93}),
whence a formal tangent-to-identity transformation~$\th$ which conjugates $f$ and~$\fl$.

%%%%%%%%%%%%%%%%%%%%%%%%%%%%%%%%%%%%%%%%%%%%%%%%%%%%%%%%%%%%%%%%%%%%%%%%%
%%%%%%%%%%%%%%%%%%%%%%%%%%%%%%%%%%%%%%%%%%%%%%%%%%%%%%%%%%%%%%%%%%%%%%%%%

\parag
In both previous problems, it is a classical result that a formal linearising
transformation~$\th$ exists under a weaker non-resonance assumption:
namely, it is sufficient that~\eqref{eqStNRass} or~\eqref{eqStNRassMult} 
hold with~$\Z^n\setminus\{0\}$ replaced by~$\cN\setminus\{0\}$.
Unfortunately, this is not clear on the mould-comould expansion, since under
this weaker assumption the formula~\eqref{eqdefsollin}
or~\eqref{eqdefsollinMult} may involve a zero divisor, thus the mould~$M^\bul$
is not well-defined.

J.~\'Ecalle has invented a technique called {\em arborification} which solves
this problem and which goes far beyond: arborification also allows to recover
the Bruno-R\"ussmann theorem, according to which the formal linearisation~$\th$
is convergent whenever the vector field~$X$ or the transformation~$f$ is
convergent and the spectrum~$\la$ or~$\ell$ satisfies the so-called {\em Bruno
condition} (a Diophantine condition which states that the divisors $\lan m,\la
\ran$ or $\ell^m - 1$ do not approach zero ``abnormally well'').

The point is that, even when $X$ or~$f$ is convergent and the spectrum is
Diophantine, it is hard to check that~$\th_i(y)$ is convergent because it is
represented as the sum of a formally summable family $(M^\mb \bB_\mb
y_i)_{\mb\in\cN^\bul}$ in~$\C[[y_1,\dotsc,y_n]]$, but
the family $\big(|M^\mb \bB_\mb y_i|\big)_{\mb\in\cN^\bul}$ may fail to be
summable in~$\C$ for any $y\in\C^n\setminus\{0\}$.
However, arborification provides a systematic way of reorganizing the terms of
the sum: $\th_i(y)$ then appears as the sum of a summable family indexed by
``arborescent sequences'' rather than words.
The reader is referred to \cite{EcaNonAb}, \cite{dulac}, \cite{EV}, and also to the
recent article \cite{Eca05}.

%%%%%%%%%%%%%%%%%%%%%%%%%%%%%%%%%%%%%%%%%%%%%%%%%%%%%%%%%%%%%%%%%%%%%%%%%
%%%%%%%%%%%%%%%%%%%%%%%%%%%%%%%%%%%%%%%%%%%%%%%%%%%%%%%%%%%%%%%%%%%%%%%%%

\parag
There is another context, totally different, in which J.~\'Ecalle has used mould
calculus with great efficiency.
The multizeta values 
$$
\ze(s_1,s_2,\dotsc,s_r) = \sum_{n_1>n_2>\dotsb>n_r>0}
\frac{1}{n_1^{s_1}n_2^{s_2}\dotsm n_r^{s_r}}
$$
naturally present themselves as a scalar mould on~$\N^*$; in fact, 
\begin{multline*}
\ze(s_1,\dotsc,s_r) = \ZE^{\left(\begin{smallmatrix}
0,&\dotsc\,,&0\\s_1,&\dotsc\,,&s_r\end{smallmatrix}\right)},\\
\text{with}\quad
\ZE^{\left(\begin{smallmatrix}
\eps_1,&\dotsc\,,&\eps_r\\s_1,&\dotsc\,,&s_r\end{smallmatrix}\right)}
= \sum_{n_1>\dotsb>n_r>0}
\frac{\ee^{2\pi\I(n_1\eps_1+\dotsb+n_r\eps_r)}}{n_1^{s_1}\dotsm n_r^{s_r}}
\end{multline*}
for $s_1,\dotsc,s_r\in\N^*$, $\eps_1,\dotsc,\eps_r\in\Q/\Z$
(with a suitable convention to handle possible divergences).
The mould~$\ZE^\bul$ is the central object; 
it turns out that it is {\em symmetrel}.
It is called a {\em bimould} because the letters of the alphabet are naturally given
as members of a product space, here $\N^*\times(\Q/\Z)$; this makes it possible
to define new operations and structures.
This is the starting point of a whole theory, aimed at describing the algebraic
structures underlying the relations between multizeta values.
See \cite{EcaBil} or \cite{EcaARI}.

%%%%%%%%%%%%%%%%%%%%%%%%%%%%%%%%%%%%%%%%%%%%%%%%%%%%%%%%%%%%%%%%%%%%%%%%%
%%%%%%%%%%%%%%%%%%%%%%%%%%%%%%%%%%%%%%%%%%%%%%%%%%%%%%%%%%%%%%%%%%%%%%%%%

\vspace{.7cm}

\subsubsection*{Acknowledgements}
It is a pleasure to thank F.~Fauvet, F.~Menous, F.~Patras and B.~Teissier for
their help.
I am also indebted to F.~Fauvet and J.-P.~Ramis, for having organized the conference
``Renormalization and Galois theories'' and given me a great opportunity of
lecturing on Resurgence theory.
This work was done in the framework of the project {\em Ph\'enom\`ene de Stokes,
renormalisation, th\'eories de Galois} from the {\em agence nationale pour la recherche}.

\vspace{.3cm}

%%%%%%%%%%%%%%%%%%%%%%%%%%%%%%%%%%%%%%%%%%%%%%%%%%%%%%%%%%%%%%%%%%%%%%%%%
%%%%%%%%%%%%%%%%%%%%%%%%%%%%%%%%%%%%%%%%%%%%%%%%%%%%%%%%%%%%%%%%%%%%%%%%%

\end{document}